\documentclass[preprint,12pt]{elsarticle}
\usepackage[T1]{fontenc}
\usepackage{amsmath}
\usepackage{amsthm}
\usepackage{newtxtext}
\usepackage{newtxmath}
\usepackage{mathtools}
\usepackage{bm}
\usepackage{graphicx}
\usepackage{subcaption}
\usepackage{booktabs}
\usepackage{multirow}
\usepackage{algorithm}
\usepackage{algpseudocode}
\usepackage{xcolor}
\usepackage{xr-hyper}
\usepackage[colorlinks=true,linkcolor=blue!50!black,citecolor=blue!50!black,urlcolor=blue!50!black]{hyperref}
\usepackage{cleveref}
\usepackage{lineno}
\usepackage[normalem]{ulem}

\newcommand{\T}{\mathbb{T}}
\newcommand{\R}{\mathbb{R}}
\newcommand{\Z}{\mathbb{Z}}
\newcommand{\C}{\mathbb{C}}
\newcommand{\Ph}{\mathcal{P}_h}
\newcommand{\Ih}{\mathcal{I}_h}
\newcommand{\Dth}{\mathcal{D}_\theta}
\newcommand{\Ra}{\mathcal{R}_\alpha}
\newcommand{\Ev}{\mathcal{E}_{\Delta t}}
\newcommand{\Sh}{\mathcal{S}_h}
\newcommand{\dfh}{\delta_h}
\newcommand{\Dm}{\nabla^{-}}
\newcommand{\xf}[1]{x_{#1+1/2}}
\newcommand{\Li}{\operatorname{Li}}
\newcommand{\sgn}{\operatorname{sgn}}
\newcommand{\sinc}{\operatorname{sinc}}
\newcommand{\Real}{\operatorname{Re}}
\newcommand{\Imag}{\operatorname{Im}}
\newcommand{\tr}{\operatorname{tr}}
\newcommand{\Cov}{\operatorname{Cov}}
\newcommand{\Var}{\operatorname{Var}}
\newcommand{\EE}{\mathbb{E}}
\newcommand{\psia}{\psi_{a}}
\newcommand{\dd}{\,\mathrm{d}}
\newcommand{\errA}{\mathrm{err}_{W1}}
\newcommand{\errB}{\mathrm{err}_{W2}}

\theoremstyle{plain}
\newtheorem{theorem}{Theorem}[section]
\newtheorem{lemma}[theorem]{Lemma}
\newtheorem{proposition}[theorem]{Proposition}
\newtheorem{corollary}[theorem]{Corollary}
\newtheorem{assumption}[theorem]{Assumption}
\theoremstyle{definition}
\newtheorem{definition}[theorem]{Definition}
\theoremstyle{remark}
\newtheorem{remark}[theorem]{Remark}

\journal{Journal of Computational Physics}

\begin{document}

\begin{frontmatter}

\title{A hybrid finite-volume scheme with learned rational reconstructions and exact nonlocal fluxes for the space-fractional Burgers equation}

\author[osu]{Zeyuan Song\corref{cor1}}
\ead{taekwon.song@okstate.edu}
\author[osu]{Zheyu Jiang}
\ead{zjiang@okstate.edu}
\cortext[cor1]{Corresponding author.}
\affiliation[osu]{organization={School of Chemical Engineering, Oklahoma State University},
            city={Stillwater},
            postcode={74078},
            state={OK},
            country={USA}}

\begin{abstract}
On coarse grids, fronts of fractional conservation laws can be narrower than one cell, while the flux through each face depends on the whole domain.
For the forced space-fractional Burgers equation, we write the finite-volume defect exactly, with the Riesz flux acting as a nonlocal interface flux.
Our closure, NRC, is decoder-only, since the encoder is the exact cell-average projection and a small network only maps cell averages to a sub-cell reconstruction.
The decoder adds rational atoms with learned complex poles to a trigonometric histopolant, preserves every cell average, and has closed-form advective and Riesz fluxes.
We prove that the scheme conserves mass for every value of the weights and that the learned sub-grid flux has a dipole tail.
On 64 cells, NRC reaches a median test error of 0.0308\% over three seeds, against 0.431\% for WENO5-Z and 0.153\% for a black-box flux correction.
Ablations show that the exact global operator, the projection and the sub-grid Riesz flux carry most of this gain, and most sub-cell fronts carry a deep learned pole within a quarter cell.
{NRC keeps its advantage over WENO5-Z under shifts of the order and the forcing and, when trained on several diffusivities, under the tested diffusivity shifts by factors from 0.6 to 2.5 in the median over seeds.
From 32 to 128 cells, NRC stays an order of magnitude more accurate than WENO5-Z on the same grid.
The gain lies in accuracy per degree of freedom and in an interpretable closure.}
\end{abstract}

\begin{keyword}
Space-fractional Burgers equation \sep Finite-volume method \sep Learned closure \sep Rational approximation \sep Riesz flux \sep Discretization error
\end{keyword}

\end{frontmatter}

\section{Introduction}\label{sec:intro}

Conservation laws with nonlocal, fractional dissipation describe anomalous transport, for instance the heavy-tailed dispersion of solutes in porous media \cite{metzler2000random,benson2000application}.
The fractional Burgers equation $u_t+uu_x=-\nu(-\Delta)^{\alpha/2}u+f$ with $1<\alpha<2$ is the canonical nonlinear model of this class, since it couples nonlinear steepening with nonlocal dissipation \cite{biler1998fractal,droniou2003global}.
For these orders the unforced periodic problem has smooth solutions without shocks \cite{alibaud2007occurrence,kiselev2008blowup}, yet at small diffusivity its fronts can be narrower than one cell of a practical coarse grid.
Refinement is costly, because the fractional operator couples every cell to all others and its stiffness grows in the order of $h^{-\alpha}$ as the cell width $h$ decreases.

For the nonlocal operator, finite-difference quadratures and finite-element methods are well developed \cite{huang2014numerical,bonito2018numerical}, and convergent finite-volume schemes exist for fractal conservation laws \cite{droniou2010numerical}.
For the advective part, high-resolution finite-volume schemes such as MUSCL and WENO reconstruct sub-cell states from cell averages with limited or weighted stencils \cite{vanleer1979towards,jiang1996efficient,borges2008improved,shu2009high,leveque2002finite}.
Exponential integrators treat the stiff linear part exactly and leave only the nonlinear flux to explicit stages \cite{cox2002exponential,hochbruck2010exponential}.
These classical building blocks have well-understood properties, but their reconstructions were designed for local fluxes and are piecewise polynomials that jump at the cell faces.

A growing body of work replaces selected parts of such schemes by neural network components.
Bar-Sinai et al.~\cite{barsinai2019learning} learned solution-dependent stencil coefficients inside a conservative finite-volume update, an idea carried over to two-dimensional flows and scalar transport \cite{kochkov2021machine,zhuang2021learned}.
Neural closure models add a trained term to a coarse discretization \cite{sirignano2020dpm}, and training them through a differentiable solver on unrolled trajectories has been found to matter for accuracy and stability \cite{um2020solver,list2022learned,melchers2023comparison}.
Closures with memory \cite{gupta2023generalized}, learned simulators \cite{stachenfeld2022learned,brandstetter2022message}, troubled-cell indicators \cite{ray2018artificial} and learned propagators \cite{maddu2023stencil} broaden this picture, as reviewed for turbulence by Duraisamy et al.~\cite{duraisamy2019turbulence}.

In parallel, neural operators such as the Fourier neural operator and DeepONet learn maps between function spaces and serve as surrogates for parametric partial differential equations \cite{li2021fourier,lu2021learning,kovachki2023neural}.
Such surrogates, and black-box corrections of coarse schemes in general, inherit neither conservation nor the exact treatment of known linear operators unless these are built in, and their failures are hard to predict in numerical terms.
Weak baselines and selective reporting can make learned solvers appear more favorable than they are \cite{mcgreivy2024weak}, so we compare against strong classical schemes on the same grid.

Learning methods for fractional and nonlocal equations have mostly targeted the solution, the fractional order and coefficients, or nonlocal network layers \cite{pang2019fpinns,gulian2019machine}, and fractional operators have also served as closures for local equations \cite{akhavansafaei2021data}.
Meanwhile, learned closures for conservation laws with fractional fluxes remain poorly understood, and two features of the fractional flux set them apart.
First, the flux through a cell face is a singular integral over the entire domain, and its part beyond any fixed number of cells tends to the whole flux as the grid is refined.
A closure built on a local surrogate of the operator, or on a network with a fixed receptive field, therefore carries a model error of order one rather than a truncation error (Proposition~\ref{prop:nearfar}).
Second, a jump in the reconstructed state produces an interface flux that diverges like $|x-x_0|^{1-\alpha}$, so the piecewise polynomial reconstructions of MUSCL and WENO have unbounded fractional fluxes (Lemma~\ref{lem:jump}).
In practice the fractional term is therefore evaluated separately, for instance spectrally from a smooth global reconstruction, and the sub-cell structure of the fronts never enters the nonlocal flux.
As a result, it remains unclear what a learned closure should represent here, which parts of the discretization error can be computed exactly, and which must be learned.

Two classical theories suggest a principled approach to these questions.
Optimal prediction identifies the mean-square optimal closure with the conditional expectation of the unresolved dynamics given the resolved variables \cite{chorin2000optimal}, as in optimal large-eddy simulation \cite{langford1999optimal}.
For cell averages the closure must also absorb the error of the discrete operator, as discrete-level closures make explicit \cite{beck2019deep,agdestein2025discretize}, and approximate deconvolution reconstructs an unfiltered field before applying the exact flux \cite{stolz1999approximate}.
Rational approximation theory explains why poles are efficient degrees of freedom near singularities, since rational functions approximate $|x|$ and $x^\beta$ at root-exponential rates \cite{newman1964rational,stahl2003best}.
Poles clustered near a singularity achieve such rates constructively \cite{stenger1986explicit,trefethen2021exponential,gopal2019solving,herremans2023resolution}, and the AAA algorithm locates them adaptively from samples \cite{nakatsukasa2018aaa}.
Rational spectral methods have also been developed for the fractional Laplacian \cite{tang2020rational}.
The steepness of a smooth front is governed by complex singularities near the real axis, whose distance can be estimated from the decay of Fourier coefficients \cite{sulem1983tracing}.
A few poles placed near such singularities can therefore represent a front narrower than one cell, whereas a coarse trigonometric reconstruction cannot (Proposition~\ref{prop:gap}).

These developments motivate interpreting a learned closure not merely as a generic corrector, but \emph{as a data-adaptive reconstruction operator} that replaces the limited polynomial reconstruction of a high-resolution finite-volume scheme.
In this work, we develop and analyze such a closure, nonlocal reconstruction closure or NRC, for the forced space-fractional Burgers equation on a periodic domain.
With the encoder fixed to the exact cell-average projection $\Ph$, the closure is decoder-only and learns only a map $\Dth$ from cell averages to a sub-cell reconstruction.
The decoder adds to a trigonometric histopolant a sum of rational atoms, whose complex poles and amplitudes a small convolutional network predicts cell by cell.
{This network consists of four circular convolutions with a receptive field of 17 cells and a pointwise head that outputs two poles per cell, with about $2.5\times10^4$ parameters in total, and \Cref{sec:decoder} describes the structure of the decoder in detail.}
A projection removes the cell averages of these atoms, so that they only redistribute mass inside each cell.
{The network therefore learns only the shape of the solution inside each cell, and never the equation itself.
The advective and Riesz fluxes are evaluated exactly on the decoded state, and the only learned source of error is the reconstruction.
Because this state is continuous at the faces, every face value is single valued, so the scheme needs no Riemann solver and the Riesz flux stays bounded (Lemma~\ref{lem:jump}).} 
The stiff global part of the operator acts only on the histopolant and is treated exactly by an exponential integrator, so the network never supplies the far field.
Unlike learned stencil coefficients \cite{barsinai2019learning} or manifold decoders inside projection-based integrators \cite{lee2020model}, our decoder preserves cell averages exactly and feeds exact nonlocal fluxes.
Unlike our earlier finite-volume hybrids with learned encoders \cite{song2026mp,song2025physics}, the encoder here is fixed, so the coarse variables keep their meaning as cell averages.
{Our goal is a higher \emph{accuracy per degree of freedom}.
The degrees of freedom of a finite-volume scheme are its $N$ cell averages, and one scheme has a higher accuracy per degree of freedom than another if it reaches a lower error with the same number of cells or the same error with fewer cells.
We quantify it by the error at equal $N$ and by the factor by which a classical scheme must refine its grid to reach the error of NRC (\Cref{sec:grid}).
As for other learned discretizations, the decoder is trained for the grid on which it is used.}

On a 64-cell grid, the pooled median relative error of NRC on the in-distribution test split is about 14 times lower than that of a WENO5-Z scheme with the same exact fractional operator (Section~\ref{sec:accuracy}).
This comparison covers three training seeds, and the error of NRC does not grow when rollouts continue past the end of the training trajectories.
{The error of NRC is also 2.9 to 5 times lower than that of a black-box face-flux correction with the same backbone and flux target, the lower factor for a single seed trained as long as NRC, and NRC on 64 cells is more accurate than WENO5-Z on 128 cells (\Cref{sec:grid}).}

The main contributions of this work are summarized as follows:
\begin{itemize}
\item {\emph{An exact target for learned closures.}} We write the classical finite-volume balance as an exact defect identity with the Riesz flux as a time-averaged nonlocal interface flux (Section~\ref{sec:commutator}).
We specialize the optimal-prediction view to this defect and establish the two features of the fractional flux stated above (Proposition~\ref{prop:nearfar} and Lemma~\ref{lem:jump}).
\item {\emph{A decoder-only closure with structural guarantees.}} We propose NRC, whose cell-average preserving decoder makes the proposed reconstruct-evolve-average scheme conserve mass for every value of the network weights (Section~\ref{sec:decoder}){, while an exponential integrator treats the stiff global operator exactly}.
\item {\emph{Rational atoms with closed-form fluxes.}} We represent sub-cell fronts by rational atoms whose values, cell averages and Riesz fluxes have closed forms, and which no coarse trigonometric reconstruction resolves to better than $e^{-\pi\delta/h}$ for pole depth $\delta$ (Section~\ref{sec:atoms}).
Because the learned correction has zero cell averages, its Riesz flux decays like that of a dipole, which keeps the learned part local although the operator is not.
\item {\emph{An accurate and interpretable learned reconstruction.}} We verify these predictions on frozen cell averages and in long rollouts, against classical high-resolution schemes and learned flux corrections of comparable size trained on the same target (Sections~\ref{sec:apriori} and~\ref{sec:coupled}).
In ablations with three seeds and a shortened curriculum, a local surrogate of the global operator raises the median test error about 228 times (Section~\ref{sec:ablation}).
Removing the projection or the sub-grid Riesz flux raises it $9.9$ and $12.5$ times, respectively.
{For every seed and order, a deep learned pole lies within a quarter cell of 78 to 89\% of the sub-cell fronts (Section~\ref{sec:mechanism}).
NRC keeps its advantage over WENO5-Z under shifts of the fractional order and of the forcing and, when trained on several diffusivities, under the tested shifts of the diffusivity in the median over three seeds (Section~\ref{sec:ood}), as well as at equal resolution on 32 to 128 cells (Section~\ref{sec:grid}).}
\end{itemize}

\section{Theory and methods}\label{sec:theory}

\subsection{Problem formulation}\label{sec:problem}

We consider the space-fractional Burgers equation on the periodic domain $\T=\R/2\pi\Z$, written as a conservation law whose flux contains a nonlocal part,
\begin{equation}\label{eq:pde}
\partial_tu+\partial_xF(u)=f,\qquad F(u)=\tfrac12u^2-\nu\Ra u,\qquad x\in\T,\ t>0.
\end{equation}
Here $\alpha\in(1,2)$ is the fractional order, $\nu>0$ is the fractional diffusivity, and $f$ is a prescribed forcing with zero spatial mean.
The Riesz flux operator $\Ra$ is the Fourier multiplier with symbol $i\,\sgn(k)|k|^{\alpha-1}$, so that $\partial_x\Ra=-(-\Delta)^{\alpha/2}$ and
\begin{equation}\label{eq:pde-laplacian}
u_t+uu_x=-\nu(-\Delta)^{\alpha/2}u+f.
\end{equation}
The usual definitions of the fractional Laplacian differ on bounded domains~\cite{lischke2020what}, but on the torus the Fourier definition is unambiguous.
We study the orders $\alpha=1.2$, $1.5$ and $1.8$ with $\nu=0.3$, $0.12$ and $0.05$, respectively, and use other values only in the tests of \Cref{sec:ood}.

Following Bar-Sinai et al.~\cite{barsinai2019learning}, the forcing is a sum of twenty traveling sinusoids,
\begin{equation}\label{eq:forcing}
f(x,t)=\sum_{i=1}^{20}A_i\sin(\omega_it+l_ix+\vartheta_i).
\end{equation}
For every trajectory, $A_i$, $\omega_i$, $\vartheta_i$ and $l_i$ are sampled independently and uniformly from $[-0.5,0.5]$, $[-0.4,0.4]$, $[0,2\pi)$ and $\{3,4,5,6\}$, respectively.
Since every $l_i$ is a nonzero integer, $f$ has zero mean at all times, and its cell averages and their time integrals are known in closed form.
The initial data have Fourier coefficients $\hat u_0(k)=(\xi_k+i\xi_k')k^{-3/2}$ for $1\le k\le16$, with independent standard normal $\xi_k$ and $\xi_k'$, $\hat u_0(-k)=\overline{\hat u_0(k)}$ and zero otherwise.
Each sample is rescaled to the root-mean-square value $0.5$. 
{Since $\hat u_0(0)=0$ and $f$ has zero mean, the mass $\int_\T u\dd x$ is zero initially and conserved, so it vanishes for all times.} 
Without forcing, these solutions decay at least like $e^{-\nu t}$ in $L^2$, whereas the forcing sustains fronts that can be narrower than one coarse cell.

\begin{lemma}[Riesz flux operator]\label{lem:riesz}
Let $1<\alpha<2$, and define the kernel constants of the fractional Laplacian and of the Riesz flux operator by
\begin{equation}\label{eq:kappa}
\begin{gathered}
C_\alpha=\frac{\Gamma(1+\alpha)\sin(\pi\alpha/2)}{\pi},\\
\kappa_\alpha=\frac{C_\alpha}{\alpha}=\frac{1}{2\Gamma(1-\alpha)\cos(\pi\alpha/2)}=\frac{2^{\alpha-1}\Gamma(\frac{1+\alpha}{2})}{\sqrt{\pi}\,\Gamma(1-\frac{\alpha}{2})}>0.
\end{gathered}
\end{equation}
\begin{enumerate}
\item[(i)] If $u\in C^1(\T)$ and $\sum_k|k|^{\alpha-1}|\hat u_k|<\infty$, then for every $x\in\T$ the Riesz flux operator has the two kernel representations
\begin{equation}\label{eq:riesz-kernel}
\begin{aligned}
\Ra u(x)&=\kappa_\alpha\int_0^\infty\frac{u(x+z)-u(x-z)}{z^\alpha}\dd z\\
&=\kappa_\alpha\int_\R\sgn(y-x)\,\frac{u(y)-u(x)}{|y-x|^\alpha}\dd y,
\end{aligned}
\end{equation}
where $u$ is extended periodically and both integrals converge absolutely.
\item[(ii)] If $u\in C^{0,\eta}(\T)$ with $\eta>\alpha-1$, the integral over $z>0$ in Eq.~\eqref{eq:riesz-kernel} converges absolutely for every $x$ and defines a continuous function with Fourier coefficients $i\,\sgn(k)|k|^{\alpha-1}\hat u_k$. We use this integral as the definition of $\Ra u$ for such $u$.
\item[(iii)] The operator satisfies $\partial_x\Ra=-(-\Delta)^{\alpha/2}$, the conservation law of Eq.~\eqref{eq:pde} is therefore dissipative, and for every $u$ with $\sum_k|k|^\alpha|\hat u_k|<\infty$, in particular for $u\in H^s(\T)$ with $s>\alpha+\frac12$, the hypersingular representation
\[
-\nu(-\Delta)^{\alpha/2}u(x)=\nu C_\alpha\,\mathrm{P.V.}\!\int_\R\frac{u(y)-u(x)}{|x-y|^{1+\alpha}}\dd y
\]
holds with the same constant $C_\alpha${, where $\mathrm{P.V.}$ denotes the Cauchy principal value, $\mathrm{P.V.}\!\int_\R g(y)\dd y=\lim_{\varepsilon\to0^+}\int_{|y-x|>\varepsilon}g(y)\dd y$, which is needed because for $\alpha>1$ the integrand is in general not absolutely integrable near $y=x$}.
\item[(iv)] The operator maps $\cos$ to $-\sin$, and at the endpoint $\alpha=2$ the Riesz flux operator reduces to $\mathcal R_2=\partial_x$.
\end{enumerate}
\end{lemma}

The proof is given in \ref{app:proofs:riesz}, the hypersingular form in~(iii) is that of Kwa\'snicki~\cite{kwasnicki2017ten}, and the kernel form in~(i) is the periodic Riesz fractional gradient of Shieh and Spector~\cite{shieh2015new}.
The flux $-\nu\Ra u$ plays the role of the viscous flux $-\nu u_x$, but its kernel decays only like $z^{-\alpha}$, so the flux through a face depends on all of $\T$.

\begin{lemma}[A jump has an unbounded Riesz flux]\label{lem:jump}
Let $1<\alpha<2$, $x_0\in\R$, $c\in\R$ and $\varepsilon>0$.
Let $u$ be bounded on $\R$ with $u=c\,H_{x_0}+r$ on $(x_0-\varepsilon,x_0+\varepsilon)$, where $H_{x_0}=\mathbf{1}_{[x_0,\infty)}$ and $r$ is H\"older continuous of order $\eta>\alpha-1$ there with seminorm $[r]_\eta$.
We let $\Ra u$ be given by the integral over $z>0$ in Eq.~\eqref{eq:riesz-kernel}, which converges for every $x$ with $0<|x-x_0|<\varepsilon$.
Then the Riesz flux near the jump satisfies, for every $x$ with $0<|x-x_0|<\varepsilon/2$, the expansion
\begin{gather}
\Ra u(x)=\frac{c\,\kappa_\alpha}{\alpha-1}\,|x-x_0|^{1-\alpha}+E(x),\label{eq:jump}\\
\nonumber|E(x)|\le\kappa_\alpha\Big(\frac{2^\eta[r]_\eta(\varepsilon/2)^{1+\eta-\alpha}}{1+\eta-\alpha}+\frac{(2\|u\|_\infty+|c|)(\varepsilon/2)^{1-\alpha}}{\alpha-1}\Big).
\end{gather}
Two consequences for reconstructions on a coarse grid follow from this expansion.
\begin{enumerate}
\item[(i)] A periodic piecewise-constant reconstruction $\sum_jv_j\mathbf{1}_{C_j}$ has an unbounded Riesz flux near every face where $v_{j+1}\neq v_j$.
The same holds for any reconstruction with a jump at a face, such as the piecewise polynomials of MUSCL or WENO schemes, so pointwise interface fluxes require reconstructions that are continuous at the faces.
\item[(ii)] H\"older continuity of order $\eta>\alpha-1$ near the faces is sufficient for finite fluxes but not necessary.
Let $u$ be bounded and equal to the even cusp $|x-x_0|^\eta$ or to the odd cusp $\sgn(x-x_0)|x-x_0|^\eta$ on $(x_0-\varepsilon,x_0+\varepsilon)$, with $\eta>0$.
The even cusp has a finite flux at $x_0$ for every $\eta>0$, which vanishes when $u$ is even about $x_0$, while the odd cusp has a finite flux at $x_0$ if and only if $\eta>\alpha-1$.
\end{enumerate}
\end{lemma}

Lemma~\ref{lem:jump}, proved in \ref{app:proofs:jump}, thus excludes the reconstructions of high-resolution finite-volume schemes~\cite{harten1983high,vanleer1979towards,jiang1996efficient,borges2008improved}, which jump at the faces by design.
No Riemann solver can then select a finite face value between two infinite one-sided fluxes.
Our reconstructions are instead analytic on $\T$, a trigonometric histopolant enriched by rational functions with poles off the real axis (\Cref{sec:decoder} and~\ref{sec:atoms}).

The coarse grid has $N$ cells $C_j=[jh,(j+1)h)$ of width $h=2\pi/N$, with right faces $\xf{j}=(j+1)h$, centers $x_j=(j+\frac12)h$ and indices modulo $N$.
The coarse unknowns are the cell averages $v=\Ph u\in\R^N$, where the projection $\Ph$ acts as
\begin{equation}\label{eq:Ph}
(\Ph u)_j=\frac1h\int_{C_j}u\dd x,\qquad j=0,\dots,N-1.
\end{equation}
We use the discrete norm $\|v\|_h^2=h\sum_j|v_j|^2$, the discrete mass $m_h(v)=h\sum_jv_j$ and the backward difference $(\Dm g)_j=g_{j+1/2}-g_{j-1/2}$ of face quantities.
The coarse time step is $\Delta t$ with $t_n=n\Delta t$, and the experiments use $N=64$ and $\Delta t=0.02$ (\Cref{sec:setup}).

For the unforced problem, global smooth solutions are classical~\cite{droniou2003global,biler1998fractal}, shocks do not form~\cite{alibaud2007occurrence}, and solutions on the circle become smooth for $t>0$~\cite{kiselev2008blowup}. To our {knowledge}, no such result with forcing exists, so we make the following regularity assumption below.

\begin{assumption}[Regularity]\label{ass:regularity}
On each time window $[t_n,t_{n+1}]$, the solution of Eq.~\eqref{eq:pde} is unique and satisfies $u\in C^1([t_n,t_{n+1}],H^s(\T))$ for some $s>\alpha+\frac12$.
\end{assumption}

Under Assumption~\ref{ass:regularity}, $u$, $u_x$, $\Ra u$ and $\partial_x\Ra u$ are continuous in $(x,t)$, and Eq.~\eqref{eq:pde} holds pointwise.
We denote by $\Ev$ the exact solution map from $u(t_n)$ to $u(t_{n+1})$ and suppress its dependence on $t_n$ and the forcing.
A coarse scheme is a map $\Sh:\R^N\to\R^N$ that advances the cell averages by one step.
We call it conservative if it has the finite-volume form
\begin{equation}\label{eq:fvscheme}
\Sh(v)_j=v_j-\frac{\Delta t}{h}\Big(\hat F_{j+1/2}(v)-\hat F_{j-1/2}(v)\Big)+\int_{t_n}^{t_{n+1}}(\Ph f)_j\dd t,
\end{equation}
with numerical fluxes $\hat F_{j+1/2}:\R^N\to\R$ that may be nonlocal, aggregated over the stages of a Runge-Kutta method, or defined implicitly.

The one-step defect, or discretization commutator, measures how far the coarse scheme is from commuting with the projection,
\begin{equation}\label{eq:defect-def}
\dfh=\Ph\Ev-\Sh\Ph.
\end{equation}
Panel~(a) of \Cref{fig:schematic} shows the corresponding diagram, which commutes on a set $M$ of states exactly when $\dfh$ vanishes on $M$.
Since this fails here, we call the diagram closable on $M$ if some map $\widetilde{\mathcal S}_h:\R^N\to\R^N$ satisfies $\widetilde{\mathcal S}_h(\Ph u)=\Ph\Ev u$ for all $u\in M$, and $\mathcal M^\ast=\widetilde{\mathcal S}_h-\Sh$ is then the ideal additive closure of $\Sh$.

\begin{figure}[!tbp]
\centering
\includegraphics[width=\textwidth]{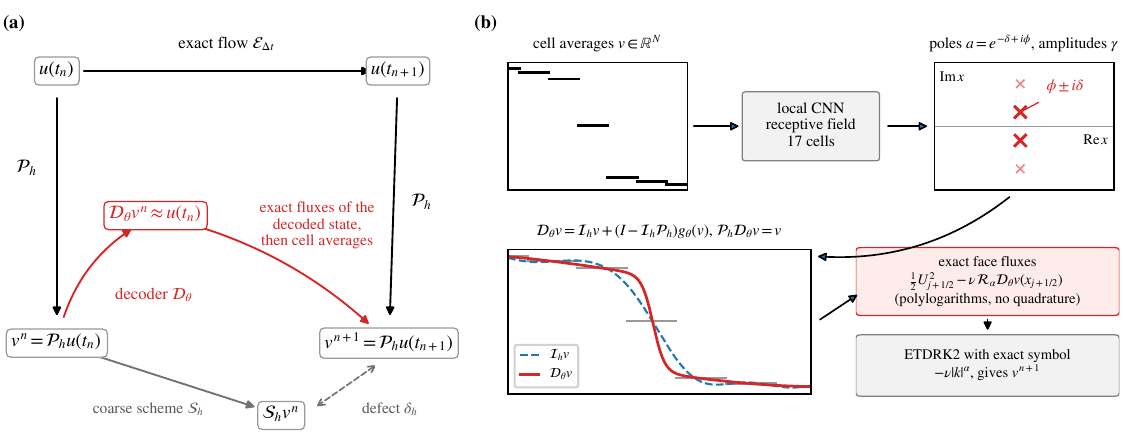}
\caption{The discretization commutator and the decoder-based closure, with panel~(b) computed on a synthetic front with 16 cells.
\textbf{(a)} The exact flow $\Ev$ acts on functions in the top row and the coarse scheme $\Sh$ on cell averages in the bottom row, and $\dfh$ is their mismatch.
The red path decodes the cell averages, evaluates the exact fluxes of the decoded state and averages them over the cells.
\textbf{(b)} A local CNN maps the cell averages to poles and amplitudes, and the decoded state $\Dth v$ (red) follows the front while $\Ih v$ (blue dashed) oscillates.
Exact face fluxes of $\Dth v$ and ETDRK2 \cite{cox2002exponential} with the exact symbol $-\nu|k|^\alpha$ then give $v^{n+1}$.}
\label{fig:schematic}
\end{figure}

\subsection{The discretization commutator and what a closure must learn}\label{sec:commutator}

We first give the structure of the defect $\dfh$, the exact target of any closure, from the finite-volume balance~\cite{leveque2002finite} with a nonlocal interface flux.

\begin{theorem}[Exact finite-volume defect identity]\label{thm:defect}
Let $u$ be a solution of Eq.~\eqref{eq:pde} that satisfies Assumption~\ref{ass:regularity}, and let $\Sh$ be a conservative coarse scheme of the form of Eq.~\eqref{eq:fvscheme} with arbitrary numerical fluxes.
Then the one-step defect is the discrete divergence of a face-flux error,
\begin{equation}\label{eq:defect}
\begin{gathered}
\dfh(u(t_n))_j=-\frac{\Delta t}{h}\Big(G_{j+1/2}-G_{j-1/2}\Big),\\
G_{j+1/2}=\frac{1}{\Delta t}\int_{t_n}^{t_{n+1}}F(u(t))(\xf{j})\dd t-\hat F_{j+1/2}(\Ph u(t_n)),
\end{gathered}
\end{equation}
where the nonlocal part $-\nu\Ra u(\xf{j})$ of the exact interface flux depends on the values of $u$ on all of $\T$.
The cell averages of the exact solution satisfy the semi-discrete identity
\begin{equation}\label{eq:semidiscrete}
\frac{\mathrm{d}}{\mathrm{d}t}(\Ph u)_j=-\frac1h\Big(F(u)(\xf{j},t)-F(u)(x_{j-1/2},t)\Big)+(\Ph f)_j.
\end{equation}
If $\Sh$ integrates the forcing only approximately from cell averages of $f$, for instance at stage times, then $\dfh$ contains in addition the forcing quadrature error.
This error has zero sum because $\sum_j(\Ph f)_j(t)=h^{-1}\int_\T f\dd x=0$ at every time, and it is therefore also a discrete divergence by \Cref{cor:flux}.
\end{theorem}

\Cref{thm:defect}, proved in \ref{app:proofs:defect}, places the closure problem at the cell faces, where the exact flux depends on $u$ on the whole torus while a coarse scheme sees only cell averages.
A related identity with a face-averaging filter and local fluxes underlies the large-eddy closures of Agdestein and Sanderse~\cite{agdestein2025discretize}.

\begin{corollary}[Flux form, mass and the size of the ideal closure]\label{cor:flux}
Let $\mathcal T_h=\dfh/\Delta t$ denote the truncation error per unit time, so that a scheme of order $q_t$ in time and $q$ in space has $\mathcal T_h=O(\Delta t^{q_t}+h^q)$ on smooth solutions.
\begin{enumerate}
\item[(i)] A vector $r\in\R^N$ can be written as $r=-\frac{\Delta t}{h}\Dm g$ if and only if $\sum_jr_j=0$, and then $g$ is unique up to an additive constant.
\item[(ii)] Every coarse scheme with $\sum_j\Sh(v)_j=\sum_jv_j+\sum_j\int_{t_n}^{t_{n+1}}(\Ph f)_j\dd t$ for all $v$ has the conservative form of Eq.~\eqref{eq:fvscheme}, with numerical fluxes that are unique up to a $v$-dependent constant.
In particular, \Cref{thm:defect} applies to the exponential Runge-Kutta step with the exact fractional symbol used in this work (\Cref{sec:training} and Proposition~\ref{prop:mass}).
\item[(iii)] The defect of an exact solution carries no mass, $m_h(\dfh(u))=0$, and every correction of the form $\mathcal M=-\frac{\Delta t}{h}\Dm g_\theta$ conserves $m_h$ for every value of the parameters $\theta$.
\item[(iv)] On a closable set, the ideal additive closure is exactly
\[
\mathcal M^\ast(\Ph u)=\dfh(u)=\Delta t\,\mathcal T_h(u)=-\frac{\Delta t}{h}\Dm G,
\]
with $G$ from Eq.~\eqref{eq:defect} for the numerical fluxes of item~(ii), which absorb the forcing quadrature error when $\Sh$ integrates the forcing only approximately.
The ideal closure therefore equals $\Delta t$ times the truncation error per unit time.
\item[(v)] If $t\mapsto F(u(t))(\xf{j})$ is twice continuously differentiable, then the face-flux error has the expansion
\begin{align*}
G_{j+1/2}&=\Big[F(u(t_n))(\xf{j})-\hat F_{j+1/2}(\Ph u(t_n))\Big]\\
&\quad+\frac{\Delta t}{2}\,\partial_tF(u)(\xf{j},t_n)+O(\Delta t^2).
\end{align*}
For a numerical flux evaluated at $t_n$ only, the bracket is the reconstruction and model error of the numerical flux and the second term is the time-averaging error.
\end{enumerate}
\end{corollary}

Thus, the flux form and discrete mass conservation are equivalent on a periodic grid (\ref{app:proofs:flux}), so the flux form of a learned correction adds only the locality of $g_\theta$.
For a flux evaluated at $t_n$, item~(v) in \Cref{cor:flux} leaves the time-averaging error to the integrator and isolates the reconstruction error $F(u)(\xf{j})-F(\mathcal D\Ph u)(\xf{j})$ of a decoded state $\mathcal Dv$.
A multi-stage integrator approximates the time-averaging term to its own order.
If $\mathcal Dv$ has the correct cell averages, the reconstruction error is governed mainly by the decoding error near the faces, since the far-field flux of an error with zero cell averages cancels to leading order (\Cref{sec:atoms}).

Learned closures usually rely on local stencils or convolutional networks with a finite receptive field~\cite{barsinai2019learning,kochkov2021machine}, and the next proposition quantifies which part of the flux they cannot see.

\begin{proposition}[Near field, far field and local surrogates]\label{prop:nearfar}
Let $1<\alpha<2$, and let $C_\alpha$ and $\kappa_\alpha$ be the constants of Lemma~\ref{lem:riesz}.
\begin{enumerate}
\item[(i)] Let $u\in C^4(\T)$ be extended periodically to $\R$, and let the cell width $h$ tend to zero. Then
\begin{equation}\label{eq:farfield}
C_\alpha\int_{|z|>h}\frac{u(x)-u(x+z)}{|z|^{1+\alpha}}\dd z=(-\Delta)^{\alpha/2}u(x)+\frac{C_\alpha}{2-\alpha}\,u''(x)\,h^{2-\alpha}+O(h^{4-\alpha}),
\end{equation}
so the far field is the whole operator up to $O(h^{2-\alpha})$, and the near field vanishes only like $h^{2-\alpha}$.
\item[(ii)] Suppose that the advective part of the numerical flux is consistent of order $q$, and that its fractional part is consistent of order $q$ with a surrogate $-\nu\tilde{\mathcal R}u$ instead of $-\nu\Ra u$, for instance a local diffusion.
Consistency is meant in the flux-difference sense, in which the face errors at $t_n$ and their discrete divergences $\frac1h\Dm$ are both $O(h^q)$.
If, in addition, $(\Ra-\tilde{\mathcal R})u$ is continuously differentiable in $x$ and $\partial_tF(u)$ and $\partial_t\partial_xF(u)$ are bounded, then the face-flux error and the truncation error satisfy
\begin{gather*}
G_{j+1/2}=-\nu(\Ra-\tilde{\mathcal R})u(\xf{j})+O(h^q+\Delta t),\\
\mathcal T_h=\nu\Ph\big[\partial_x(\Ra-\tilde{\mathcal R})u\big]+O(h^q+\Delta t).
\end{gather*}
The leading term of $\mathcal T_h$ tends to $\nu\,\partial_x(\Ra-\tilde{\mathcal R})u$ as $h\to0$, which is generically nonzero, so this is an $O(1)$ model error and not a truncation error.
\item[(iii)] For $u\in C^1(\T)$, the part of $\Ra u(x)$ that comes from $|y-x|<\rho$ is at most $2\kappa_\alpha\|u'\|_\infty\rho^{2-\alpha}/(2-\alpha)$ in absolute value.
At a face $x$, the part of $\Ra u(x)$ from $|y-x|\ge Rh$ equals $\Ra u(x)-O((Rh)^{2-\alpha})$ and tends to $\Ra u(x)$ as $h\to0$ at fixed $R$.
It depends on the values of $u$ in the $R$ cells on each side of $x$ only through their periodic images, whose contribution is $O(\|u'\|_\infty(Rh)^3)$.
\end{enumerate}
\end{proposition}

The proof is given in \ref{app:proofs:nearfar}. By part~(i) the far field dominates the fractional flux when $\alpha$ is close to one.
Truncated fractional centered differences~\cite{ortigueira2006riesz,celik2012crank} are local surrogates in the sense of part~(ii), and by part~(iii) learned models with a finite receptive field miss the far part in the same way.
We therefore evaluate the Riesz flux of a global reconstruction exactly by the fast Fourier transform and let a local network with a receptive field of $17$ cells correct only the reconstruction (\Cref{sec:locality}).

To ask whether the cell averages determine the defect, we work in a separable Banach space $U$ of states that is continuously embedded in $L^2(\T)$, for instance $H^s(\T)$ with $s>\alpha+\frac12$, so that $\Ph$ is continuous on $U$.
For fixed $t_n$, forcing and coefficients, $\Ev$ is defined on the set $\mathrm{dom}\,\Ev\subset U$ where Assumption~\ref{ass:regularity} holds, which we assume to be Borel, with $\Ev$ Borel on it.
For a decoder $\mathcal D:\R^N\to U$ we write $\Phi_{\mathcal D}(v)=-\frac1h\Dm\big[F(\mathcal Dv)(x_{\cdot+1/2})\big]$ for the reconstruct-evolve-average right-hand side of \Cref{sec:decoder}.

\begin{proposition}[Closability and the $L^2$-optimal closure]\label{prop:closable}
Let $M\subset\mathrm{dom}\,\Ev$ be a set of states, and let $\Sh$ be a coarse scheme with one-step defect $\dfh$.
Items~(i) to~(iii) concern a fixed forcing and fixed coefficients, and items~(i) and~(ii) also concern a fixed time $t_n$, except for the composition of step maps at the end of item~(i).
\begin{enumerate}
\item[(i)] \emph{(Fiber condition)} The diagram is closable on $M$ if and only if $u,w\in M$ and $\Ph u=\Ph w$ imply $\Ph\Ev u=\Ph\Ev w$.
In that case $\widetilde{\mathcal S}_h$ is unique on $\Ph(M)$ and $\mathcal M^\ast(\Ph u)=\dfh(u)$.
If the step maps $\mathcal E^{(k)}$ from $t_k$ to $t_{k+1}$ all send $M$ into itself and satisfy the fiber condition on $M$, then their closing maps $\widetilde{\mathcal S}_h^{(k)}$ satisfy
\[
\widetilde{\mathcal S}_h^{(n-1)}\circ\cdots\circ\widetilde{\mathcal S}_h^{(0)}\circ\Ph=\Ph\circ\mathcal E^{(n-1)}\circ\cdots\circ\mathcal E^{(0)}\quad\text{on }M\text{ for every }n.
\]
\item[(ii)] \emph{(Ideal decoder)} If $\Ph$ is injective on $M$, the decoder $\mathcal D^\ast=(\Ph|_M)^{-1}$ gives $\widetilde{\mathcal S}_h=\Ph\Ev\mathcal D^\ast=\Sh+\dfh\circ\mathcal D^\ast$ on $\Ph(M)$.
If $M$ is compact in a topology in which $\Ph$ is continuous, then $\mathcal D^\ast$ is continuous.
If $M$ is bounded in $L^2$ with $\|u-w\|_{L^2}\le C|\Ph u-\Ph w|^\eta$ on $M$ for some $\eta\in(0,1]$, $\Ph\Ev$ is Lipschitz from $L^2$ and $\Sh$ is Lipschitz on $\Ph(M)$, then $\widetilde{\mathcal S}_h$ and $\mathcal M^\ast$ are $\eta$-H\"older on $\Ph(M)$.
Both maps then extend to $\eta$-H\"older maps on all of $\R^N$.
\item[(iii)] \emph{(Semi-discrete version)} A map $\Phi$ with $\frac{\mathrm{d}}{\mathrm{d}t}\Ph u=\Phi(\Ph u)+\Ph f(t)$ along all trajectories in $M$ with the same forcing exists if and only if $\Ph u=\Ph w$ implies $\Dm[F(u)(x_{\cdot+1/2})]=\Dm[F(w)(x_{\cdot+1/2})]$ for $u,w\in M$.
If $\Ph$ is injective on $M$, the reconstruct-evolve-average right-hand side $\Phi_{\mathcal D^\ast}$ is exact on $\Ph(M)$.
\item[(iv)] \emph{($L^2$-optimal closure)} Let parameters $\xi$ in a Polish space $\Xi$ collect $t_n$, the forcing sample and the coefficients $\alpha$ and $\nu$, on which $\Ev$ and $\Sh$ may depend.
Let $\mu$ be a Borel probability measure on $U\times\Xi$, for instance the occupation measure of the training trajectories, and let $V$ be a Borel function of $(u,\xi)$.
Examples of such coarse variables are $V=\Ph u$ and the network input $V=(\Ph u,\alpha,\nu)$.
Let $Y=Y(u,\xi)\in L^2(\mu,\R^d)$ be any Borel target, such as the defect $\dfh$, the exact face fluxes $(F(u)(\xf{j}))_j$ or the face values.
Let $m^\star$ be a Borel map with $m^\star(V)=\EE[Y\mid V]$, which exists by the Doob-Dynkin lemma.
Then, with $|\cdot|$ the Euclidean norm on $\R^d$ or the norm of any other inner product, every Borel map $m$ with $m(V)\in L^2$ satisfies the Pythagorean identity of conditional expectation,
\[
\EE|Y-m(V)|^2=\EE|Y-m^\star(V)|^2+\EE|m^\star(V)-m(V)|^2.
\]
\item[(v)] \emph{(Irreducible error)} The minimum of the mean-square error over all such $m$ is $J_{\mathrm{irr}}(Y)=\EE\,\tr\Cov(Y\mid V)$, with the trace taken in the same inner product.
It involves no network, and it does not change when a square-integrable Borel function of $V$ is added to $Y$, so for $Y=\dfh$ it is the same for all schemes that differ by such a function.
It vanishes if and only if $Y$ is almost surely a Borel function of $V$.
If $\mu$ is concentrated on one value of $\xi$, with $\mu(\mathrm{dom}\,\Ev)=1$, $V=\Ph u$, $Y=\dfh$ and a Borel scheme $\Sh$ with $\Sh(V)\in L^2$, this is equivalent to closability on a set of full $\mu$-measure.
In that case $\widetilde{\mathcal S}_h=\Sh+m^\star$ holds almost everywhere with respect to $\mu\circ\Ph^{-1}$.
\end{enumerate}
\end{proposition}

The proof is given in \ref{app:proofs:closable}, and parts~(iv) and~(v) are the optimal prediction framework of Chorin, Hald and Kupferman~\cite{chorin2000optimal,chorin2002optimal}.
The ideal sub-grid model of large-eddy simulation is the same conditional average~\cite{langford1999optimal}, and Beck et al.~\cite{beck2019deep} studied closures whose target includes the discretization operator itself.
When the cell averages determine the state, part~(ii) shows that the exact coarse scheme reconstructs, evolves exactly and averages, as Godunov-type schemes do~\cite{leveque2002finite}, which motivates the decoder-only closure of \Cref{sec:decoder}.
For $Y=\dfh$, with $G$ shifted to zero mean over the faces, part~(iv) gives $m^\star(V)=-\frac{\Delta t}{h}\Dm\EE[G\mid V]$ in the flux form of \Cref{cor:flux}.

\begin{remark}[Dependence on the data distribution]\label{rem:distribution}
If the data are not closable almost surely, the optimal closure $m^\star$ depends on the measure $\mu$ of the initial data, forcing, coefficients and time windows.
A closure that attains the floor of Proposition~\ref{prop:closable} on one distribution is then biased on another.
Our $\mu$ contains one diffusivity per fractional order and the forcing wavenumbers 3 to 6 only, so a change of $\nu$ at fixed $\alpha$, or of the wavenumbers, moves the target itself and not only the network inputs.
A history of cell averages cannot raise the floor, by the tower property, but the closure studied here uses the current cell averages only.
\end{remark}

\subsection{A decoder-only closure}\label{sec:decoder}

Since the optimal closure of the flux defect depends on the sub-cell state (Theorem~\ref{thm:defect}, \Cref{prop:closable}), we learn only a decoder $\Dth$ from cell averages to functions on $\T$.
The encoder is the fixed projection $\Ph$, and $\Dth$ is tied to the grid, so we make no claim of discretization invariance.

\paragraph{The trigonometric histopolant}
For even $N$, let $\T_N$ be the span of the modes $e^{ikx}$ with $|k|\le N/2-1$ and of the Nyquist mode $\sin(Nx/2)$, which vanishes at every face.
The base decoder is the trigonometric histopolant $\Ih v$, the unique element of $\T_N$ whose exact cell averages are the data $v\in\R^N$.
With the discrete Fourier coefficients $\hat v_k=N^{-1}\sum_{j=0}^{N-1}v_je^{-ikjh}$ and $\mu_k=(e^{ikh}-1)/(ikh)$ of $\Ph$, the histopolant is
\begin{equation}\label{eq:histopolant}
\Ih v(x)=\sum_{|k|\le N/2-1}\frac{\hat v_k}{\mu_k}\,e^{ikx}+\frac{\pi}{2}\,\hat v_{N/2}\,\sin\Big(\frac N2x\Big).
\end{equation}
Since $\mu_k\ne0$ for $|k|\le N/2-1$ and the Nyquist mode has cell averages $\frac2\pi(-1)^j$, we have $\Ph\Ih=I$. 
For odd $N$ the Nyquist term is absent and $\T_N$ is spanned by the modes with $|k|\le(N-1)/2$.
The smooth $\Ih v$ avoids the obstruction of \Cref{lem:jump}, but unlike WENO, it is global and oscillates over several cells near a sub-cell front.

\paragraph{Decoders that preserve cell averages}
The next proposition characterizes all decoders with $\Ph\mathcal D=I$ and the Godunov-type scheme~\cite{leveque2002finite} in which they are used.

\begin{proposition}[Cell-average preserving decoders]\label{prop:decoders}
Let $\eta\in(\alpha-1,1]$, let $X=C^{0,\eta}(\T)$, and let $\Ih:\R^N\to X$ be any linear map with $\Ph\Ih=I$.
\begin{enumerate}
\item[(i)] For any map $g_\theta:\R^N\to X$, the decoder $\Dth v=\Ih v+(I-\Ih\Ph)g_\theta(v)$ satisfies $\Ph\Dth v=v$ for every $v\in\R^N$ and every value of $\theta$.
Conversely, every decoder $\mathcal D:\R^N\to X$ with $\Ph\mathcal D=I$ has this form, with the choice $g_\theta=\mathcal D$.
\item[(ii)] For every such decoder, the reconstruct-evolve-average semi-discretization, which we call the REA scheme,
\begin{equation}\label{eq:rea}
\frac{\mathrm{d}v}{\mathrm{d}t}=\Phi_{\mathcal D}(v)+\Ph f,\qquad \Phi_{\mathcal D}(v)=-\frac1h\Dm\big[F(\mathcal Dv)(x_{\cdot+1/2})\big],
\end{equation}
is well defined, and it is conservative in the sense that $h\sum_j\Phi_{\mathcal D}(v)_j=0$ for every $v$.
\item[(iii)] If $\Ih$ is the histopolant of Eq.~\eqref{eq:histopolant}, the fractional part of $\Phi_{\Ih}$ is the exact coarse fractional operator, in the sense that
\[
-\frac1h\Dm\big[-\nu\Ra\Ih v(x_{\cdot+1/2})\big]=\Ph\big[-\nu(-\Delta)^{\alpha/2}\Ih v\big]=\mathcal F_N^{-1}\big(-\nu|k|^\alpha\hat v_k\big),
\]
where $\mathcal F_N^{-1}$ denotes the inverse discrete Fourier transform over the wavenumbers $-N/2<k\le N/2$.
\end{enumerate}
\end{proposition}

Item~(i), proved with items~(ii) and~(iii) in \ref{app:proofs:decoders}, shows that the constraint {entails no loss of generality}, and NRC uses the decoder
\begin{equation}\label{eq:nrcdecoder}
\Dth v=\Ih v+(I-\Ih\Ph)\,g_\theta(v),
\end{equation}
where $g_\theta$ is the rational enrichment of \Cref{sec:atoms} and $w_\theta=(I-\Ih\Ph)\,g_\theta(v)$ is the sub-grid part.
Like approximate deconvolution~\cite{stolz1999approximate} and learned face reconstructions~\cite{barsinai2019learning},  {NRC first reconstructs a sub-cell state from the cell averages and then evaluates the fluxes on that state.
Unlike them, it keeps $\Ph\Dth=I$ and uses the exact fractional operator of item~(iii).}

\paragraph{{Structure of the NRC decoder}}
{The decoder of NRC is built in four stages, sketched in \Cref{fig:schematic}(b) and implemented in Algorithm~\ref{alg:nrc}.
(i)~A backbone of four circular convolutions with kernel size 5, 40 channels and GELU activations maps the cell averages, together with the order $\alpha$ and the diffusivity $\nu$, to 40 features per cell, with a receptive field of 17 cells.
(ii)~A pointwise head maps the features of cell $j$ to the positions, depths and complex amplitudes of $p=2$ poles inside $C_j$ and to a gate that switches the enrichment of the cell on or off.
(iii)~The $pN$ poles define the rational enrichment $g_\theta$, a sum of the atoms of \Cref{sec:atoms}, whose values, cell averages and Riesz fluxes have closed forms.
(iv)~The projection $I-\Ih\Ph$ removes the cell averages of $g_\theta$, and adding the result to the histopolant gives $\Dth v$ in Eq.~\eqref{eq:nrcdecoder}.
The network has 25,329 parameters and never evaluates a flux, and \Cref{sec:training} specifies its inputs, its output layer and its training.
The learned baselines of \Cref{sec:setup} use the same backbone or a Fourier neural operator, but their heads output a correction of the face flux or of the update instead of a sub-cell reconstruction (\Cref{tab:hyper}).}

\paragraph{How well a fixed histopolant closes the diagram}
\Cref{thm:rea} describes how far the ideal REA map $\mathcal S_{\mathrm{REA}}=\Ph\circ\Ev\circ\Ih$ is from closing the diagram of \Cref{sec:problem}, with $Q_N$ the $L^2$-orthogonal projection onto the complement of $\T_N$.

\begin{theorem}[Approximate closability by the REA map]\label{thm:rea}
Let $\sigma(k)=\sinc(kh/2)$ with $\sinc z=\sin z/z$, and let $\|v\|_h^2=h\sum_jv_j^2$ be the discrete norm.
For odd $N$ let $C_N=\sigma_{\min}^{-1}=1/\sinc(\frac{\pi(N-1)}{2N})$, with $\sigma_{\min}$ the minimum of $\sigma$ over the resolved wavenumbers $|k|\le(N-1)/2$, and for even $N$ let $C_N=\max\{1/\sinc(\frac{\pi(N-2)}{2N}),\,\pi/(2\sqrt2)\}$.
\begin{enumerate}
\item[(i)] For every $u\in L^2(\T)$ we have $\|u-\Ih\Ph u\|_{L^2}\le C_N\|Q_Nu\|_{L^2}$ with $C_N<\pi/2$, where the constant is sharp for odd $N$ and $N=64$ gives $C_N=1.52354$.
\item[(ii)] If $M\subset L^2(\T)$ satisfies $\|\Ph\Ev u-\Ph\Ev\Ih\Ph u\|_h\le L_E\|u-\Ih\Ph u\|_{L^2}$ for all $u\in M$, then the REA map satisfies $\|\mathcal S_{\mathrm{REA}}(\Ph u)-\Ph\Ev u\|_h\le L_EC_N\|Q_Nu\|_{L^2}$ for every $u\in M$.
\item[(iii)] If $L_E$ is a Lipschitz constant of $\Ph\Ev$ on $M$, the fiber oscillation obeys
\[
\sup_{u,w\in M,\ \Ph u=\Ph w}\|\Ph\Ev u-\Ph\Ev w\|_h\le2L_EC_N\sup_{u\in M}\|Q_Nu\|_{L^2}.
\]
\item[(iv)] On the Sobolev ball $2\pi\sum_k(1+k^2)^s|\hat u_k|^2\le R^2$ and on the analytic class $2\pi\sum_ke^{2\tau|k|}|\hat u_k|^2\le R^2$, the right-hand side of item~(ii) is at most $L_EC_N\lceil N/2\rceil^{-s}R$ and at most $L_EC_Ne^{-\tau\lceil N/2\rceil}R$, respectively.
\end{enumerate}
\end{theorem}

The proof in \ref{app:proofs:rea} combines the classical aliasing identity with the Cauchy-Schwarz inequality, and we claim no novelty for $C_N$.

\paragraph{Mass conservation for every parameter value}
By Corollary~\ref{cor:flux}, conservation on the periodic grid is equivalent to flux form, which the time-discrete NRC step inherits.

\begin{proposition}[Mass conservation of the NRC step]\label{prop:mass}
Let $m_h(v)=h\sum_jv_j$, and let one NRC step be the ETDRK2 step~\cite{cox2002exponential} for $\dot v=\mathcal Lv+\mathcal N_\theta(v,t)$, where $\mathcal L$ has the exact symbol $\lambda_k=-\nu|k|^\alpha$ and
\[
\mathcal N_\theta(v,t)=-\frac1h\Dm\Big[\tfrac12(\Dth v)^2-\nu\Ra w_\theta\Big](x_{\cdot+1/2})+\Ph f(t)
\]
collects the explicit flux divergence and the forcing.
Then for every network output the mass changes only through the forcing, $m_h(v^{n+1})=m_h(v^n)+\frac{\Delta t}{2}\big(m_h(\Ph f(t_n))+m_h(\Ph f(t_{n+1}))\big)$.
In particular, with zero-mean forcing the mass is conserved exactly for every value of $\theta$.
\end{proposition}

The proof in \ref{app:proofs:mass} evaluates the ETDRK2 step on the mean mode, where $\lambda_0=0$, and our forcing has zero mean.

\subsection{Rational atoms and exact nonlocal fluxes}\label{sec:atoms}

By Theorem~\ref{thm:rea}, $g_\theta$ is needed only on sub-cell scales, where it must resolve transitions, stay continuous at the faces (\Cref{lem:jump}) and have a closed-form Riesz flux.
Rational functions with complex poles near the real axis meet all three requirements.

\begin{definition}[Rational atoms]\label{def:atoms}
For a pole parameter $a=e^{-\delta+i\phi}$ with depth $\delta>0$ and position $\phi\in\T$, let $z(x)=\bar ae^{ix}=e^{-\delta+i(x-\phi)}$.
The rational atom with pole parameter $a$ is the smooth, $2\pi$-periodic and mean-free function
\begin{equation}\label{eq:atom}
\psia(x)=\frac{z}{1-z}=\sum_{n\ge1}e^{-n\delta}e^{in(x-\phi)}.
\end{equation}
\end{definition}

The real field $2\Real[\gamma\psia]$ with $\gamma\in\C$ has simple poles at $\phi\mp i\delta+2\pi\Z$, and the depth $\delta$ sets the width of the transition.
The NRC enrichment is
\begin{equation}\label{eq:enrichment}
g_\theta(v)(x)=2\Real\sum_{j=0}^{N-1}\sum_{m=1}^{p}\gamma_{jm}\,\psi_{a_{jm}}(x),\qquad a_{jm}=e^{-\delta_{jm}+i\phi_{jm}},
\end{equation}
where a network maps $v$ to positions $\phi_{jm}\in C_j$, depths $\delta_{jm}\in(0.02h,3.02h)$ and complex amplitudes $\gamma_{jm}$ with $p=2$ (\Cref{sec:training}), and never evaluates a flux.

\paragraph{Why poles}
Trigonometric reconstructions on $N$ cells cannot represent a pole closer to the real axis than a fraction of a cell, whatever the data given.

\begin{proposition}[One-atom gap]\label{prop:gap}
Let $\delta>0$, $\phi\in\T$, $\gamma\in\C\setminus\{0\}$ and $u=2\Real[\gamma\psia]$ with $a=e^{-\delta+i\phi}$.
For every trigonometric polynomial $s$ of degree at most $N/2$, the relative error obeys $\|u-s\|_{L^2}/\|u\|_{L^2}\ge e^{-\delta N/2}=e^{-\pi\delta/h}$.
In particular every reconstruction with values in $\T_N$, including $\Ih\Ph u$, has a relative $L^2$ error of at least $e^{-\pi\delta/h}$, which equals 0.39 for $\delta=0.3h$, while a single atom represents $u$ exactly.
\end{proposition}

The proof in \ref{app:proofs:gap} compares the geometric Fourier tail of $u$ beyond degree $N/2$ with its $L^2$ norm, and \Cref{sec:whypoles} tests the bound.

\paragraph{Closed-form fluxes}
The next lemma gives the closed forms, whose derivations and implementation are provided in \ref{app:atoms}.

\begin{lemma}[Closed forms for rational atoms]\label{lem:atoms}
Let $\psia$ be the atom of Eq.~\eqref{eq:atom} with $z=z(x)$, and let $\Li_s(w)=\sum_{n\ge1}w^nn^{-s}$ denote the polylogarithm.
\begin{enumerate}
\item[(i)] The face values and slopes are given by $\psia(x)=z/(1-z)$ and $\psia'(x)=iz/(1-z)^2$ at every $x\in\T$.
\item[(ii)] The cell averages are $(\Ph\psia)_j=\frac{1}{ih}\big[\log(1-z(x_{j-1/2}))-\log(1-z(\xf{j}))\big]$ with the principal logarithm, which is continuous along the cell because $\Real(1-z)\ge1-e^{-\delta}>0$.
\item[(iii)] The Riesz flux and the fractional Laplacian are $\Ra\psia(x)=i\Li_{1-\alpha}(z(x))$ and $(-\Delta)^{\alpha/2}\psia(x)=\Li_{-\alpha}(z(x))$.
Hence the fractional flux of the real field $2\Real\sum_k\gamma_k\psi_{a_k}$ equals $-2\nu\Real\sum_ki\gamma_k\Li_{1-\alpha}(z_k(x))$.
\item[(iv)] With $\mu=-\delta+i(x-\phi)$ and $x-\phi$ reduced to $[-\pi,\pi)$, the expansion of DLMF 25.12.12~\cite{olver2026dlmfgen} reads
\begin{equation}\label{eq:dlmf}
\Li_{1-\alpha}(e^\mu)=\Gamma(\alpha)(-\mu)^{-\alpha}+\sum_{k\ge0}\frac{\zeta(1-\alpha-k)}{k!}\mu^k,\qquad|\mu|<2\pi,
\end{equation}
and the condition $|\mu|<2\pi$ holds whenever $\delta<\sqrt3\,\pi$.
The derivative $\partial_\mu\Li_s(e^\mu)=\Li_{1}(e^\mu)$ gives exact gradients with respect to the pole parameters.
\item[(v)] Uniformly for $\delta\le1$, $\Ra\psia(x)=i\Gamma(\alpha)(-\mu)^{-\alpha}+O(1)$, so one atom has a flux of size $\Gamma(\alpha)\delta^{-\alpha}$ next to its pole that decays like $|x-\phi|^{-\alpha}$ within the period.
\end{enumerate}
\end{lemma}

The proof in \ref{app:proofs:atoms} applies Fourier multipliers termwise to the absolutely convergent series of Eq.~\eqref{eq:atom} and quotes DLMF 25.12.12.
{With $K=pN$ atoms, all face quantities follow from $O(pN^2)$ evaluations of logarithms and polylogarithms, without quadrature or aliasing errors.
These evaluations dominate the cost of a step (\ref{app:atoms}), and \Cref{sec:future} outlines how to reduce it.}

\paragraph{Locality of the learned flux}
By \Cref{prop:nearfar} the Riesz flux of the solution is dominated by its far field, but the corrector $w_\theta=(I-\Ih\Ph)g_\theta(v)$ has zero mean on every cell.

\begin{lemma}[Far-field cancellation for zero cell means]\label{lem:dipole}
Let $x$ be a face of the uniform grid of width $h$, and let $\zeta$ denote the Riemann zeta function.
\begin{enumerate}
\item[(i)] If $w\in W^{1,\infty}(\R)$ has compact support and zero mean on every cell, then
\[
|\Ra w(x)|\le\kappa_\alpha\Big(\frac{2h^{2-\alpha}}{2-\alpha}\|w'\|_{L^\infty}+\alpha\,\zeta(1+\alpha)\,h^{1-\alpha}\|w\|_{L^\infty}\Big),
\]
where $\|w'\|_{L^\infty}$ may be taken over the two cells adjacent to $x$.
For $w$ periodic on $\T$ with zero cell means and $N$ even, the same bound holds after adding $\kappa_\alpha\pi c_{\mathrm{per}}h\|w\|_{L^\infty}$ with $c_{\mathrm{per}}=2\alpha\sum_{n\ge1}(\pi(2n-1))^{-1-\alpha}$.
\item[(ii)] If $u$ and $v$ are bounded on $\R$, or periodic and bounded, with $\Ph u=\Ph v$, and $\rho=m_0h$ with an integer $m_0\ge2$, then the far-field part of their flux difference obeys
\[
\Big|\kappa_\alpha\int_{|y-x|\ge\rho}\sgn(y-x)\,\frac{u(y)-v(y)}{|y-x|^\alpha}\dd y\Big|\le\kappa_\alpha h\,\|u-v\|_{L^\infty}(\rho-h)^{-\alpha}.
\]
\item[(iii)] A cell that is not adjacent to $x$, whose center lies at signed distance $d$ from $x$ and on which $w$ has mean $\bar w$, contributes to $\Ra w(x)$ the amount $\kappa_\alpha h\bar w\,\sgn(d)|d|^{-\alpha}+O\big(h^2|d|^{-1-\alpha}\|w\|_{L^\infty}\big)$.
\end{enumerate}
\end{lemma}

The proof in \ref{app:proofs:dipole} subtracts the kernel value at each cell center, which the zero cell means allow.

\begin{corollary}[Windowed sub-grid flux]\label{cor:window}
In NRC the sub-grid Riesz flux is the flux of $w_\theta=(I-\Ih\Ph)g_\theta(v)$, which has zero mean on every cell.
Restricting its evaluation to $|y-x|<m_0h$ changes each face flux by at most $\kappa_\alpha h\|w_\theta\|_{L^\infty}((m_0-1)h)^{-\alpha}$.
Relative to the natural scale $\kappa_\alpha h^{1-\alpha}\|w_\theta\|_{L^\infty}$ of item~(i) of \Cref{lem:dipole}, the truncation error is therefore at most $(m_0-1)^{-\alpha}$.
\end{corollary}

\begin{proof}
Apply item~(ii) of \Cref{lem:dipole} with $u=w_\theta$ and $v=0$, which have the same zero cell averages, and divide by the natural scale.
\end{proof}

Zero cell means turn the monopole decay $h|d|^{-\alpha}$ of a cell into a dipole decay $h^2|d|^{-1-\alpha}$, so the projection makes the learned flux local while $\Ih$ carries the global part exactly.

\paragraph{What the poles represent}
To interpret the learned poles, we recall where fronts have complex singularities, which for fractional dissipation requires an assumption.

\begin{assumption}[Traveling wave]\label{ass:tw}
For $\alpha\in(1,2)$ there is a stationary profile $\chi_\alpha$ of $(\chi^2/2)'=-(-\Delta)^{\alpha/2}\chi$ on $\R$ with $\chi_\alpha(\mp\infty)=\pm1$, unique up to translation and analytic in a maximal strip $|\Imag x|<\delta_1(\alpha)$ with $0<\delta_1(\alpha)<\infty$.
\end{assumption}

\begin{proposition}[Poles of fronts]\label{prop:poles}
Let $U>0$ and $\nu>0$, and let $x_0\in\R$ denote the position of the front in both statements below.
\begin{enumerate}
\item[(i)] For $\alpha=2$, the stationary viscous shock of $u_t+uu_x=\nu u_{xx}$ joining $U$ to $-U$ is $u=-U\tanh(U(x-x_0)/(2\nu))$.
It is meromorphic with simple poles at $x_0+i\pi\nu(2k+1)/U$, $k\in\Z$, each with residue $-2\nu$, so its strip of analyticity has half-width $\delta=\pi\nu/U$.
With the front width $\ell=\Delta U/(2\max|u_x|)=2\nu/U$ used in our pole diagnostics, the depth of the nearest poles reads $\delta=\frac\pi2\ell$.
\item[(ii)] Under Assumption~\ref{ass:tw}, the stationary profiles of $(u^2/2)'=-\nu(-\Delta)^{\alpha/2}u$ joining $U$ to $-U$ are exactly $u=U\chi_\alpha((x-x_0)/\ell_\alpha)$ with $\ell_\alpha=(\nu/U)^{1/(\alpha-1)}$.
Their strip of analyticity, and hence the depth of their nearest complex singularities, has the half-width $\delta(U,\nu)=\delta_1(\alpha)(\nu/U)^{1/(\alpha-1)}$, with $\delta_1(2)=\pi$.
\end{enumerate}
\end{proposition}

The proof in \ref{app:proofs:poles} integrates the stationary equation once for item~(i) and uses the dilation homogeneity of $(-\Delta)^{\alpha/2}$ for item~(ii).
Item~(ii) is conditional, since monotone traveling waves of fractal Burgers equations are known to exist only in special cases~\cite{achleitner2018asymptotic}.
Unforced computations give exponents $\mathrm{d}\log\delta/\mathrm{d}\log\nu$ of 1.355 and 2.001 against the predicted $4/3$ and 2 at $\alpha=1.75$ and $1.5$, and 2.82 against 4 at $\alpha=1.25$, not yet asymptotic.
Our data use one diffusivity per order, so \Cref{sec:mechanism} compares learned depths with front widths instead of testing this scaling.

\begin{remark}[A conjecture on branch points]\label{rem:branch}
We conjecture that for $\alpha\in(1,2)$ the singularities of $\chi_\alpha$ nearest to the real axis are branch points with local behavior $A(x-x^\ast)^{1-\alpha}$ rather than poles.
Fourier coefficients of quasi-steady periodic fronts would then behave like $k^{\alpha-2}e^{-\delta k}$, but the fitted prefactor exponents lie above the conjectured $2-\alpha$, at 0.55 to 0.59 for $\alpha=1.5$ and 0.26 to 0.28 for $\alpha=1.75$.
AAA fits~\cite{nakatsukasa2018aaa} of such fronts use pole clusters, which is consistent with $p>1$ but does not determine $p$.
\end{remark}

Poles at a fixed positive distance give at most algebraic rates on algebraic singularities~\cite{borwein1994chebyshev}, which is why NRC places its poles adaptively (\ref{app:rational}).

\subsection{Consistency of the coupled scheme}\label{sec:consistency}

We now ask how accurate a decoder must be for the REA scheme of Eq.~\eqref{eq:rea} to be consistent.
Let $u$ be a solution with the regularity assumed in Theorem~\ref{thm:defect}, let $v=\Ph u$, and let $\mathcal D$ be any decoder with $\Ph\mathcal D=I$ whose output is bounded and H\"older continuous of order $\eta>\alpha-1$.
We measure the reconstruction error by $w=u-\mathcal Dv$ and the consistency error by the semi-discrete residual $r=\mathrm{d}v/\mathrm{d}t-\Phi_{\mathcal D}(v)-\Ph f$, and we write $[w]_\eta$ for the H\"older seminorm.

\begin{theorem}[Semi-discrete consistency of the reconstruction closure]\label{thm:consistency}
Under these assumptions the following four statements hold for the Burgers flux $F(u)=\frac12u^2-\nu\Ra u$.
\begin{enumerate}
\item[(i)] The error $w$ has zero mean on every cell, and $r_j=-\frac1h(\Delta F_{j+1/2}-\Delta F_{j-1/2})$ with $\Delta F=F(u)-F(\mathcal Dv)$ at the faces, so that $\|r\|_\infty\le\frac2h\max_j|\Delta F_{j+1/2}|$.
\item[(ii)] At every face the flux error obeys $|\Delta F_{j+1/2}|\le A\|w\|_{L^\infty}+\nu|\Ra w(\xf{j})|$ with $A=\|u\|_{L^\infty}+\frac12\|w\|_{L^\infty}$.
\item[(iii)] For $N$ even and every face $x$, the nonlocal part is bounded by the minimum of three quantities,
\begin{align*}
|\Ra w(x)|&\le\min\{B_1,B_2,B_3\},\qquad B_1=\sum_k|k|^{\alpha-1}|\hat w_k|,\\
B_2&=\kappa_\alpha\Big[\frac{2h^{2-\alpha}}{2-\alpha}\|w'\|_{L^\infty}+\big(\alpha\zeta(1+\alpha)h^{1-\alpha}+\pi c_{\mathrm{per}}h\big)\|w\|_{L^\infty}\Big],\\
B_3&=K_{\alpha,\eta}\|w\|_{L^\infty}^{1-\frac{\alpha-1}{\eta}}[w]_\eta^{\frac{\alpha-1}{\eta}}.
\end{align*}
The quantity $B_2$ requires $w'$ to be bounded, and $K_{\alpha,\eta}=\kappa_\alpha\big(\frac{2}{1+\eta-\alpha}+\frac{2}{\alpha-1}\big)$ is the constant of the interpolation inequality in \ref{app:classical}.
\item[(iv)] For $\mathcal D=\Ih$, the quantities $\sum_k|k|^{\alpha-1}|\hat w_k|$ and $\|w\|_{L^\infty}$ are at most $1+\frac\pi2$ times $\sum_{|k|\ge N/2}|k|^{\alpha-1}|\hat u_k|$ and $\sum_{|k|\ge N/2}|\hat u_k|$, respectively.
For $u$ bounded and analytic in a strip of half-width $\tau$, the unweighted sum is of order $e^{-\tau N/2}$ and the weighted sum of order $N^{\alpha-1}e^{-\tau N/2}$.
\end{enumerate}
\end{theorem}

The proof in \ref{app:proofs:consistency} compares the REA scheme with Theorem~\ref{thm:defect} and uses \Cref{lem:dipole} and an interpolation inequality.

\begin{corollary}[Accuracy required of a decoder]\label{cor:decoder}
Assume that $N$ is even and that $\|u\|_{L^\infty}\le C_u$, $\|w\|_{L^\infty}\le C_0h^q$ and $\|w'\|_{L^\infty}\le C_1h^{q-1}$, uniformly on a class of solutions.
Then the residual satisfies $\|r\|_\infty=O(h^{q-\alpha})$ on that class, so the REA scheme is consistent whenever $q>\alpha$.
A sufficient condition for consistency is therefore that the decoder is accurate in $W^{1,\infty}$, or alternatively in the weighted Wiener norm $\sum_k|k|^{\alpha-1}|\hat w_k|$ or in a H\"older norm of order $\eta>\alpha-1$.
\end{corollary}

\begin{proof}
Substituting the assumptions into items~(i) to~(iii) of Theorem~\ref{thm:consistency} shows that the advective part contributes $O(h^{q-1})$ to $\|r\|_\infty$.
The fractional part contributes $\frac2h\,O(h^{1+q-\alpha})=O(h^{q-\alpha})$ through the bound $B_2$, which dominates the advective part because $\alpha>1$.
\end{proof}

\begin{remark}[Stability and design rules]\label{rem:stability}
Consistency alone does not give convergence, and the learned scheme carries no stability guarantee that holds for every value of $\theta$.
\ref{app:tube} states a tube theorem in the spirit of L\'opez-Marcos and Sanz-Serna~\cite{lopezmarcos1988stability}, with $L$ the sum of the Lipschitz rates of the coarse scheme and the learned correction on a tube of radius $\varrho$ around $U^n=\Ph u(t_n)$.
Let the one-step consistency error of the corrected scheme, $\|\dfh(u(t_n))-\mathcal M(U^n)\|$, be at most $\Delta t\,\varepsilon$, and let $L>0$.
The error of the coarse states $V^n$ then stays below $e^{Lt_n}\|V^0-U^0\|+\varepsilon(e^{Lt_n}-1)/L$ as long as this bound stays below $\varrho$. Here, $\varepsilon$ is a supremum along the test trajectory, which $L^2$ training does not control, and over all reference states non-closability forces it to be at least the fiber oscillation of Theorem~\ref{thm:rea} divided by $2\Delta t$ (Remark~\ref{rem:tube}).
The theorem suggests three design rules for learned closures of stiff nonlocal equations.
Rule (i) lets the learned correction enter the update with a factor $\Delta t$, so that its Lipschitz contribution per step is of order $\Delta t$.
Rule (ii) keeps the stiff linear fractional operator in the coarse scheme, treated implicitly or exponentially, since a network that supplies the far field has a Lipschitz constant of order $\nu h^{-\alpha}$.
Rule (iii) controls the Lipschitz constant of the learned correction on the tube, for instance by spectral normalization.
{NRC follows rules (i) and (ii) through ETDRK2 with the exact symbol.
It does not enforce rule (iii), so its stability is an empirical observation, and none of its 216 in-distribution test rollouts blows up (\Cref{sec:accuracy}).}
\end{remark}

\subsection{Network, time integration and training}\label{sec:training}

This subsection fixes how the poles are chosen, the time integrator and the training, and \ref{app:training} gives the remaining settings.

\paragraph{Network}
Four circular convolutions of width 40 and kernel size 5 with GELU activations~\cite{hendrycks2016gaussian} map the inputs to 40 features per cell.
The inputs are $v_j/0.6$, $(\alpha-1.5)/0.3$ and $(\ln\nu+2.12)/0.9$, a fourth channel is affine in the last two (\ref{app:training}), and the forcing is not an input.
The receptive field of 17 cells acts as a shift-equivariant stencil, but $\Ih$, $\Ih\Ph$ and the tails of the atoms couple all cells (\Cref{sec:locality}).
A pointwise convolution maps the features of cell $j$ to the outputs $o^{\phi}_{jm}$, $o^{\delta}_{jm}$, $o^{\mathrm{re}}_{jm}$, $o^{\mathrm{im}}_{jm}$ of its $p=2$ poles and a gate output $o^{\pi}_{j}$, and sets
\begin{align}
\phi_{jm}&=x_j+\frac h2\tanh\big(o^{\phi}_{jm}\big), &
\delta_{jm}&=0.02\,h+3h\,\mathrm{sig}\big(o^{\delta}_{jm}\big), \nonumber\\
\pi_j&=\mathrm{sig}\big(o^{\pi}_{j}\big), &
\gamma_{jm}&=0.6\,\pi_j\,\delta_{jm}\big(o^{\mathrm{re}}_{jm}+i\,o^{\mathrm{im}}_{jm}\big),\label{eq:nrc_head}
\end{align}
where $\mathrm{sig}(s)=1/(1+e^{-s})$ and the cell center $x_j=(j+\frac12)h$.
The enrichment of Eq.~\eqref{eq:enrichment} thus has $Np=128$ atoms, and backbone and head have $840+3\times8040$ and 369 parameters, respectively.
{\Cref{tab:hyper} in \Cref{sec:setup} lists these settings together with those of the learned baselines.}

Confining each pole to its own cell limits cancellation in partial-fraction sums, and the lower depth bound keeps the near-pole Riesz flux $\Gamma(\alpha)\delta^{-\alpha}$ of a unit atom bounded (\Cref{lem:atoms}).
At the upper bound, the gap $e^{-\pi\delta/h}$ of \Cref{prop:gap} is about $8\times10^{-5}$, so deeper atoms add little to $\Ih$.
The gate $\pi_j\in(0,1)$ switches the enrichment of a whole cell on or off and leaves smooth regions to the histopolant.
Scaling $\gamma_{jm}$ by $\delta_{jm}$ makes the peak value $|\gamma_{jm}|/\delta_{jm}$ independent of the depth, lowers the Riesz flux to order $\delta^{1-\alpha}$ and gives the amplitude used in \Cref{sec:mechanism}.

\paragraph{Interface fluxes}
The advective flux is the central flux
\begin{equation}
F^{\mathrm{adv}}_{j+1/2}(v)=\tfrac12\,U_{j+1/2}^2,\qquad U_{j+1/2}=(\Dth v)(\xf{j}),\label{eq:nrc_adv}
\end{equation}
with no Riemann solver, limiter or added viscosity, which is admissible because $\Dth v$ is analytic on $\T$, whereas a jump would make the Riesz flux unbounded (\Cref{lem:jump}).
{On the 216 in-distribution test rollouts, the trained NRC needs no numerical dissipation and never blows up, whereas the same central flux on the histopolant alone, with $g_\theta=0$, blows up on one test and two out-of-distribution trajectories (\Cref{sec:coupled}).
In distribution, the learned reconstruction therefore stabilizes the scheme, which we observe but do not prove (\Cref{rem:stability}).}
By linearity of $\Ra$, the Riesz flux at a face splits into the flux of the histopolant and a sub-grid flux,
\begin{align}
-\nu\big(\Ra\Dth v\big)(\xf{j})&=-\nu\big(\Ra\Ih v\big)(\xf{j})+F^{\mathrm{sg}}_{j+1/2}(v),\nonumber\\
F^{\mathrm{sg}}_{j+1/2}(v)&=-\nu\Big(\Ra(I-\Ih\Ph)g_\theta(v)\Big)(\xf{j}).\label{eq:nrc_sg}
\end{align}
The first part has the exact diagonal divergence $\mathcal L$ with symbol $\lambda_k=-\nu|k|^\alpha$ (\Cref{prop:decoders}), and the second, the flux of a field with zero cell averages (\Cref{lem:dipole}), follows in closed form from $\Ra\psi_a=i\Li_{1-\alpha}(\bar ae^{ix})$ (\Cref{lem:atoms}).
Its polylogarithms are summed in single precision, with a relative accuracy of about $10^{-5}$, far below the rollout errors (\ref{app:atoms}).
NRC thus advances the cell averages by the reconstruct-evolve-average semi-discretization $\mathcal Lv+\mathcal N_\theta=\Phi_{\Dth}(v)+\Ph f$ of \Cref{prop:decoders},
\begin{equation}
\frac{\mathrm{d}v}{\mathrm{d}t}=\mathcal Lv+\mathcal N_\theta(v,t),\qquad
\mathcal N_\theta(v,t)=-\frac1h\Dm\Big[F^{\mathrm{adv}}(v)+F^{\mathrm{sg}}(v)\Big]+\Ph f(t).\label{eq:nrc_semi}
\end{equation}

\paragraph{Exponential time integration}
We integrate Eq.~\eqref{eq:nrc_semi} with the scheme ETDRK2 of Cox and Matthews~\cite{cox2002exponential}, which in coarse discrete Fourier variables reads
\begin{align}
\hat{\tilde v}_k&=e^{\Delta t\lambda_k}\,\hat v^n_k+\Delta t\,\varphi_1(\Delta t\lambda_k)\,\widehat{\mathcal N}_{\theta,k}(v^n,t_n),\label{eq:nrc_etd1}\\
\hat v^{n+1}_k&=\hat{\tilde v}_k+\Delta t\,\varphi_2(\Delta t\lambda_k)\Big(\widehat{\mathcal N}_{\theta,k}(\tilde v,t_{n+1})-\widehat{\mathcal N}_{\theta,k}(v^n,t_n)\Big),\label{eq:nrc_etd2}
\end{align}
with $\varphi_1(z)=(e^z-1)/z$, $\varphi_2(z)=(e^z-1-z)/z^2$, $\varphi_1(0)=1$ and $\varphi_2(0)=\frac12$.
The step conserves the discrete mass for every $\theta$ (\Cref{prop:mass}), and the forcing enters through exact cell averages at the stage times, so its quadrature error is a zero-sum term in the defect of Theorem~\ref{thm:defect}.

\begin{algorithm}[t]
\caption{One NRC time step from $v^n$ at $t_n$ to $v^{n+1}$ at $t_{n+1}=t_n+\Delta t$. Hats denote coarse discrete Fourier coefficients, and each call of the explicit right-hand side evaluates the network and all $Np$ atoms at all faces.}
\label{alg:nrc}
\begin{algorithmic}[1]
\Require cell averages $v^n$, time $t_n$, order $\alpha$, diffusivity $\nu$, trained weights $\theta$
\Ensure cell averages $v^{n+1}$
\Function{ExplicitRHS}{$y,t$}\Comment{returns $\mathcal N_\theta(y,t)$ of Eq.~\eqref{eq:nrc_semi}}
\State $(\phi_{jm},\delta_{jm},\gamma_{jm})_{j,m}\gets$ backbone and head applied to $(y,\alpha,\nu)$, Eq.~\eqref{eq:nrc_head}
\State face values, cell averages and face Riesz fluxes of $g=2\Real\sum_{j,m}\gamma_{jm}\psi_{a_{jm}}$ by \Cref{lem:atoms}
\State $U_{j+1/2}\gets(\Ih y)(\xf{j})+g(\xf{j})-(\Ih\Ph g)(\xf{j})$ for all $j$
\State $F^{\mathrm{sg}}_{j+1/2}\gets-\nu\big[(\Ra g)(\xf{j})-(\Ra\Ih\Ph g)(\xf{j})\big]$ for all $j$
\State $\hat F_{j+1/2}\gets\frac12U_{j+1/2}^2+F^{\mathrm{sg}}_{j+1/2}$ for all $j$\Comment{central flux, no Riemann solver}
\State \Return $-\frac1h\Dm\hat F+\Ph f(t)$
\EndFunction
\State $\mathcal N^0\gets$ \Call{ExplicitRHS}{$v^n,t_n$}
\State $\hat{\tilde v}_k\gets e^{\Delta t\lambda_k}\hat v^n_k+\Delta t\,\varphi_1(\Delta t\lambda_k)\,\hat{\mathcal N}^0_k$ for all $k$, with $\lambda_k=-\nu|k|^\alpha$
\State $\mathcal N^1\gets$ \Call{ExplicitRHS}{$\tilde v,t_{n+1}$}
\State $\hat v^{n+1}_k\gets\hat{\tilde v}_k+\Delta t\,\varphi_2(\Delta t\lambda_k)\big(\hat{\mathcal N}^1_k-\hat{\mathcal N}^0_k\big)$ for all $k$
\State \Return $v^{n+1}$\Comment{discrete mass conserved for every $\theta$}
\end{algorithmic}
\end{algorithm}

{Each step of Algorithm~\ref{alg:nrc} evaluates all atoms at all faces twice, so a step of NRC costs more than a classical finite-volume step on the same grid (\ref{app:atoms}), and \Cref{sec:future} outlines how this cost can be reduced.}

\paragraph{Stage A, a priori training on exact face fluxes}
Stage A, the a priori stage in the sense of Melchers et al.~\cite{melchers2023comparison}, fits the 48,192 fine training states, stored every 0.1 time units, with $v=\Ph u$ and exact face fluxes $F^\ast_{j+1/2}=\big(\frac12u^2-\nu\Ra u\big)(\xf{j})$ by minimizing
\begin{align}
\mathcal J_A(\theta)=\bigg\langle&\bigg(\frac{F(\Dth v)(\xf{j})-F^\ast_{j+1/2}}{\sigma_F(\alpha)}\bigg)^2+\bigg(\frac{U_{j+1/2}-u(\xf{j})}{\sigma_U(\alpha)}\bigg)^2\nonumber\\
&+\bigg(\frac{h\,\partial_x(\Dth v-u)(\xf{j})}{\sigma_U(\alpha)}\bigg)^2\bigg\rangle.\label{eq:nrc_lossA}
\end{align}
Here $\langle\cdot\rangle$ averages over faces and snapshots, and $\sigma_F(\alpha)$ and $\sigma_U(\alpha)$, the root-mean-square face-flux and face-value errors of $\Ih v$ on every seventh training snapshot, make each term of order one for $g_\theta=0$.
The flux target is gauge fixed because $F^\ast$ is a physical flux rather than a defect known up to a constant (Corollary~\ref{cor:flux}).
The value and slope terms form an $H^1$-type face loss, a discrete counterpart of the sufficient $W^{1,\infty}$ accuracy of Corollary~\ref{cor:decoder}, whereas an $L^2$ fit would leave the nonlocal flux uncontrolled.
Flux and value terms have different minimizers in general, since the conditional mean of $\frac12u^2$ exceeds half the squared conditional mean of $u$ by half its conditional variance (\Cref{prop:closable}).

\paragraph{Stage B, unrolled training through the scheme}
In a rollout the network sees the drifting states of the scheme, and the optimal closure is a conditional expectation under their distribution (\Cref{prop:closable}).
Stage B therefore trains through the solver~\cite{um2020solver,list2022learned}, unrolling Algorithm~\ref{alg:nrc} for $H$ steps with full backpropagation from exact cell averages at a random order, trajectory and start time $t_{n_0}\ge1$, and minimizes
\begin{equation}
\mathcal J_B(\theta)=\frac1H\sum_{q=1}^{H}\bigg\langle\frac{\sum_{k=0}^{N/2}W_k\,\big|\hat e^{(q)}_k\big|^2}{\sum_{k=0}^{N/2}W_k\,\big|\hat r^{(q)}_k\big|^2}\bigg\rangle.\label{eq:nrc_lossB}
\end{equation}
Here $\hat e^{(q)}$ and $\hat r^{(q)}$ are the real Fourier coefficients of $v^{n_0+q}_\theta-\Ph u(t_{n_0+q})$ and $\Ph u(t_{n_0+q})$, and $W_k=4$ for $1\le k\le4$, else $1$.
The horizon grows as $H=4,8,16,32$ but never exceeds 0.64 time units, so {the rollouts of 40 time units in \Cref{sec:coupled} test a long-time behavior that training never sees}.
Both stages use Adam~\cite{kingma2015adam} with cosine schedules~\cite{loshchilov2017sgdr} and gradient-norm clipping~\cite{pascanu2013difficulty}.

\paragraph{Model selection}
Checkpoints of NRC and of all learned baselines are selected on validation rollouts only.
For all three seeds, the last checkpoint was selected, so the validation score was still decreasing when training stopped.
Including data loading and validation, training takes from 2528 to 2716~seconds per seed on two threads of a shared CPU host.

\subsection{Scope, analysis and organization}\label{sec:scope}

Classical finite-volume schemes fix the reconstruction in advance and meet the limits of \Cref{lem:jump} for jumps and of \Cref{prop:gap} for trigonometric reconstructions.
Learned closures typically regress a correction to the coarse update and leave conservation, the nonlocal operator and the meaning of the learned quantity to the network.
NRC {combines the strengths of both views}, since its learned reconstruction preserves cell averages, its interface fluxes are exact and its learned object is a set of poles and amplitudes{ with a direct interpretation as sub-cell structure}.

By training a decoder on coarse cell averages and inserting it into a conservative scheme, we investigate two groups of questions.
\emph{A) For the reconstruction in isolation}, on frozen cell averages of reference solutions, we raise three questions and address them in \Cref{sec:apriori}.
(A1) Which parts of the defect of \Cref{thm:defect} can be computed exactly, and which must be learned (\Cref{sec:nonlocality})?
(A2) Why should the learned sub-cell degrees of freedom be complex poles rather than more Fourier modes or fixed stencils (\Cref{sec:whypoles})?
(A3) How local can the learned part of the reconstruction be when the Riesz flux itself is global (\Cref{sec:locality})?
\emph{B) For the coupled scheme}, where the decoder feeds the fluxes of a time-stepping scheme over long rollouts, we ask three questions and address them in \Cref{sec:coupled}.
(B1) How much accuracy per degree of freedom does the decoder {provide} over high-resolution finite-volume schemes and black-box learned corrections (\Cref{sec:accuracy} and~\ref{sec:grid})?
(B2) Which structural components carry this accuracy, and what do the learned poles represent (\Cref{sec:ablation} and~\ref{sec:mechanism})?
(B3) 
{How robust is the closure to shifts of the fractional order, the forcing and the diffusivity, and how does its advantage change under grid refinement and in the integer-order limit (\Cref{sec:ood}, \ref{sec:grid} and~\ref{sec:barsinai})?}

Following these research questions, \Cref{sec:apriori} treats nonlocality (A1), complex poles (A2) and locality (A3) on frozen cell averages.
\Cref{sec:coupled} reports the accuracy per degree of freedom of the coupled scheme (B1), the components and poles that carry it (B2), and its {robustness and scope} (B3).
{The last part covers shifts of order, forcing and diffusivity, grid refinement and the integer-order setting of Bar-Sinai et al.~\cite{barsinai2019learning}.}
This lets us assess how far NRC can be read as a \emph{data-adaptive} realization of a classical reconstruct-evolve-average scheme.
{The proofs of \Cref{sec:theory}, the implementation of the atoms, the training and evaluation settings, additional results and supporting classical results are given in Sections~S1 to~S5 of the Supplementary Material.}

\section{The reconstruction in isolation}\label{sec:apriori}

We begin with the decoder on frozen and exact cell averages, where neither error accumulation nor distribution drift obscures the analysis of \Cref{sec:theory}.
Only \Cref{sec:locality} and one band in \Cref{fig:poles_gap} use trained NRC decoders, from three seeds, and all other computations are deterministic.

\subsection{Nonlocality and the cost of local surrogates}\label{sec:nonlocality}

\paragraph{Problem setup}
To test three predictions of \Cref{prop:nearfar} and \Cref{lem:jump}, we first compute the share of $(-\Delta)^{\alpha/2}u(x)$ from $|y-x|>h$ at $x=0$, for $u=e^{-x^2}$ on the real line and $u=\cos x$ on $\T$.
Second, we compare the exact coarse operator $s_0$, the Riesz flux of $\Ih v$ at the faces by item~3 of \Cref{prop:decoders}, with the local surrogate $s_2$ of \Cref{sec:ablation}.
This surrogate is the fractional centerd difference~\cite{ortigueira2006riesz,celik2012crank} truncated after four neighbors on each side, and both act on one smooth state of the fractional Burgers equation with $\nu=0.5$ for $N$ from 16 to 1024.
We record the relative operator error and the one-step defect $\|\dfh\|_\infty/\Delta t$ relative to $\|\partial_t\Ph u\|_\infty$, for forward Euler with a Rusanov flux at Courant number $0.25$.
Third, we evaluate the Riesz flux at the center of the front $u=-\tanh(\sin(x-\pi)/\ell)$ by a fast Fourier transform, for $\ell$ from $4h$ to $h/1024$ with $h=2\pi/64$.

\begin{figure}[!tbp]
\centering
\includegraphics[width=\textwidth]{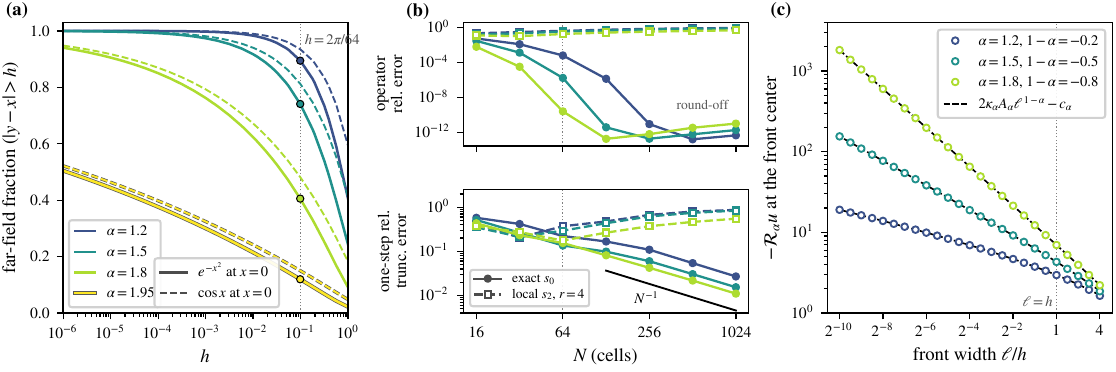}
\caption{Nonlocality of the Riesz flux and the cost of local surrogates, computed without training.
\textbf{(a)} Share of the operator from $|y-x|>h$ at $x=0$ for the Gaussian (solid) and the cosine (dashed), with dots at $h=2\pi/64$.
\textbf{(b)} Relative operator error (top) and relative one-step defect (bottom) of $s_0$ (solid) and $s_2$ (dashed).
\textbf{(c)} Front-center Riesz flux against $\ell/h$, with the law of \Cref{lem:jump} dashed.}
\label{fig:nonlocal}
\end{figure}

\paragraph{Results and observations}
At our resolution the far field carries more than half of the operator only for $\alpha\le1.5$, with Gaussian shares at $h=2\pi/64$ of 0.895, 0.741, 0.406 and 0.120 for $\alpha=1.2$, $1.5$, $1.8$ and $1.95$ (\Cref{fig:nonlocal}(a)). Therefore, neither a stencil of fixed width, which misses the far field, nor a quadrature of distant cell averages, which misses the near field, is accurate at $N=64$. As part (ii) of \Cref{prop:nearfar} predicts for a model error, the relative error of $s_2$ grows from 0.439, 0.356 and 0.189 at $N=64$ to 0.908, 0.803 and 0.499 at $N=1024$ for $\alpha=1.2$, $1.5$ and $1.8$, respectively.
The error of $s_0$, which is at most $6.6\times10^{-4}$ at $N=64$ and below $1.1\times10^{-11}$ for $N\ge256$, comes only from $u-\Ih\Ph u$, with the spectral decay following item~4 of \Cref{thm:consistency}.

{The one-step defect separates $s_0$ from $s_2$ less sharply, because on coarse grids the first-order error of forward Euler with a Rusanov flux masks the model error of $s_2$ (\Cref{fig:nonlocal}(b), bottom).}
The ratio of the $s_2$ defect to the $s_0$ defect lies between 0.49 and 0.86 at $N=16$ for all orders and at $N=32$ for $\alpha\le1.5$, and grows to between 1.14 and 2.17 at $N=64$.
A learned closure on one coarse grid can thus absorb part of a model error as if it were a truncation error.
By part (iii) of \Cref{prop:nearfar}, such a correction must see the whole far field, which a receptive field of fixed cell count cannot do as $h\to0$.

Without the bounded contribution of the opposite front, the flux at the center of a sharpening front follows the law $\ell^{1-\alpha}$ of \Cref{lem:jump}, with fitted slopes $-0.200$, $-0.500$ and $-0.800$ for $\ell\le h/4$ (\Cref{fig:nonlocal}(c)).
A reconstruction that smears a sub-cell front over one cell thus underestimates the face flux by a factor that grows like $(h/\ell)^{\alpha-1}$, and a jump makes the flux infinite.

Together, these results answer (A1) for the fractional part of the defect of \Cref{thm:defect}, since the Riesz flux of the histopolant can be computed exactly and only the sub-cell reconstruction must be learned.
{Even at $\alpha=1.8$, where the near field dominates at $N=64$, the local surrogate keeps a relative error of 0.189, so the exact global evaluation is needed at all three orders.}

\begin{figure}[!tbp]
\centering
\includegraphics[width=\textwidth]{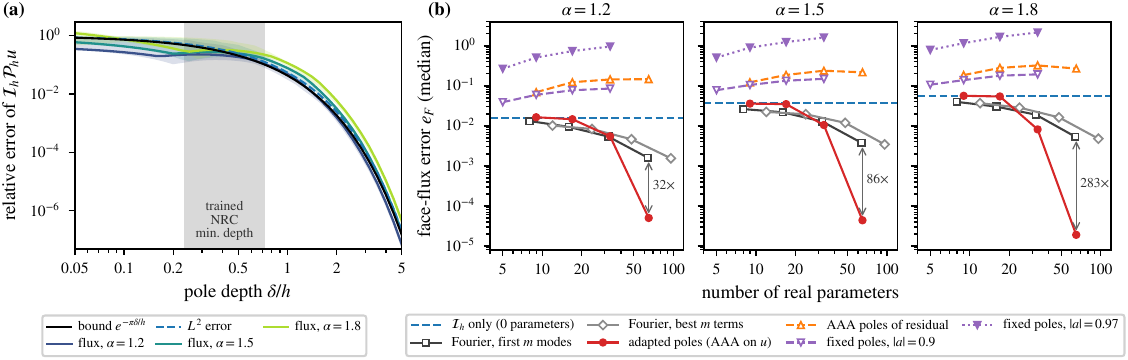}
\caption{One-atom gap of trigonometric reconstructions and oracle headroom of rational enrichment on 64 cells.
\textbf{(a)} Relative error of $\Ih\Ph u$ for one atom against $\delta/h$, with the bound $e^{-\pi\delta/h}$ (black), the $L^2$ error (dashed) and the maximum face-flux error for three orders.
Light bands span 10 pole positions, and the gray band spans the trained minimum pole depths.
\textbf{(b)} Median $e_F$ against the number of real parameters, an oracle bound on the decoder class and not a training result.}
\label{fig:poles_gap}
\end{figure}

\subsection{Why poles}\label{sec:whypoles}

\paragraph{Problem setup}
We test \Cref{prop:gap} on one atom $u=2\Real[\psia]$ with $N=64$, for 41 depths $\delta/h$ from 0.05 to 5 at the pole position $\phi=20.4h$ and at ten more positions $\phi=(20+f)h$ with $f$ from 0.05 to 0.95.
For each depth, we compute the relative $L^2$ error of $\Ih\Ph u$ and its maximal relative face-flux error, and we mark the minimum atom depths of the trained NRC over the 576 test states of \Cref{sec:locality}.
A second, oracle computation asks what a richer reconstruction class gains at a given number of real parameters, with every enrichment $g$ of the decoder~\eqref{eq:nrcdecoder} fitted to the true fine state, which a learned decoder never sees.
It uses 96 states per order, six times on each of 16 held-out forced fractional Burgers trajectories from a separate precursor study with $N=64$ and the diffusivities of \Cref{sec:problem}.
Fourier enrichments add the first $m$ modes above the Nyquist mode or the $m$ dominant modes, with $2m$ and $3m$ parameters and $m=4$, $8$, $16$ or $32$.
Rational enrichments use $K=2$, $4$, $8$ or $16$ poles, placed by AAA~\cite{nakatsukasa2018aaa} on the fine state or on the residual $u-\Ih\Ph u$ with $4K+1$ parameters, or fixed uniformly on the circles $|a|=0.9$ and $|a|=0.97$ with $2K+1$ parameters.
The error $e_F$ is the root mean square of the advective plus fractional face-flux error over that of the exact face flux minus its mean, and we report its median over states.

\paragraph{Results and observations}
The relative $L^2$ error lies above the bound $e^{-\pi\delta/h}$ at all 41 depths and 10 positions, as \Cref{prop:gap} requires (\Cref{fig:poles_gap}(a)).
The trained poles have minimum depths from $0.231h$ to $0.732h$, where every trigonometric reconstruction keeps a relative $L^2$ error of at least 0.100 to 0.484 for a single atom.
The face-flux error, which enters the scheme, follows the bound only loosely, since it ranges from 0.180 to 0.485 over 10 positions for $\alpha=1.5$ and $\delta=0.5h$ and reaches 1.22 at $\alpha=1.8$ and $\delta=0.05h$.
A Fourier enrichment by $m$ further modes has degree $N/2+m$, so the argument of \Cref{prop:gap} gives the bound $e^{-\delta(N/2+m)}$, and $m=32$ on 64 cells only squares the gap, whereas one atom with the correct pole removes it.

Sixteen adapted poles with 65 parameters lower the median $e_F$ of the first 32 Fourier modes with 64 parameters by factors of 32, 86 and 283 for $\alpha=1.2$, $1.5$ and $1.8$ (\Cref{fig:poles_gap}(b)).
At 90\% quantile over states the factors decrease to 16, 16 and 87.
With 2 or 4 adapted poles $e_F$ stays within 6\% of $\Ih$ alone, so the advantage of poles appears only once enough poles are used.
Fixed poles and poles of the residual are less accurate than $\Ih$ alone at every parameter count, by factors of at least 1.9 and 3.3, which for fixed poles is the obstruction of \Cref{sec:atoms}.
This holds even for 16 poles on the circle $|a|=0.97$, whose depth $-\ln|a|\approx0.31h$ is close to the trained depths, so depth alone is not sufficient.

Together, these results answer (A2) and justify letting the network choose both the position and the depth of every pole.
The oracle places up to 16 poles anywhere on the torus, whereas NRC places two in every cell, so the study bounds the class of rational enrichments and not what training reaches.

\subsection{Locality of the learned correction}\label{sec:locality}

\paragraph{Problem setup}
A synthetic test evaluates the Riesz flux of the Legendre polynomials $P_0$, $P_1$ and $P_2$ on one cell by quadrature on the real line.
The in-situ test applies the trained NRC decoders of seeds 0, 1 and 2 to exact cell averages of the first eight test trajectories at eight times from $t=5$ to $45$.
For the corrector $w_\theta=(I-\Ih\Ph)g_\theta(v)$ and the unprojected atoms $g_\theta(v)$, we record the far-field fraction $\max|F_{\mathrm{far}}|/\max|F|$ over faces, where $F_{\mathrm{far}}$ is the part of the Riesz face flux $F$ from $|y-x|\ge m_0h$, for $m_0$ from 2 to 31.

\begin{figure}[!tbp]
\centering
\includegraphics[width=\textwidth]{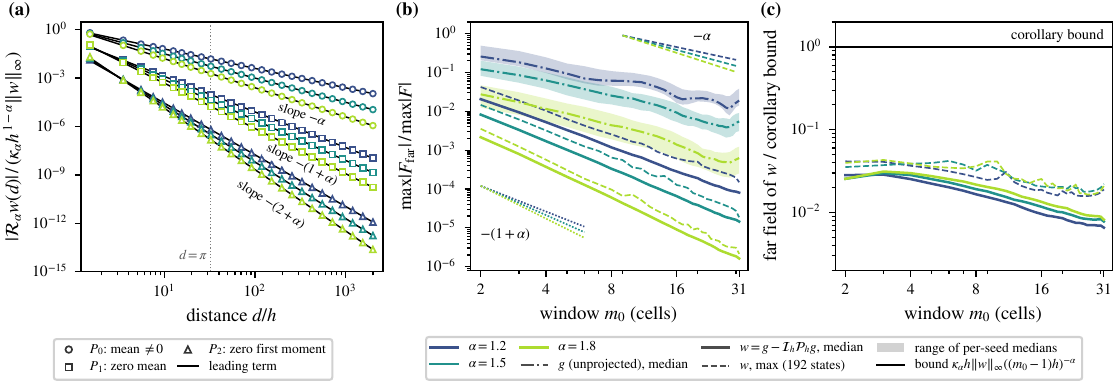}
\caption{Locality of the sub-grid Riesz flux, synthetic and in situ for the trained NRC.
\textbf{(a)} Normalized Riesz flux of single-cell Legendre sources against $d/h$, with leading asymptotes in black.
\textbf{(b)} Far-field fraction against $m_0$ for $w_\theta$ (solid median, dashed maximum) and $g_\theta$ (dash-dot median).
\textbf{(c)} Far field of $w_\theta$ over the bound of \Cref{cor:window}, pooled median (solid) and maximum (dashed).
Panels (b) and (c) pool 192 test states per order from three seeds, with bands spanning per-seed medians.}
\label{fig:dipole}
\end{figure}

\paragraph{Results and observations}
The Legendre sources decay at the monopole, dipole and quadrupole rates of item~3 of \Cref{lem:dipole}, with fitted slopes for $d/h$ from 8 to 200 within $1.1\times10^{-3}$ of $-\alpha$, $-(1+\alpha)$ and $-(2+\alpha)$ (\Cref{fig:dipole}(a)).
In situ, the far field of $w_\theta$ decays with pooled median slopes of $-2.06$, $-2.35$ and $-2.67$ for $m_0\ge4$, slightly slower than the dipole rates $-2.2$, $-2.5$ and $-2.8$, and that of $g_\theta$ close to the monopole rate.
At $m_0=8$ the pooled median far-field fraction of $w_\theta$ is $1.18\times10^{-3}$, $3.31\times10^{-4}$ and $6.02\times10^{-5}$, against 0.0614, 0.0253 and 0.00437 for $g_\theta$ (\Cref{fig:dipole}(b)).
The per-seed medians differ by at most 11\% for $w_\theta$ but by factors of up to 5.0 for $g_\theta$, possibly because the projection removes the cell means and with them the seed-dependent part of the far field.

The window $m_0=8$ is the half-width of the 17-cell receptive field, and beyond it $w_\theta$ contributes less than 0.26\% of the maximum sub-grid face flux in each of the 576 states.
{The current implementation does not yet exploit this locality, since $w_\theta$ on the window depends on all atoms through $\Ih\Ph$, and \Cref{sec:future} discusses how it can reduce the cost of a step.}
The far field of $w_\theta$ stays below 4.3\% of the bound of \Cref{cor:window} for every window and state, so the bound is valid but loose (\Cref{fig:dipole}(c)).
The bound allows a dipole of maximal strength in every cell beyond the window, whereas the measured decay is close to that of a single dipole, which suggests a corrector concentrated near the fronts.

Together, these results address (A3), since the learned part of the nonlocal flux is local in the sense of a dipole tail, while the global part is evaluated exactly.
{All results of this section use frozen and exact cell averages, and \Cref{sec:coupled} turns to the rollout setting, in which the decoder sees the states of the scheme itself.}
\section{The coupled scheme for forced fractional Burgers}\label{sec:coupled}

We now insert the learned decoder into the time-stepping scheme and answer question (B1) over long rollouts, in which reconstruction errors feed back into the states that the network sees.

\subsection{Data, baselines and protocol}\label{sec:setup}

\paragraph{Dynamical system}
All experiments solve the forced fractional Burgers equation~\eqref{eq:pde} on $\T$ with the random forcing~\eqref{eq:forcing}, initial data, orders and diffusivities of Section~\ref{sec:problem}.
The coarse data are the exact averages on $N=64$ cells, stored in single precision every $\Delta t=0.02$, of a Fourier pseudo-spectral ETDRK4 solution~\cite{cox2002exponential,kassam2005fourth} on $N_f=1024$ points with the step $\Delta t/16$.
On this grid the advective Courant number is about $0.3$, and the minimal front width at the training orders lies between 0.38 and 0.48 coarse cells.

The splits of Table~\ref{tab:splits} have disjoint data seeds, with 64 and 16 trajectories per order up to $t=25$ for training and validation and 24 per order up to $t=45$ for testing.
Three shifted splits for Section~\ref{sec:ood} change the order to $\alpha\in\{1.1,1.3,1.7,1.9\}$ with the rule~\eqref{eq:C:nurule} fixed before training, the forcing wavenumbers to $l_i\in\{7,8\}$, or $\nu$ at $\alpha=1.5$ by the factors $0.7$ and $2$.

\paragraph{Closure setup}
NRC is used as specified in Section~\ref{sec:training}, with $p=2$ poles per cell, 128 atoms, a receptive field of 17 cells and 25,329 parameters, and with the seeds 0, 1 and 2 trained on the CPU.

\paragraph{Baselines}
\begin{table}[!tbp]
\centering
\caption{Hyperparameters, parameter counts, seeds and training wall times of NRC and the learned baselines. Wall times are measured on the CPU with two threads per process on a shared host, and ranges are over seeds. All networks receive the four input channels of \ref{app:training:nrc}, and \texttt{c5\_fno} also receives the forcing increment. {Model names encode the learned object and its base: \texttt{bb} marks a black-box face-flux correction on C0 (\texttt{s0}), C0+ (\texttt{weno}), the central flux on $\Ih v$ (\texttt{central}) or C0 with the local operator (\texttt{s2}), \texttt{fluxA} marks the stage-A flux target of NRC, \texttt{fno} a Fourier neural operator, \texttt{point} and \texttt{mult} additive and multiplicative pointwise corrections, \texttt{\_long} four times the iterations, the prefixes \texttt{c1} to \texttt{c5} number the comparison groups, and \texttt{nrc\_nudiv} is the original CPU seed of the diffusivity-diverse NRC of Section~\ref{sec:ood}.}}
\label{tab:hyper}
{\footnotesize\textbf{(a)} Models\par}\smallskip
\resizebox{\textwidth}{!}{%
\begin{tabular}{lllrllccc}
\toprule
Model & Backbone, output & Base flux, operator & Params & Stage A & Stage B & Seeds & Wall A [s] & Wall B [s] \\
\midrule
\texttt{nrc} & CNN, poles & central on $\Dth v$, exact & 25,329 & flux and $H^1$, 1500 & full & 0, 1, 2 & 267 to 305 & 2152 to 2384 \\
\texttt{nrc\_nudiv} & CNN, poles & central on $\Dth v$, exact & 25,329 & flux and $H^1$, 1500 & full & 0 & 315 & 2144 \\
\texttt{bb\_weno\_fluxA} & CNN, face flux & WENO5-Z, Rusanov, exact & 25,041 & flux, 1500 & full & 0, 1, 2 & 54 to 56 & 592 to 724 \\
\texttt{bb\_weno\_fluxA\_long} & CNN, face flux & WENO5-Z, Rusanov, exact & 25,041 & flux, 6000 & long & 0 & 261 & 2416 \\
\texttt{bb\_central\_fluxA} & CNN, face flux & central on $\Ih v$, exact & 25,041 & flux, 1500 & full & 0 & 57 & 500 \\
\texttt{bb\_s0} & CNN, face flux & MUSCL-MC, Godunov, exact & 25,041 & one step, 1500 & full & 0 & 116 & 521 \\
\texttt{bb\_weno} & CNN, face flux & WENO5-Z, Rusanov, exact & 25,041 & one step, 1500 & full & 0 & 105 & 609 \\
\texttt{bb\_central} & CNN, face flux & central on $\Ih v$, exact & 25,041 & one step, 1500 & full & 0 & 105 & 511 \\
\texttt{c2\_bb\_s2} & CNN, face flux & MUSCL-MC, Godunov, local & 25,041 & one step, 1500 & full & 0 & 92 & 405 \\
\texttt{c4\_fno\_s0} & FNO, face flux & MUSCL-MC, Godunov, exact & 40,221 & one step, 1500 & full & 0, 1, 2 & 59 to 97 & 299 to 450 \\
\texttt{c1\_point} & CNN, additive & C0 step & 25,001 & one step, 3000 & none & 0 & 97 & 0 \\
\texttt{c1\_mult} & CNN, ratio & C0 step & 25,001 & ratio, 3000 & none & 0 & 119 & 0 \\
\texttt{c5\_fno} & FNO, increment & none & 40,221 & one step, 1500 & full & 0 & 46 & 164 \\
\texttt{c5\_fno\_long} & FNO, increment & none & 40,221 & one step, 6000 & long & 0 & 217 & 754 \\
\bottomrule
\end{tabular}}

\medskip
{\footnotesize\textbf{(b)} Shared settings\par}\smallskip
{\footnotesize
\setlength{\tabcolsep}{4pt}
\begin{tabular}{@{}lp{0.80\textwidth}@{}}
\toprule
Setting & Value \\
\midrule
CNN backbone & 4 circular convolutions, kernel 5, width 40, GELU, layers 2 to 4 residual, receptive field 17 cells \\
FNO backbone & width 20, 12 Fourier modes, 4 layers, GELU, complex weights counted twice \\
Optimizer & Adam with PyTorch defaults, gradient norm clipped at 1, non-finite steps skipped \\
Stage A & batch of 32 snapshots, learning rate from $10^{-3}$ to $10^{-4}$ on a cosine schedule \\
Stage B & batch of 16 windows, learning rate from $3\times10^{-4}$ to $3\times10^{-5}$ on a cosine schedule \\
Full curriculum & $(H,\text{iterations})=(4,300)$, $(8,300)$, $(16,250)$, $(32,100)$ \\
Short curriculum & $(4,300)$, $(8,150)$, used by the ablations and the receptive-field sweep \\
Long curriculum & stage A 6000 iterations, then $(4,1200)$, $(8,1200)$, $(16,1000)$, $(32,400)$ \\
Model selection & validation score every 150 stage-B iterations and after each phase, lowest kept \\
Arithmetic & double precision, polylogarithm sums in complex64, reference data in single precision \\
\bottomrule
\end{tabular}}
\end{table}
All baselines use the same 64 cells and step, and all except the two local-operator rows and the mean-free surrogate \texttt{c5\_fno} use the exact global fractional operator.
C0 is MUSCL with the monotonized central limiter~\cite{vanleer1977towards,vanleer1979towards}, the Godunov flux~\cite{leveque2002finite} and SSPRK3~\cite{shu1988efficient}, and C0+ is WENO5-Z~\cite{jiang1996efficient,borges2008improved} with the Rusanov flux~\cite{rusanov1962calculation} and ETDRK2, both selected on the validation split (\ref{app:training:classical}).
Two further classical rows isolate parts of NRC, namely the NRC base with $g=0$, which is the central flux on $\Ih v$, and C0 with the local operator~\eqref{eq:C:s2}.
The learned baselines of \ref{app:training:baselines} use the backbone of NRC or a Fourier neural operator (FNO)~\cite{li2021fourier}.
The pointwise corrections \texttt{c1\_point} and \texttt{c1\_mult} act additively or multiplicatively on the C0 step, are trained on one step only and are not conservative.
The face-flux models correct the interface flux of C0{, C0+ or the NRC base} and conserve mass by construction, and \texttt{bb\_weno\_fluxA} replaces the one-step target of \texttt{bb\_weno} by the stage-A flux target of NRC.
{Each CNN face correction is computed by a pointwise head from the backbone features of the two cells adjacent to the face, so it sees 18 cells, and the suffix of its name gives the base, namely \texttt{s0} for C0, \texttt{weno} for C0+, \texttt{central} for the NRC base and \texttt{s2} for C0 with the local operator.
The FNO models, with four Fourier layers of width 20 and 12 modes, either correct the face flux of C0 (\texttt{c4\_fno\_s0}) or act as a surrogate that adds a mean-free learned increment and the exact forcing integral to the cell averages (\texttt{c5\_fno}).
\Cref{tab:hyper} summarizes the architecture, parameter count, training stages and training time of every learned model.
Apart from the pointwise corrections, which are trained on one step only, all learned baselines share the data, the two-stage structure of the training and the validation-only model selection of NRC.
Besides their base scheme and stage-A target, they differ from NRC mainly in what the network outputs, namely a correction of the face flux or of the update instead of a sub-cell reconstruction.}

\paragraph{Evaluation}
Every model is rolled out from the exact cell averages at $t=5$ for 2000 steps, and $e(t_n)=\|V^n-R^n\|_2/\|R^n\|_2$ compares the model cell averages $V^n$ with the reference $R^n$.
The window errors $\errA$ and $\errB$ are the time means of $e$ over $t\in(5,25]$ and over $t\in(25,45]$, beyond the end of the training trajectories.
The pooled median is the median of $\errA$ over the 72 test trajectories, and its seed range in brackets shows the spread of training outcomes rather than a confidence interval.
Other aggregations, such as the pooled mean, move the value of NRC seed 0 from 0.0328\% to at most 0.0349\%.
For the paired comparison we divide the per-trajectory $\errA$ of NRC by that of a baseline for every seed pair, and report the median ratio and the fraction below one as medians over pairs (\ref{app:training:metrics}).

\subsection{Accuracy on the coarse grid}\label{sec:accuracy}

\paragraph{Results and observations}
\begin{table}[!tbp]
\centering
\caption{Accuracy on the in-distribution test split, with 72 trajectories at $\alpha\in\{1.2,1.5,1.8\}$ rolled out on 64 cells from exact cell averages at $t=5$. The errors $\errA$ and $\errB$ average the relative $L^2$ error over $t\in(5,25]$ and over $t\in(25,45]$, beyond the training horizon. Single runs give the pooled median with the 10th and 90th percentiles, and three-seed rows the median over seeds with the seed range. The spectrum error is in units of $10^{-3}$, and the mass drift floor of about $10^{-7}$ comes from the single-precision reference. Of the lower block only \texttt{bb\_weno\_fluxA\_long} matches the training time of NRC, with 2677~s against 2457 to 2651~s.}
\label{tab:main}
\footnotesize
\setlength{\tabcolsep}{3pt}
\resizebox{\textwidth}{!}{%
\begin{tabular}{lrrcccccc}
\toprule
Method & Seeds & Params & $\errA$ & $\errB$ & Front & Spectrum & Mass & Blow- \\
 & & & (\%) & (\%) & (cells) & ($10^{-3}$) & drift & ups \\
\midrule
\multicolumn{9}{l}{\emph{Classical schemes}} \\
C0 & n/a & 0 & 1.50 & 1.42 & 0.577 & 50.9 & $1.01{\times}10^{-7}$ & 0 \\
\hspace{0.5em}{\scriptsize MUSCL-MC, Godunov} &  &  & {\scriptsize(1.12 to 2.02)} & {\scriptsize(1.03 to 1.99)} &  &  &  &  \\
C0+ & n/a & 0 & 0.431 & 0.384 & 0.240 & 15.3 & $1.01{\times}10^{-7}$ & 0 \\
\hspace{0.5em}{\scriptsize WENO5-Z, Rusanov} &  &  & {\scriptsize(0.291 to 0.679)} & {\scriptsize(0.258 to 0.558)} &  &  &  &  \\
NRC base, $g=0$ & n/a & 0 & 3.34 & 2.72 & 0.535 & 175 & $1.01{\times}10^{-7}$ & 1 \\
\hspace{0.5em}{\scriptsize central flux on $\Ih v$} &  &  & {\scriptsize(1.04 to 9.29)} & {\scriptsize(0.821 to 10.2)} &  &  &  &  \\
C0, local operator & n/a & 0 & 39.5 & 44.8 & 6.24 & 173 & $1.01{\times}10^{-7}$ & 0 \\
\hspace{0.5em}{\scriptsize Eq.~\eqref{eq:C:s2}} &  &  & {\scriptsize(8.41 to 87.2)} & {\scriptsize(11.1 to 104)} &  &  &  &  \\
\midrule
\multicolumn{9}{l}{\emph{Learned baselines}} \\
\texttt{c1\_point} & 1 & 25,001 & 4.10 & 9.03 & 1.30 & 16.4 & 0.462 & 0 \\
\hspace{0.5em}{\scriptsize pointwise additive} &  &  & {\scriptsize(2.02 to 7.82)} & {\scriptsize(4.19 to 17.4)} &  &  &  &  \\
\texttt{c1\_mult} & 1 & 25,001 & 13.0 & 27.8 & 2.81 & 59.6 & 4.00 & 0 \\
\hspace{0.5em}{\scriptsize pointwise multiplicative} &  &  & {\scriptsize(5.09 to 22.7)} & {\scriptsize(11.5 to 43.1)} &  &  &  &  \\
\texttt{c2\_bb\_s2} & 1 & 25,041 & 7.45 & 8.57 & 1.69 & 41.4 & $1.01{\times}10^{-7}$ & 0 \\
\hspace{0.5em}{\scriptsize flux CNN, local operator} &  &  & {\scriptsize(1.71 to 13.4)} & {\scriptsize(2.04 to 17.3)} &  &  &  &  \\
\texttt{c4\_fno\_s0} & 3 & 40,221 & 0.675 & 0.679 & 0.332 & 9.01 & $1.01{\times}10^{-7}$ & 0 \\
\hspace{0.5em}{\scriptsize flux FNO on C0} &  &  & {\scriptsize[0.674, 0.737]} & {\scriptsize[0.667, 0.715]} & {\scriptsize[0.319, 0.404]} & {\scriptsize[8.27, 9.93]} &  &  \\
\texttt{c5\_fno} & 1 & 40,221 & 7.26 & 8.07 & 2.42 & 145 & $1.01{\times}10^{-7}$ & 0 \\
\hspace{0.5em}{\scriptsize FNO surrogate} &  &  & {\scriptsize(4.56 to 11.2)} & {\scriptsize(4.62 to 13.1)} &  &  &  &  \\
\texttt{bb\_s0} & 1 & 25,041 & 0.448 & 0.464 & 0.211 & 4.98 & $1.01{\times}10^{-7}$ & 0 \\
\hspace{0.5em}{\scriptsize flux CNN on C0} &  &  & {\scriptsize(0.329 to 0.706)} & {\scriptsize(0.324 to 0.804)} &  &  &  &  \\
\texttt{bb\_weno} & 1 & 25,041 & 0.221 & 0.223 & 0.110 & 3.28 & $1.01{\times}10^{-7}$ & 0 \\
\hspace{0.5em}{\scriptsize flux CNN on C0+} &  &  & {\scriptsize(0.120 to 0.383)} & {\scriptsize(0.116 to 0.363)} &  &  &  &  \\
\texttt{bb\_weno\_fluxA} & 3 & 25,041 & 0.153 & 0.149 & 0.0930 & 3.75 & $1.01{\times}10^{-7}$ & 0 \\
\hspace{0.5em}{\scriptsize on C0+, NRC flux target} &  &  & {\scriptsize[0.153, 0.168]} & {\scriptsize[0.144, 0.168]} & {\scriptsize[0.0905, 0.113]} & {\scriptsize[3.72, 3.83]} &  &  \\
\texttt{bb\_central\_fluxA} & 1 & 25,041 & 0.199 & 0.193 & 0.0675 & 8.09 & $1.01{\times}10^{-7}$ & 0 \\
\hspace{0.5em}{\scriptsize on NRC base, NRC target} &  &  & {\scriptsize(0.112 to 0.344)} & {\scriptsize(0.106 to 0.330)} &  &  &  &  \\
\midrule
\multicolumn{9}{l}{\emph{Proposed method}} \\
NRC & 3 & 25,329 & 0.0308 & 0.0293 & 0.0190 & 0.534 & $1.01{\times}10^{-7}$ & 0 \\
\hspace{0.5em}{\scriptsize rational atoms} &  &  & {\scriptsize[0.0238, 0.0328]} & {\scriptsize[0.0224, 0.0336]} & {\scriptsize[0.00200, 0.0230]} & {\scriptsize[0.533, 0.632]} &  &  \\
\midrule
\multicolumn{9}{l}{\emph{Four times the training iterations, single seed}} \\
\texttt{c5\_fno\_long} & 1 & 40,221 & 1.57 & 1.59 & 0.702 & 25.8 & $1.01{\times}10^{-7}$ & 0 \\
\hspace{0.5em}{\scriptsize 4 times the iterations} &  &  & {\scriptsize(0.773 to 3.11)} & {\scriptsize(0.789 to 3.41)} &  &  &  &  \\
\texttt{bb\_weno\_fluxA\_long} & 1 & 25,041 & 0.0906 & 0.0885 & 0.0480 & 2.47 & $1.01{\times}10^{-7}$ & 0 \\
\hspace{0.5em}{\scriptsize 4 times the iterations} &  &  & {\scriptsize(0.0458 to 0.179)} & {\scriptsize(0.0449 to 0.203)} &  &  &  &  \\
\bottomrule
\end{tabular}}
\end{table}
Table~\ref{tab:main} shows that on the 64-cell grid NRC reaches a pooled median $\errA$ of 0.0308\% ($[0.0238, 0.0328]$) over three seeds.
This value is 14 times below the 0.431\% of C0+ and 49 times below the 1.50\% of C0, two deterministic schemes with the same exact fractional operator.
Per trajectory, the median paired ratio to C0+ is 0.0710 over the three seed pairs, with a 10th to 90th percentile range from 0.0470 to 0.0910.
As Figure~\ref{fig:snapshot} shows, the difference concentrates in the cell holding a front of width $0.47h$, where $\Dth v$ follows the fine solution while $\Ih v$ oscillates over several cells.
From a numerical analysis perspective, the decoder removes the Gibbs-type oscillation of the global reconstruction while keeping every cell average, a role that limiters and nonlinear weights play in MUSCL and WENO.

\begin{figure}[!tbp]
\centering
\includegraphics[width=\textwidth]{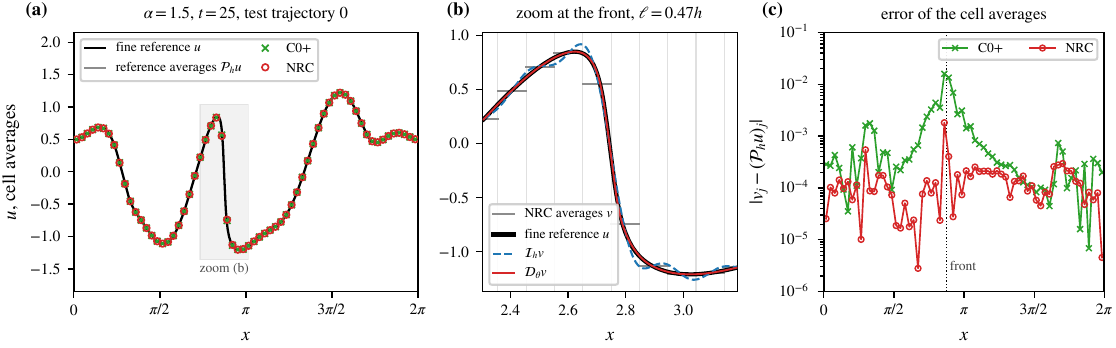}
\caption{Test trajectory 0 at $\alpha=1.5$ and $t=25$, for NRC seed 0 and C0+ started from the exact cell averages at $t=5$.
\textbf{(a)} Reference $u$ and $\Ph u$ with the cell averages of NRC and C0+, which agree to plotting accuracy.
\textbf{(b)} Zoom at the front of width $\ell=0.47h$, where the decoded state $\Dth v$ (red) follows $u$ while the histopolant $\Ih v$ of the same averages (blue dashed) oscillates.
\textbf{(c)} Pointwise error of the cell averages, which peaks in the front cell at $1.82\times10^{-3}$ for NRC and $1.59\times10^{-2}$ for C0+.
Relative errors are 0.0361\% for NRC and 0.388\% for C0+ in the cell averages, and 0.137\% for $\Dth v$ against 2.45\% for $\Ih v$.
The trajectory was chosen from the reference alone, and its $\errA$ ranks 17th of 24 for NRC seed 0 and 13th of 24 for C0+ at this order.}
\label{fig:snapshot}
\end{figure}

\begin{figure}[!tbp]
\centering
\includegraphics[width=\textwidth]{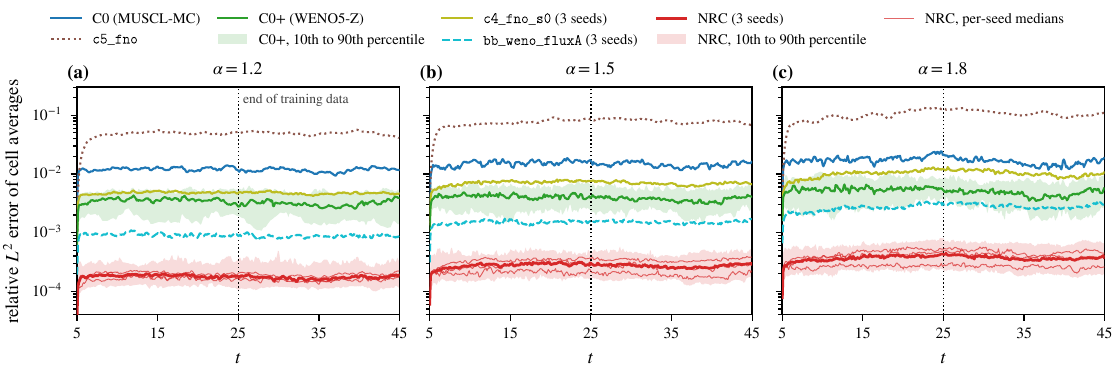}
\caption{Relative $L^2$ error of the 64 cell averages against time on the test split, for $\alpha=1.2$, $1.5$ and $1.8$ in panels \textbf{(a)} to \textbf{(c)}.
Curves are medians over trajectories, pooled over seeds for NRC, \texttt{bb\_weno\_fluxA} and \texttt{c4\_fno\_s0}, and shaded bands give the 10th to 90th percentiles of NRC and C0+.
Thin red lines are the medians of the three NRC seeds, and the dotted line marks the end of the training trajectories.
The 90th percentile of NRC stays below the 10th percentile of C0+ and of \texttt{bb\_weno\_fluxA} at every stored time and for every order.
From $t\in(5,25]$ to $t\in(25,45]$, the time mean of the NRC median changes by factors from 0.94 to 1.02, against 1.05 to 1.13 for \texttt{c5\_fno}.
These differ from the pooled median $\errA$, which averages each trajectory in time before the median.}
\label{fig:errtime}
\end{figure}

Among the learned baselines, only the two models trained on the stage-A flux target of NRC, \texttt{bb\_weno\_fluxA} and \texttt{bb\_central\_fluxA}, stay below 0.2\%.
On the C0+ base this target lowers the error of the face correction from 0.221\% for \texttt{bb\_weno} to 0.153\% ($[0.153, 0.168]$) over three seeds.
NRC remains about five times below \texttt{bb\_weno\_fluxA}, with a median paired ratio of 0.193 over nine seed pairs and a 10th to 90th percentile range from 0.122 to 0.248.
Since \texttt{bb\_central\_fluxA} shares the base, the flux target, the backbone and the curriculum with NRC, its error measures the effect of learning a reconstruction rather than a face correction (Section~\ref{sec:ablation}).
A global receptive field alone does not close the gap, because the face-flux FNO \texttt{c4\_fno\_s0} reaches 0.675\% ($[0.674, 0.737]$), above C0+.

{To check that the gap to \texttt{bb\_weno\_fluxA} does not stem from the longer training of NRC, about 3.7 times that of this baseline, we also trained one seed of \texttt{bb\_weno\_fluxA} with four times the iterations, which matches the training time of NRC.} Its pooled median $\errA$ fell to 0.0906\%, still 2.9 times above NRC, with a median paired ratio of 0.337 over three seed pairs. With the same extension \texttt{c5\_fno} improved from 7.26\% to 1.57\%, so some baselines were limited by their training budget rather than by their structure. 

The paired comparison is uniform in sign, since in every seed pair NRC has the lower $\errA$ on each of the 72 test trajectories against every baseline, including the two longer runs.
{Per order, NRC stays 13.7 to 19.4 times below C0+, with errors from 0.0189\% at $\alpha=1.2$ to 0.0439\% at $\alpha=1.8$ (Table~\ref{tab:peralpha}).}
The front position error of NRC is 0.0190 cells, against 0.0930 for \texttt{bb\_weno\_fluxA} and 0.240 for C0+, and its spectrum error lies about 29 times below that of C0+.
The gain is therefore not confined to the relative $L^2$ error of the cell averages.

Figure~\ref{fig:errtime} shows that the errors settle after a short transient and do not grow beyond the end of the training trajectories at $t=25$.
For NRC the pooled median $\errB$ is 0.0293\% ($[0.0224, 0.0336]$), slightly below its $\errA$, and the error $\mathrm{err}_{200}$ after 200 steps from $t_0=5$, $10$ and $15$ is 0.0280\% ($[0.0206, 0.0295]$).
The pointwise corrections more than double their error in the second window.
Over the 216 test rollouts of its three seeds, NRC has no blow-up, although stage B never unrolled more than 32 steps and Remark~\ref{rem:stability} gives no stability guarantee.
Every flux-form method and \texttt{c5\_fno} keep the mass drift at the single-precision floor, in line with Corollary~\ref{cor:flux} and Proposition~\ref{prop:mass}.
The pointwise corrections drift by 0.462 and 4.00, since nothing in their update ties the sum of the cell averages to the forcing.

{Three rows of \Cref{tab:main} illustrate failure modes that are consistent with our analysis and that the design of NRC avoids, namely the multiplicative pointwise correction \texttt{c1\_mult}, the NRC base with $g=0$ and C0 with the local operator.}
With a single seed each, the multiplicative correction \texttt{c1\_mult} reaches 13.0\%, as opposed to 4.10\% for the additive \texttt{c1\_point} and 1.50\% for C0 without any correction.
This ordering is consistent with Proposition~\ref{prop:hadamard} in \ref{app:ratio}, by which a ratio target has an infinite mean when the base step has a positive density near zero.
Our zero-mean data come close to this condition, since the base step changes sign in every snapshot.
Note that one seed and one architecture do not test the proposition, and part of the error of both pointwise corrections may stem from their missing stage B.
The NRC base with $g=0$ reaches 3.34\% and blows up on one test trajectory at $\alpha=1.5$ after $t=25$.
Its central advective flux adds no numerical dissipation, so {it is the trained reconstruction that stabilizes NRC, which has no blow-up in its 216 test rollouts}.
The C0 scheme with the local operator reaches 39.5\%, and its per-order error falls from 81.3\% at $\alpha=1.2$ to 9.71\% at $\alpha=1.8$.
This trend is consistent with the $O(1)$ model error of Proposition~\ref{prop:nearfar}, since the near-field share $h^{2-\alpha}$ of the fractional flux shrinks as $\alpha$ approaches one, and Section~\ref{sec:nonlocality} finds the same ordering on frozen states{, which is why NRC keeps the exact global operator}.

Together, these results indicate that on 64 cells and for the training coefficients, the decoder lowers $\errA$ about 14 times against WENO5-Z and 2.9 to 5 times against a black-box face correction on the same target{, with a lower error on every test trajectory in every seed pair and no error growth beyond the training horizon}.

\subsection{Which components carry the accuracy}\label{sec:ablation}

\paragraph{Problem setup}
To answer the first part of question (B2), we conduct systematic ablation studies by changing one component of NRC at a time and {keeping} everything else fixed (\ref{app:training:baselines}).
All variants and the reference \texttt{nrc\_short} use three GPU seeds and the short curriculum of \Cref{tab:hyper}, so they are compared with this reference only.
Factors quoted in the text are ratios of medians over seeds, unlike the last column of \Cref{tab:ablation}.
A sweep of the receptive field over 5, 17, 33 and 65 cells compares NRC with \texttt{c2\_bb\_s2}, a black-box face flux on the MUSCL-MC and Godunov base of C0 with the local surrogate of Eq.~\eqref{eq:C:s2}.
The default of 17 cells was fixed before this sweep, and the wider fields use dilated kernels.

\paragraph{Results and observations}
\begin{table}[!tbp]
\centering
\caption{Ablations of NRC under the short protocol, with three GPU-trained seeds per row, a realization separate from the CPU-trained models of Table~\ref{tab:main}. Entries give the median over seeds of the pooled median with the seed range, and the extrapolation column pools the orders $1.1$ and $1.9$. Ratios to the reference, 0.0370\%, are paired by trajectory and seed, with the 10th and 90th percentiles of all 216 ratios. Variants (i) to (vii) are worse on all 216 pairs, the sparsity penalty (viii) on 54.6\% of them, and no trajectory blew up.}
\label{tab:ablation}
\footnotesize
\setlength{\tabcolsep}{3pt}
\begin{tabular}{lrcccc}
\toprule
Variant & Params & $\errA$ (\%) & $\errB$ (\%) & Extrap. (\%) & Ratio to \\
 & & test & test & $\alpha\in\{1.1,1.9\}$ & reference \\
\midrule
NRC reference & 25,329 & 0.0370 & 0.0365 & 0.0433 & 1 \\
\hspace{0.5em}{\scriptsize short protocol, $p=2$} & & {\scriptsize[0.0360, 0.0393]} & {\scriptsize[0.0340, 0.0375]} & {\scriptsize[0.0365, 0.0439]} &  \\
Stage A only & 25,329 & 0.0506 & 0.0438 & 0.0567 & 1.34 \\
\hspace{0.5em}{\scriptsize no solver in the loop} & & {\scriptsize[0.0469, 0.0555]} & {\scriptsize[0.0424, 0.0482]} & {\scriptsize[0.0494, 0.0621]} & {\scriptsize(1.04 to 1.67)} \\
\midrule
(i) local surrogate $s_2$ & 25,329 & 8.42 & 10.7 & 4.77 & 266 \\
\hspace{0.5em}{\scriptsize instead of the exact operator} & & {\scriptsize[8.11, 9.36]} & {\scriptsize[9.45, 11.4]} & {\scriptsize[4.67, 5.58]} & {\scriptsize(48.0 to 605)} \\
(ii) no projection $\Ph$ & 25,329 & 0.365 & 0.360 & 0.304 & 9.73 \\
\hspace{0.5em}{\scriptsize $\mathcal{D}v=\Ih v+g_\theta$} & & {\scriptsize[0.308, 0.398]} & {\scriptsize[0.327, 0.405]} & {\scriptsize[0.296, 0.337]} & {\scriptsize(7.40 to 13.6)} \\
(iii) no sub-grid Riesz flux & 25,329 & 0.463 & 0.436 & 0.363 & 12.4 \\
\hspace{0.5em}{\scriptsize decoder in advective flux only} & & {\scriptsize[0.433, 0.490]} & {\scriptsize[0.374, 0.465]} & {\scriptsize[0.341, 0.449]} & {\scriptsize(9.42 to 17.0)} \\
(iv) black-box face flux & 25,041 & 0.306 & 0.288 & 0.204 & 8.21 \\
\hspace{0.5em}{\scriptsize same base, target, backbone} & & {\scriptsize[0.262, 0.308]} & {\scriptsize[0.253, 0.295]} & {\scriptsize[0.202, 0.242]} & {\scriptsize(5.75 to 10.5)} \\
(v) sub-grid field head & 25,616 & 0.240 & 0.249 & 0.245 & 6.23 \\
\hspace{0.5em}{\scriptsize 16 nodal values per cell} & & {\scriptsize[0.184, 0.250]} & {\scriptsize[0.187, 0.263]} & {\scriptsize[0.231, 0.273]} & {\scriptsize(3.99 to 8.94)} \\
(vi) one pole per cell & 25,165 & 0.0556 & 0.0547 & 0.0542 & 1.53 \\
\hspace{0.5em}{\scriptsize $p=1$} & & {\scriptsize[0.0539, 0.0597]} & {\scriptsize[0.0493, 0.0612]} & {\scriptsize[0.0438, 0.0560]} & {\scriptsize(1.26 to 1.91)} \\
(vii) network sees only $v$ & 25,329 & 0.0859 & 0.0809 & 0.106 & 2.43 \\
\hspace{0.5em}{\scriptsize parameter channels zeroed} & & {\scriptsize[0.0854, 0.0860]} & {\scriptsize[0.0771, 0.0851]} & {\scriptsize[0.100, 0.115]} & {\scriptsize(1.79 to 3.39)} \\
(viii) sparsity penalty & 25,329 & 0.0350 & 0.0329 & 0.0444 & 0.946 \\
\hspace{0.5em}{\scriptsize weight 0.02 in stage A} & & {\scriptsize[0.0304, 0.0543]} & {\scriptsize[0.0293, 0.0464]} & {\scriptsize[0.0337, 0.0622]} & {\scriptsize(0.833 to 1.54)} \\
\bottomrule
\end{tabular}
\end{table}

Under the short protocol, removing or replacing any of the components (i) to (v) raises the median $\errA$ over three seeds at least {6.5 times}, with a higher error on all 216 seed-matched pairs.
The local surrogate $s_2$ in place of the exact operator {raises the median $\errA$ 228 times}, consistent with the $O(1)$ model error of \Cref{prop:nearfar}.
Its median seed-matched ratio falls from 495 at $\alpha=1.2$ to 57.5 at $\alpha=1.8$, as do the far-field shares of \Cref{sec:nonlocality}.
Dropping the projection {raises the median $\errA$ 9.9 times}, and the atoms then alter the cell averages and have the far field of a monopole (\Cref{sec:locality}).
By \Cref{thm:rea} and \Cref{prop:decoders}, the projection instead confines the learned part to functions with zero cell averages, the class of the error $(I-\Ih\Ph)u$ of the histopolant.
Removing the sub-grid Riesz flux {raises the median $\errA$ 12.5 times}, although the decoder still enters the advective flux.
In the bounds of \Cref{thm:consistency} the fractional part of the residual dominates the advective part by a factor $h^{1-\alpha}$ (\Cref{cor:decoder}), and this variant leaves the face-flux error of $\Ih v$ in it.

{Without rational atoms, the median $\errA$ rises 8.3 times for the black-box face flux (iv) on the base and flux target of NRC and 6.5 times for the sub-grid field head (v).}
\Cref{prop:gap} traces the difficulty to the Fourier tail $e^{-n\delta}$ of a pole near the real axis, which one atom encodes in its depth, while (iv) must learn the resulting flux error of the histopolant.
{The field of (v) is expressive enough, since on 1024 nodes it approximates an atom of the minimum trained depth $0.231h$ to a relative error of $9\times10^{-6}$.
Its larger error therefore points to the learning problem rather than to the representation, since its head must predict 16 coupled nodal values per cell, whereas the head of NRC predicts nine outputs per cell and each atom encodes the steepness of a transition in its depth alone.}
{Training through the solver in stage B lowers the median error by a further factor of 1.37 relative to the a priori stage-A state and improves 94.4\% of the pairs, a smaller but consistent gain.}
By \Cref{prop:closable}, stage A approximates the conditional mean under the measure of exact states, whereas in a rollout the network sees the states of the scheme, which only stage B trains on.
The remaining choices have less significant effects, since one pole per cell (vi) and zeroing the parameter channels (vii) {raise the median $\errA$ by factors of} 1.5 and 2.3.
The first {increase} is consistent with the groups of deep poles per front in \Cref{sec:mechanism}, and the second with \Cref{prop:closable}, by which the coarser input $V=v$ cannot lower the irreducible error.
Meanwhile, the sparsity penalty (viii) has no measurable effect, with per-seed ratios from 0.931 to 1.45.

\begin{figure}[!tbp]
\centering
\includegraphics[width=\textwidth]{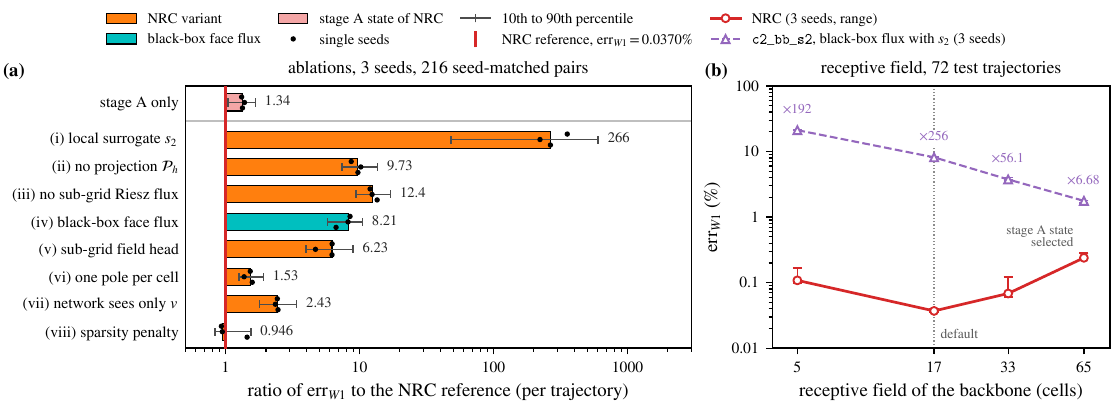}
\caption{Ablations and receptive-field sweep under the short protocol, for three GPU-trained seeds on the 72 test trajectories.
\textbf{(a)} Seed-matched ratio of $\errA$ to the NRC reference, with bars at the median of the per-seed medians, seed dots and whiskers for the 10th to 90th percentile of the 216 pairs.
\textbf{(b)} Median $\errA$ against the receptive field for NRC and \texttt{c2\_bb\_s2}, with seed ranges as bars and the seed-matched ratio of \texttt{c2\_bb\_s2} to NRC as numbers.
At 65 cells validation selected the stage-A state of NRC in all seeds, and the default of 17 cells was not chosen from this sweep.}
\label{fig:rfsweep}
\end{figure}

As \Cref{fig:rfsweep}(b) shows for three seeds, the median $\errA$ of NRC is lower at the default of 17 cells than at 5, 33 and 65 cells, with 0.0370\% as opposed to 0.108\%, 0.0683\% and 0.239\%.
In contrast, \texttt{c2\_bb\_s2} improves monotonically from 21.2\% to 1.77\%, but at 65 cells, which cover the torus, it stays 42.7 times above the NRC reference in the seed-matched ratio.
With the local surrogate, part~(iii) of \Cref{prop:nearfar} places the missing flux in the far field, which a wider window sees but must still learn from cell averages.
With the exact operator, \Cref{lem:dipole} and \Cref{cor:window} make the learned flux local, with less than 0.26\% beyond the half-width of the default window (\Cref{sec:locality}).
Consistent with this locality, even the 5-cell NRC with 9,009 parameters stays 16 times below \texttt{c2\_bb\_s2} at 65 cells.
{The default of 17 cells, fixed before this sweep, is thus also the best of the four fields for NRC, consistent with a learned flux that is negligible beyond the half-width of this field, and \ref{app:add:ablation} gives the sweep for every order.}

\subsection{Where the poles go: a mechanistic check}\label{sec:mechanism}

\paragraph{Problem setup}
To answer the second part of (B2), we compare the poles of the three NRC models of \Cref{tab:main}, decoded a priori from exact cell averages, with the fronts of the reference solution on the test split.
A front is a local minimum of $u_x$ below $-\frac14\max|u_x|$, with the tanh-equivalent width $\ell=\Delta U/(2|u_x|)$, $\Delta U=u(x-2h)-u(x+2h)$, and it is sub-cell if $\ell<h$.
Over 24 test trajectories per order and 41 times from $t=5$ to 45, this gives 6281 sub-cell fronts for each seed.
For each front we record the distance to the nearest active pole, with atoms active if $\delta<0.75h$ under the depth criterion or if $|\gamma|/\delta\geq0.05$ under the amplitude criterion.
Every pole stays in its own cell, so the distance $h/2$ mostly checks that the cell of the front holds an active pole, and $h/4$ is the informative threshold.

\begin{figure}[!tbp]
\centering
\includegraphics[width=\textwidth]{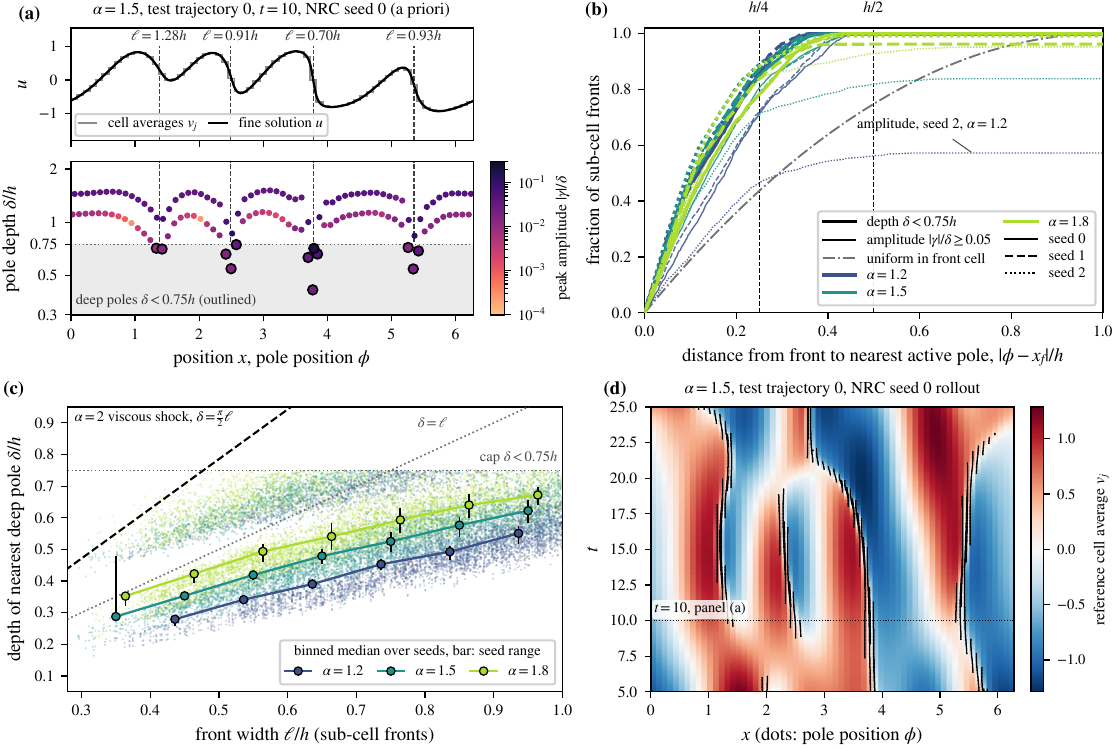}
\caption{Learned poles of the three NRC models of \Cref{tab:main} and the sub-cell fronts of the reference solution on the test split.
\textbf{(a)} Test trajectory 0 at $\alpha=1.5$ and $t=10$ above the 128 atoms of seed 0, decoded a priori, colored by $|\gamma|/\delta$ and outlined when deep.
\textbf{(b)} Cumulative distribution of the distance from each sub-cell front to the nearest active pole, thick for the depth and thin for the amplitude criterion, gray for the uniform reference.
\textbf{(c)} Depth of the nearest deep pole against the front width for 18,841 fronts with binned medians, the relation $\delta=\frac{\pi}{2}\ell$ of \Cref{prop:poles} dashed and the cap $\delta=0.75h$ dotted.
\textbf{(d)} Reference cell averages of the trajectory of (a) for $t$ from 5 to 25, with the deep atoms of the rollout of seed 0 as black dots.}
\label{fig:poles}
\end{figure}

\paragraph{Results and observations}
Under the depth criterion, every seed and order places a pole within $h/4$ of 78.2 to 89.0\% of the sub-cell fronts, and within $h/2$ of 99.4\% when pooled (\Cref{fig:poles}(b)).
A point placed uniformly in the cell of the front gives only 43.7 and 74.8\%.
Deep poles are rare, a median of 7 to 10 of 128 atoms per snapshot, and the same number placed uniformly on the circle would come within $h/4$ of only 5.1 to 7.5\% of fronts.
{The depth criterion is the informative one, for two reasons.
First, the amplitude criterion also counts shallow background atoms: in \Cref{fig:poles}(a), 19 of the 21 atoms above the amplitude threshold lie at depths near $1.1h$ and $1.5h$, and by \Cref{prop:gap} an atom at $1.1h$ is within $e^{-1.1\pi}\approx0.03$ of a trigonometric polynomial of degree $N/2$ in the relative $L^2$ norm, so these atoms act as smooth corrections.
Second, the amplitude scale of the atoms differs between seeds (\ref{app:add:interp}), so a fixed amplitude threshold selects different atoms in different seeds.
Under the amplitude criterion the fraction within $h/4$ therefore ranges from 46.5 to 86.4\%, and for seed 2 at $\alpha=1.2$, 347 of the 1650 sub-cell fronts have no atom above the threshold.}

The depth of the nearest deep pole grows with the front width and with the order and lies below the viscous shock relation $\delta=\frac\pi2\ell$ of \Cref{prop:poles}, although the cap $\delta<0.75h$ shapes \Cref{fig:poles}(c).
The pooled median of $\delta/\ell$ is 0.587, 0.709 and 0.821 at $\alpha=1.2$, $1.5$ and $1.8$, and the amplitude criterion, which has no cap, still gives 0.948 to 0.996, below $\pi/2\approx1.571$.
\Cref{prop:poles} fixes $\delta/\ell$ only at $\alpha=2$, so this growth with $\alpha$ is consistent with it but does not test it.
Each front of \Cref{fig:poles}(a) carries a group of 2 to 4 deep atoms within one cell, consistent with the pole clusters that AAA fits use in Remark~\ref{rem:branch}, which one snapshot cannot test.
In the rollout of seed 0 the deep atoms follow the fronts over $t\in[5,25]$ (\Cref{fig:poles}(d)), {and} 66.3\% of its 1940 deep atoms lie within $h/2$ of a sub-cell front of $\Ih\Ph u$.
The others sit mostly at fronts wider than a cell or below the detection threshold.
Together, for three seeds on exact cell averages, these results answer the second part of (B2), since the deep atoms act as sub-cell shape functions that mark the fronts and deepen as the fronts widen.

\subsection{Out-of-distribution order, forcing and diffusivity}\label{sec:ood}

\paragraph{Problem setup}
We now ask {how robust the closure is} once the data leave the training distribution, which is part of question (B3).
The shifted splits of \Cref{sec:setup} move the network inputs in different ways, because each training order comes with a single diffusivity.
The order shift keeps the pair $(\alpha,\nu)$ on the curve of Eq.~\eqref{eq:C:nurule} through the three training pairs, and the forcing shift changes only the distribution of the input states.
The coefficient shift multiplies $\nu$ at fixed $\alpha=1.5$ and moves the inputs off this curve, in a direction along which the network has never seen $\nu$ vary.
\Cref{tab:ood} and \Cref{fig:ood} write the diffusivity as $D$ and give each shift as a factor of $D_{\mathrm{train}}(\alpha)$.

{The $D$-diverse training set adds the factors $0.8$ and $1.5$ at every training order, and we train three GPU seeds on it with validation-only selection.
Fresh splits, generated with the unused data seed 88, serve as its test data.}
These hold 24 trajectories per training order at the factors 0.6 and 2.5, which extrapolate beyond the range $[0.8,1.5]$, as well as at 1.25, which interpolates.
The three new columns of \Cref{tab:ood} are paired, as all nine fresh files share one set of initial conditions and forcings.
At the factor 0.6 the sharpest checked fronts are 0.14 to 0.22 cells wide, and doubling the reference resolution changes their cell averages by at most a relative $9.1\times10^{-6}$.

\begin{figure}[!tbp]
\centering
\includegraphics[width=\textwidth]{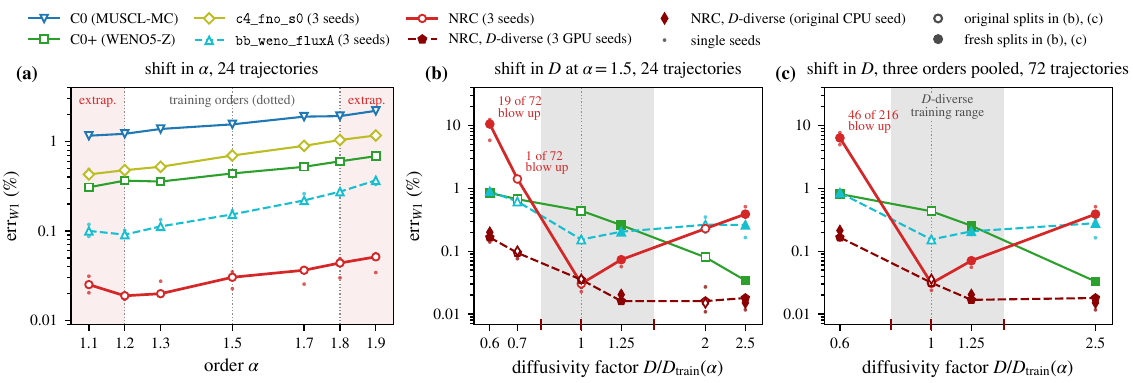}
\caption{Errors outside the training distribution in the order and the diffusivity, given as the median of $\errA$ over the trajectories of each split.
Lines join the median over seeds of the per-seed medians, and small dots show the individual seeds of the three-seed models.
\textbf{(a)} Error against $\alpha$ with 24 trajectories per point, where the dotted training orders come from the test split and shaded bands mark extrapolation.
\textbf{(b)} Error against $D/D_{\mathrm{train}}(\alpha)$ at $\alpha=1.5$, with open markers for the original splits and filled markers for the fresh splits.
\textbf{(c)} The fresh splits pooled over $\alpha\in\{1.2,1.5,1.8\}$ with 72 trajectories per point, together with the test error at factor 1.
The gray band is the training range of the $D$-diverse model and dark red diamonds mark its original CPU seed.
Numbers next to the NRC points count the trajectories with a non-finite state, summed over the three seeds.}
\label{fig:ood}
\end{figure}

\paragraph{Results and observations}
\begin{table}[!tbp]
\centering
\caption{Errors $\errA$ in percent outside the training distribution, with the rollout and the metric of Table~\ref{tab:main} and the trajectories per seed in the header. Three-seed rows give the median over seeds with the seed range, and the original splits change the forcing wavenumbers to $\{7,8\}$ or the coefficient to $0.7D$ or $2D$ at $\alpha=1.5$. The fresh splits pool the three training orders and share one data seed, so their columns are paired. Daggers mark extrapolation, and the double dagger marks $1.25D$, inside the range of the $D$-diverse rows only. {The fresh splits played no role in the design of the $D$-diverse set.}}
\label{tab:ood}
\footnotesize
\setlength{\tabcolsep}{3pt}
\resizebox{\textwidth}{!}{%
\begin{tabular}{lrcccccccccr}
\toprule
 & & Test & \multicolumn{2}{c}{Order $\alpha$ (48)} & \multicolumn{3}{c}{Original splits, $\alpha=1.5$ (24)} & \multicolumn{3}{c}{Fresh splits, three orders (72)} & Blow- \\
\cmidrule(lr){4-5}\cmidrule(lr){6-8}\cmidrule(lr){9-11}
Method & Seeds & (72) & $1.3,1.7$ & $1.1,1.9^{\dagger}$ & $l\in\{7,8\}^{\dagger}$ & $0.7D^{\dagger}$ & $2D^{\dagger}$ & $0.6D^{\dagger}$ & $1.25D^{\ddagger}$ & $2.5D^{\dagger}$ & ups \\
\midrule
\multicolumn{12}{l}{\emph{Classical schemes}} \\
C0 & n/a & 1.50 & 1.54 & 1.38 & 3.49 & 1.88 & 0.850 & 1.88 & 1.16 & 0.692 & 0 \\
\hspace{0.5em}{\scriptsize MUSCL-MC, Godunov} & &  &  &  &  &  &  &  &  &  & \\
C0+ & n/a & 0.431 & 0.420 & 0.417 & 0.571 & 0.670 & 0.0799 & 0.811 & 0.254 & 0.0330 & 0 \\
\hspace{0.5em}{\scriptsize WENO5-Z, Rusanov} & &  &  &  &  &  &  &  &  &  & \\
\midrule
\multicolumn{12}{l}{\emph{Learned baselines, trained at $D\times1$}} \\
\texttt{c4\_fno\_s0} & 3 & 0.675 & 0.625 & 0.582 & 3.05 & 1.20 & 0.612 & n/a & n/a & n/a & 0 \\
\hspace{0.5em}{\scriptsize flux FNO on C0} & & {\scriptsize[0.674, 0.737]} & {\scriptsize[0.607, 0.651]} & {\scriptsize[0.552, 0.697]} & {\scriptsize seed 0} & {\scriptsize seed 0} & {\scriptsize seed 0} &  &  &  & \\
\texttt{bb\_weno\_fluxA} & 3 & 0.153 & 0.149 & 0.189 & 0.309 & 0.617 & 0.263 & 0.849 & 0.207 & 0.279 & 0 \\
\hspace{0.5em}{\scriptsize on C0+, NRC flux target} & & {\scriptsize[0.153, 0.168]} & {\scriptsize[0.136, 0.168]} & {\scriptsize[0.159, 0.224]} & {\scriptsize[0.277, 0.311]} & {\scriptsize[0.614, 0.630]} & {\scriptsize[0.199, 0.349]} & {\scriptsize[0.817, 0.859]} & {\scriptsize[0.184, 0.248]} & {\scriptsize[0.163, 0.306]} & \\
\midrule
\multicolumn{12}{l}{\emph{Proposed method, trained at $D\times1$}} \\
NRC & 3 & 0.0308 & 0.0312 & 0.0390 & 0.0540 & 1.41 & 0.228 & 6.33 & 0.0704 & 0.388 & 47 \\
\hspace{0.5em}{\scriptsize rational atoms} & & {\scriptsize[0.0238, 0.0328]} & {\scriptsize[0.0224, 0.0319]} & {\scriptsize[0.0272, 0.0399]} & {\scriptsize[0.0422, 0.0555]} & {\scriptsize[1.27, 1.52]} & {\scriptsize[0.214, 0.285]} & {\scriptsize[4.94, 7.68]} & {\scriptsize[0.0554, 0.0741]} & {\scriptsize[0.348, 0.509]} & \\
\midrule
\multicolumn{12}{l}{\emph{NRC trained at $D\times\{0.8,1,1.5\}$}} \\
NRC, $D$-diverse & 3 & 0.0314 & 0.0302 & 0.0374 & 0.0513 & 0.0939 & 0.0160 & 0.165 & 0.0168 & 0.0180 & 0 \\
\hspace{0.5em}{\scriptsize GPU-trained, new seeds} & & {\scriptsize[0.0304, 0.0386]} & {\scriptsize[0.0282, 0.0376]} & {\scriptsize[0.0350, 0.0389]} & {\scriptsize[0.0418, 0.0641]} & {\scriptsize[0.0759, 0.105]} & {\scriptsize[0.0108, 0.0270]} & {\scriptsize[0.151, 0.225]} & {\scriptsize[0.0160, 0.0221]} & {\scriptsize[0.0116, 0.0344]} & \\
NRC, $D$-diverse & 1 & 0.0358 & 0.0338 & 0.0346 & 0.0538 & 0.100 & 0.0149 & 0.212 & 0.0204 & 0.0150 & 0 \\
\hspace{0.5em}{\scriptsize CPU-trained, original seed} & &  &  &  &  &  &  &  &  &  & \\
\bottomrule
\end{tabular}}
\end{table}
Along the training curve, interpolation in the order reproduces the test error and extrapolation raises it by a factor of about 1.3, for three seeds.
NRC reaches 0.0312\% at $\alpha=1.3$ and $1.7$ and 0.0390\% at $\alpha=1.1$ and $1.9$, against 0.0308\% on the test split and about 0.42\% for WENO5-Z.
In every seed pair NRC has the lower error on each of the 96 order-shifted trajectories against every baseline of \Cref{tab:main}.
All methods have larger errors at $\alpha=1.9$ than at $1.1$ (\Cref{fig:ood}(a)), and NRC stays 4.0 to 7.2 times below \texttt{bb\_weno\_fluxA} at these two orders.
The cost of extrapolation shows at $\alpha=1.1$, where the error of NRC rises above its value at $\alpha=1.2$ while that of WENO5-Z falls.
New forcing wavenumbers $l_i\in\{7,8\}$ at $\alpha=1.5$ raise the error relative to the test split by a factor of 1.8 for NRC and 1.3 for WENO5-Z.
NRC still stays about 11 times below WENO5-Z, with 0.0540\% as opposed to 0.571\%.

{A shift of the diffusivity at fixed order is harder, and NRC trained on one diffusivity per order becomes less accurate than WENO5-Z at both shifted diffusivities of the original split in every seed.}
At $\nu\times0.7$ NRC reaches 1.41\%, or 2.1 times the 0.670\% of WENO5-Z, and one of the 72 seed trajectories blows up.
At $\nu\times2$ it reaches 0.228\%, or 2.9 times the 0.0799\% of WENO5-Z, although the sharpest checked fronts there are 1.07 cells wide.
Relative to the test error at this order, WENO5-Z gains a factor of 5.5 at $\nu\times2$, whereas NRC loses a factor of 7.5.
The learned baseline \texttt{bb\_weno\_fluxA}, also trained at one diffusivity, degrades in the same way, to 0.92 and 3.3 times WENO5-Z.

\Cref{prop:closable} and \Cref{rem:distribution} explain why a shift of $\nu$ at fixed $\alpha$ differs from a shift along the training curve.
The network approximates $m^\star(V)=\EE[Y\mid V]$ under the occupation measure $\mu$ of the training trajectories, and $\mu$ holds one diffusivity per order.
The loss therefore never constrains the response to $\nu$ at fixed $\alpha$, and a coefficient shift moves both the inputs off the support of $\mu$ and the optimal closure itself.
A single-seed diagnostic separates the two effects by feeding the network the training value of $\nu$ while the exact operator keeps the true one.
For seed 0 this lowers $\errA$ from 0.214\% to 0.0414\% at $\nu\times2$, below WENO5-Z, so most of that failure comes from the unseen input value (Table~\ref{tab:extra}).
At $\nu\times0.7$ the error falls only to 0.779\%, and the sharpest checked fronts are 0.25 cells wide against 0.38 to 0.48 cells in training, consistent with a shift of the target.

{\Cref{rem:distribution} suggests a remedy, namely a training measure that varies $\nu$ at every order, which the $D$-diverse model realizes with the factors 0.8, 1 and 1.5.
Since the original splits at $0.7D$ and $2D$ informed the choice of these factors, they cannot serve as an unbiased test of the remedy, and we therefore test it on the fresh splits with the factors 0.6, 1.25 and 2.5, which played no role in its design (\Cref{fig:ood}(c)).
On these splits, plain NRC stays 3.6 times below WENO5-Z at $\nu\times1.25$, close to its training data, but at the extrapolated factors 0.6 and 2.5 it loses to WENO5-Z on every trajectory in every seed, by factors of 7.8 and 11.8, with 46 of its 216 rollouts blowing up at $\nu\times0.6$.}

The three GPU seeds of the $D$-diverse model stay below WENO5-Z at every fresh factor in the median over seeds, without a blow-up on any shifted split.
Their errors are 0.20, 0.066 and 0.55 times those of WENO5-Z at the factors 0.6, 1.25 and 2.5{, so training on three diffusivities turns the worst failure of plain NRC, an error of 6.33\% with 46 blow-ups at $\nu\times0.6$, into a fivefold advantage over WENO5-Z}.
{At $\nu\times2.5$, seeds 0 and 2 win on at least 22 of the 24 trajectories at every order, whereas seed 1 is above WENO5-Z, with a pooled median of 0.0344\% against 0.0330\% and median paired ratios of 1.50 and 1.57 at $\alpha=1.2$ and $1.8$.}

{This robustness costs little accuracy in distribution.
On the test split the $D$-diverse model reaches 0.0314\% against 0.0308\% for NRC, and their median paired ratio of 1.20 over nine seed pairs lies within the seed-to-seed spread of NRC itself, whose seeds differ by ratios from 0.729 to 1.37.
The comparison mixes GPU-trained and CPU-trained models, but the GPU-trained NRC models of \Cref{sec:grid} reach 0.0307\% at $N=64$, against 0.0308\% for the CPU-trained ones.}

Together, these results indicate that the closure inherits the distribution of its training data, as \Cref{prop:closable} predicts, tolerating shifts along the training relation between $\alpha$ and $\nu$ but not across it.
Widening the distribution in $\nu${, as \Cref{rem:distribution} suggests, restores the advantage of NRC over WENO5-Z on fresh data in the median over three seeds, without a blow-up on any shifted split}.

\subsection{Grid refinement and accuracy per degree of freedom}\label{sec:grid}

\paragraph{Problem setup}
All comparisons so far use 64 cells, and we now vary $N$ to measure the accuracy per degree of freedom {defined in \Cref{sec:intro}} that the decoder {provides} (B1) and its change under refinement (B3).
The data are stored as exact cell averages at $N=32$, 64, 128, 256 and 512, and the 64-cell data equal the original files bit for bit.
We train NRC separately at $N=32$, 64 and 128, with three GPU seeds each, validation-only selection and all hyperparameters tuned at $N=64$.
The receptive field thus stays at 17 cells, from 53\% down to 13\% of the domain, and the input $\ln(\Delta t\,\nu/h^\alpha)$ keeps its 64-cell normalization.
At $N=64$ these GPU-trained models reach 0.0307\%, against 0.0308\% for the separately trained CPU models of \Cref{tab:main}.
C0, C0+ and a pseudo-spectral ETDRK4 solver run at all five resolutions with $\Delta t=0.02\min(1,64/N)$, except C0 at $N=512$, which needs $\Delta t=0.00125$.
NRC keeps $\Delta t=0.02$ at every $N$, and the pseudo-spectral solver starts from point values instead of cell averages.
{We report the pooled median $\errA$ over the 72 test trajectories in two ways, as the own-grid error, which compares the $N$ cell averages of a method with the reference averaged onto the same $N$ cells, and after averaging both onto 32 cells.}
For matched comparisons we average a classical solution on $2N$ or $4N$ cells onto the $N$ cells of NRC and form paired per-trajectory ratios.

\begin{figure}[!tbp]
\centering
\includegraphics[width=\textwidth]{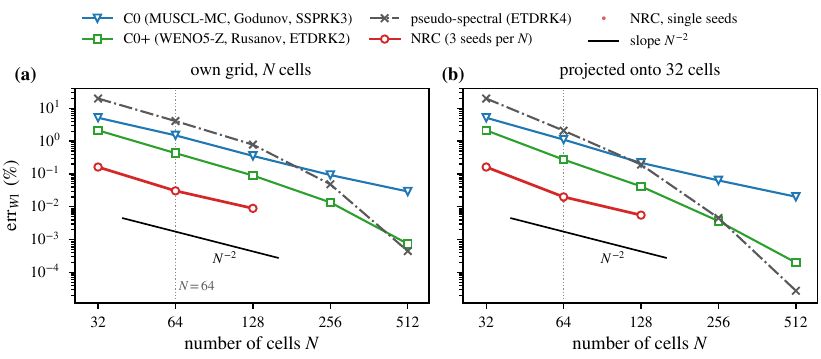}
\caption{Pooled test error against the number of cells, given as the median of $\errA$ over the 72 test trajectories of the three orders.
\textbf{(a)} Error of the $N$ cell averages against the reference $N$-cell averages, and \textbf{(b)} error after both are averaged onto 32 cells.
NRC is trained separately at each $N$ with the hyperparameters tuned at $N=64$, and small dots show its three GPU-trained seeds.
The classical schemes shrink the time step for $N\geq128$ while NRC keeps $\Delta t=0.02$, and the thin line has slope $N^{-2}$.
The dotted line marks the benchmark grid with $N=64$, and no rollout of any method shown in the figure blows up.}
\label{fig:grid}
\end{figure}

\paragraph{Results and observations}
At equal $N$ the gain of NRC over WENO5-Z hardly depends on the resolution, for three GPU seeds per $N$.
NRC has the lower $\errA$ on every test trajectory in every seed, with median paired ratios of 0.077, 0.071 and 0.084 at $N=32$, 64 and 128.
{Its own-grid errors are 0.163\%, 0.0307\% and 0.00900\%, against 2.12\%, 0.431\% and 0.0899\% for WENO5-Z, and no NRC rollout blows up at any resolution.}

{When the WENO5-Z solutions on finer grids are averaged onto the 64 cells of NRC, NRC on 64 cells is more accurate than WENO5-Z on 128 cells and less accurate than WENO5-Z on 256 cells, with 0.0307\% against 0.0703\% and 0.00885\%.}
In every seed, NRC has the lower error on at least 94\% of the trajectories against WENO5-Z on 128 cells, and on none against 256 cells.
To match NRC, the MUSCL scheme C0 needs between 256 and 512 cells and the pseudo-spectral solver between 128 and 256 cells.
{NRC on 64 cells thus reaches an accuracy that WENO5-Z attains only on a grid with two to four times as many cells, which quantifies the accuracy per degree of freedom of NRC.}

This {refinement} factor {is larger on the coarse grids with 32 and 64 cells than on 128 cells}, because the benefit of sub-cell poles depends on the width of the fronts relative to $h$.
Against WENO5-Z on twice as many cells, the median paired ratio is 0.654 at $N=32$, 0.401 at $N=64$ and 0.862 at $N=128$.
NRC at $N=128$ is thus about as accurate as WENO5-Z at $N=256$, winning on 58\% of the trajectories, and on the 32-cell metric of \Cref{fig:grid}(b) WENO5-Z is lower, with 0.00368\% against 0.00559\%.
The sharpest fronts are 0.38 to 0.48 cells wide at $N=64$, hence 0.76 to 0.96 cells at $N=128$ and 1.5 to 1.9 cells at $N=256$.
Once a front spans more than a cell, WENO5-Z approaches its asymptotic regime, and its observed order rises from about 2.3 to 2.72 and 4.19.
By \Cref{prop:gap}, a trigonometric reconstruction of an atom of depth $\delta$ keeps a relative error of at least $e^{-\pi\delta/h}$.
Using the viscous relation $\delta=\pi\ell/2$ of \Cref{prop:poles} as a guide, this floor is 0.09 to 0.15 at $N=64$ but below $10^{-3}$ at $N=256$.
{The learned poles therefore improve the accuracy most where fronts are narrower than a cell, as at $N=32$ and 64, and less where fronts span several cells and the grid itself resolves them.}
The margin at $N=64$, where the hyperparameters were tuned, exceeds those at $N=32$ and 128, so the smaller margins may partly reflect untuned settings.

{Under refinement, NRC converges with observed orders of 2.41 between 32 and 64 cells and 1.77 between 64 and 128 cells, against 2.30 and 2.26 for WENO5-Z.
NRC thus keeps the lower error at every $N$ we tested, but its lower order between 64 and 128 cells shrinks its margin there.
Between 64 and 128 cells the order of NRC decreases with the fractional order, from 1.95 at $\alpha=1.2$ to 1.46 at $\alpha=1.8$, which is consistent with the loss of $\alpha$ orders through the Riesz flux in \Cref{cor:decoder}.
\Cref{thm:consistency} does not predict the values themselves, since it bounds the residual of a decoder family with a known accuracy $O(h^q)$, whereas each $N$ here has its own trained network.}

{Two remarks help to read the classical curves of \Cref{fig:grid} at the finest grids.
Own-grid errors are not directly comparable across $N$, because the reference averages on $N$ cells resolve more small-scale content as $N$ grows, so the matched comparisons above average all solutions onto a common grid.}
At $N=256$ and 512 the pseudo-spectral solver belongs to the scheme family of the reference, so its error measures the distance to a nearly identical discretization.
These comparisons count cells {, the degrees of freedom of the scheme, rather than operations}.

Together, these results answer the resolution part of questions (B1) and (B3) for three seeds per grid on the pooled test split.
{At 64 cells, the decoder reduces the number of cells needed to reach a given accuracy by a factor between two and four relative to WENO5-Z, and this factor falls to about two at 128 cells as the grid itself resolves the fronts.}

\subsection{Relation to learned discretizations for integer-order Burgers}\label{sec:barsinai}

\paragraph{Problem setup}
The learned discretization closest to NRC is the data-driven finite-volume scheme of Bar-Sinai et al.~\cite{barsinai2019learning}, which we test on their own integer-order Burgers setting as part of question (B3).
Their network outputs solution-dependent coefficients of a six-point stencil for the face value and slope inside a conservative update, and we ported their published Burgers configuration to PyTorch.
Block (A) of Table~\ref{tab:barsinai} uses their equation with $\alpha=2$ and $\nu=0.01$, and block (B) replaces the Laplacian by $-\nu(-\Delta)^{\alpha/2}$ with our pairs $(\alpha,\nu)$ but keeps forcing, data and split.
The name Bar-Sinai thus denotes two models, their coefficient model in block (A) and, in block (B) and the last line, their order-one face stencil inside our ETDRK2 step.
Every learned entry is a single run with one seed, except the three-seed NRC entry of the last line.

\paragraph{Results and observations}
\begin{table}[!tbp]
\centering
\caption{Comparison with the learned discretization of Bar-Sinai et al.~\cite{barsinai2019learning}, where all learned entries except the three-seed NRC entry of the last line are single runs. Blocks (A) and (B) use their forcing and 8 training and 4 test trajectories on 64 cells, with the mean absolute error over $t\in(0,15]$ against a 1024-point reference. Block (A) is viscous Burgers with viscosity 0.01, and block (B) replaces the Laplacian by $-D(-\Delta)^{\alpha/2}$. Unless stated, rows use ETDRK2 at $\Delta t=0.02$ and 2000 training steps, and the coefficient model trains on the time derivative. The last line gives $\errA$ on the test split of Table~\ref{tab:main}.}
\label{tab:barsinai}
\footnotesize
\setlength{\tabcolsep}{4pt}
\begin{tabular}{lccccc}
\toprule
\multicolumn{6}{l}{\emph{(A) Integer order, $\alpha=2$, mean absolute error over 4 test trajectories in units of $10^{-3}$}} \\
\midrule
Method & Mean & Median & Min & Max & Ratio to \\
 & & & & & Bar-Sinai \\
\midrule
FV1 & 83.8 & 84.3 & 69.3 & 97.1 & 36.2 \\
\hspace{0.5em}{\scriptsize Rusanov, three-point Laplacian} & & & & & \\
WENO5-Z & 6.04 & 6.20 & 4.35 & 7.39 & 2.61 \\
\hspace{0.5em}{\scriptsize Rusanov, three-point Laplacian} & & & & & \\
Bar-Sinai & 2.32 & 2.25 & 1.71 & 3.06 & 1 \\
\hspace{0.5em}{\scriptsize coefficient model, $\Delta t=0.001$} & & & & & \\
NRC, short schedule & 5.40 & 5.11 & 4.55 & 6.86 & 2.33 \\
\hspace{0.5em}{\scriptsize 400 flux-matching, 120 unrolled steps} & & & & & \\
NRC, flux matching & 3.35 & 3.36 & 2.63 & 4.06 & 1.45 \\
\hspace{0.5em}{\scriptsize 2000 flux-matching steps} & & & & & \\
NRC, matched clock & 6.23 & 6.41 & 3.01 & 9.08 & 2.69 \\
\hspace{0.5em}{\scriptsize plus 120 unrolled steps, $\Delta t=0.001$} & & & & & \\
\midrule
\multicolumn{6}{l}{\emph{(B) Fractional variants, mean absolute error over 4 test trajectories in units of $10^{-3}$}} \\
\midrule
Order and diffusivity & WENO5-Z & Bar-Sinai & NRC & \multicolumn{2}{c}{Ratio to NRC} \\
 & {\scriptsize exact operator} & {\scriptsize face stencil} & {\scriptsize flux matching} & {\scriptsize WENO5-Z} & {\scriptsize Bar-Sinai} \\
\midrule
$\alpha=1.2$, $D=0.3$ & 0.827 & 1.73 & 0.137 & 6.06 & 12.7 \\
$\alpha=1.5$, $D=0.12$ & 1.42 & 4.22 & 0.192 & 7.42 & 22.0 \\
$\alpha=1.8$, $D=0.05$ & 2.15 & 8.17 & 0.171 & 12.6 & 47.8 \\
$\alpha=2.0$, $D=0.01$ & 7.34 & 25.3 & 3.35 & 2.19 & 7.54 \\
\midrule
Test split, $\errA$ (\%) & 0.431 & 0.467 & 0.0308 & 14.0 & 15.2 \\
\hspace{0.5em}{\scriptsize 72 trajectories, $\alpha\in\{1.2,1.5,1.8\}$} & & & {\scriptsize[0.0238, 0.0328]} & & \\
\bottomrule
\end{tabular}
\end{table}
On their integer-order setting, the coefficient model of Bar-Sinai et al.\ is more accurate than every NRC variant we tried, with a mean absolute error of $2.32\times10^{-3}$ against $3.35\times10^{-3}$ for the closest NRC run.
This NRC run keeps its error at the step $\Delta t=0.001$ of the coefficient model, so the factor of 1.45 is not a time-step effect.
At $\alpha=2$ the multiplier of $\Ra$ becomes $ik$, so the dissipative flux $-\nu u_x$ is local and the exact operator has no far field to supply.
The exact nonlocal flux and the sub-grid Riesz flux, whose replacement or removal {raises the median $\errA$ of NRC} 228 and 12.5 {times} for $\alpha<2$ (Section~\ref{sec:ablation}), then have nothing to contribute.
What remains is a local reconstruction of the face value and slope, for which a learned stencil with a polynomial accuracy constraint is the natural model.
With single runs we cannot exclude that training also contributes to the gap.

In the fractional variants of block (B), NRC is between 6.06 and 12.6 times below WENO5-Z with the exact operator for $\alpha<2$, but only 2.19 times below it at $\alpha=2$.
The face stencil with the same operator lies 2.1 to 3.8 times above WENO5-Z at every order, and it reaches 0.467\% against 0.431\% for WENO5-Z on the test split of Table~\ref{tab:main}.
At $\alpha=2$ it is 10.9 times above the coefficient model on identical data, so block (B) compares reconstructions inside one shared step and not NRC with the published method.
Together, these single-seed comparisons locate the benefit of NRC in the nonlocal flux, whereas for integer-order Burgers a learned stencil of the type of Bar-Sinai et al.\ was the more accurate choice in our tests.


\section{Conclusions and future directions}\label{sec:conclusions}

{In this work, we developed NRC, a decoder-only learned closure for the forced space-fractional Burgers equation that reconstructs sub-cell fronts with learned rational atoms, keeps every cell average exact and evaluates the advective and nonlocal Riesz fluxes of the reconstruction in closed form.
On coarse grids, where the fronts are narrower than a cell, NRC is an order of magnitude more accurate than a WENO5-Z scheme with the same exact fractional operator, it conserves mass for every value of its weights, and its learned poles mark the sub-cell fronts.
We summarize the novelty and contributions of this work, its key insights and the main directions for future research below.}

\subsection{\texorpdfstring{{Novelty and contributions}}{Novelty and contributions}}\label{sec:novelty}
{The novelty of this work lies in treating a learned closure for a fractional conservation law as a reconstruction operator with exact nonlocal fluxes, rather than as a correction of the coarse update, which places the learning problem exactly where the theory locates the error.
This idea yields the four contributions stated in \Cref{sec:intro}, and \Cref{sec:insights} collects the evidence for each of them.}
{
\begin{itemize}
\item \emph{An exact target for learned closures.} The exact defect identity places the closure problem at the cell faces, where the Riesz flux can be evaluated exactly, so only the sub-cell reconstruction must be learned, whereas local surrogates of the operator carry a model error of order one and reconstructions with jumps have unbounded Riesz fluxes.
\item \emph{A decoder-only closure with structural guarantees.} NRC learns only the map from cell averages to a sub-cell reconstruction, keeps every cell average exactly, conserves mass for every value of the network weights and treats the stiff global operator exactly, so the network never has to supply the far field.
\item \emph{Rational atoms with closed-form fluxes.} Learned complex poles give sub-cell shape functions whose values, cell averages and Riesz fluxes have closed forms, and the resolution gap and the dipole tail explain why they are efficient degrees of freedom and why the learned part stays local although the operator is global.
\item \emph{An accurate and interpretable learned reconstruction.} On coarse grids NRC reaches an error an order of magnitude below that of WENO5-Z with the same exact operator, and it keeps this advantage at equal resolution up to 128 cells and under the tested shifts of the order, the forcing and, with training on several diffusivities, the diffusivity.
Its learned poles can be checked directly against the sub-cell fronts of the solution.
\end{itemize}}

\subsection{Summary of key insights}\label{sec:insights}

\paragraph{Insights from the analysis}
The Riesz flux of the histopolant can be computed exactly, so only the sub-cell reconstruction must be learned, whereas a local surrogate of the operator carries a model error of order one.

No trigonometric reconstruction resolves an atom of depth $\delta$ to better than $e^{-\pi\delta/h}$, whereas one atom with the correct pole removes this gap.
In the oracle study, 16 adapted poles lowered the median face-flux error by factors from 32 to 283 against Fourier modes with nearly the same parameter count, 65 against 64, whereas fixed poles did not help.

Because the projected correction has zero cell averages, its Riesz flux decays like that of a dipole, and beyond the half-width of the receptive field it carries less than 0.26\% of the maximal sub-grid face flux.

\paragraph{Insights from the coupled scheme}
On 64 cells NRC reaches a pooled median error of 0.0308\% over three seeds, 14 times below WENO5-Z and 2.9 to 5 times below a black-box face correction trained on the same flux target.
It has the lower error on every test and order-shifted trajectory in every seed pair, its error does not grow beyond the training horizon, and it is {more accurate than WENO5-Z on twice as many cells at 32 and 64 cells and about as accurate at 128 cells}.

Under a short protocol, removing or replacing the exact operator, the sub-grid Riesz flux, the projection or the rational atoms raises the median error over three seeds by factors from 6.5 to 228.
For every seed and order, a deep pole lies within a quarter cell of 78 to 89\% of the sub-cell fronts, against 43.7\% for a uniformly placed point, and deeper poles mark wider fronts.
We therefore interpret the deep atoms as data-adaptive sub-cell shape functions.

{NRC keeps its advantage under shifts of the fractional order and of the forcing.
Shifts of the diffusivity at fixed order move the network inputs off the training data and the optimal closure itself, as the conditional expectation view of Proposition~\ref{prop:closable} predicts, and with training on three diffusivities per order NRC stays below WENO5-Z at all five tested factors from 0.6 to 2.5 in the median over seeds.
Under refinement, NRC stays more accurate than WENO5-Z at equal resolution, and its advantage over WENO5-Z on finer grids shrinks as the grid itself resolves the fronts.}

\subsection{\texorpdfstring{ {Limitations and future directions}}{Limitations and future directions}}\label{sec:future}
{Three points delimit the present study and indicate the most important next steps.}
{
\begin{itemize}
\item \emph{Computational cost.} In timings on an NVIDIA GeForce RTX 5090 GPU (\ref{app:training:metrics}), NRC with 64 and with 128 cells lies on the Pareto front of the finite-volume and learned schemes we tested, and NRC on 128 cells reaches about the accuracy of WENO5-Z on 256 cells in less wall time, although on this periodic benchmark a pseudo-spectral solver is faster still.
The present implementation evaluates all atoms at all faces twice per step, and windowed sub-grid fluxes with a local projection (\Cref{cor:window}), reuse of the poles across the two stages of the exponential integrator and fused GPU kernels could lower this cost further.
\item \emph{Coefficient-aware training.} A closure trained on one diffusivity per order does not transfer to diffusivity shifts at fixed order.
Training on several diffusivities restores the advantage of NRC in the median over seeds, and training data that vary all coefficients independently, tested on several fresh data sets, would make this robustness systematic.
\item \emph{Bounded domains and several dimensions.} All results concern one periodic and one-dimensional benchmark with constant coefficients, and our atoms and exponential step rely on periodicity.
Bounded domains, where no fast spectral solver applies and finite-volume schemes are the natural competitors, and multi-dimensional fronts, whose point poles would become singularities along curves or surfaces, are the natural next settings for NRC, and both require new atoms with closed-form fluxes.
\end{itemize}}

{More broadly, our results suggest that a small network becomes an accurate, conservative and interpretable part of a finite-volume scheme when it learns only what the theory cannot supply, here the sub-cell reconstruction, and everything else is computed exactly.}
Ultimately, we hope this study helps to connect data-driven closures with classical reconstruction theory for nonlocal conservation laws.

\section*{CRediT authorship contribution statement}
\noindent \textbf{Zeyuan Song:} Conceptualization, Methodology, Software, Formal analysis, Investigation, Writing, original draft.

\noindent \textbf{Zheyu Jiang:} Conceptualization, Supervision, Funding acquisition, Writing, review and editing.

\section*{Declaration of competing interest}
The authors declare that they have no known competing financial interests or personal relationships that could have appeared to influence the work reported in this paper.

\section*{Data availability}
The code and data that support the findings of this study will be made available in a public repository upon publication.

\section*{Declaration of generative AI and AI-assisted technologies in the writing process}
During the preparation of this manuscript the authors used Claude (Anthropic) in order to check the grammar and polish the manuscript.
After using this tool, the authors reviewed and edited the content as needed and take full responsibility for the content of the manuscript.

\section*{Acknowledgments}
We acknowledge financial support from the U.S. National Science Foundation Award No.~2442806 and the startup fund of the College of Engineering, Architecture, and Technology at Oklahoma State University.

\bibliographystyle{elsarticle-num-nodash}
\bibliography{references_all}

@article{lischke2020what,
  author  = {Lischke, Anna and Pang, Guofei and Gulian, Mamikon and Song, Fangying and Glusa, Christian and Zheng, Xiaoning and Mao, Zhiping and Cai, Wei and Meerschaert, Mark M. and Ainsworth, Mark and Karniadakis, George Em},
  title   = {What is the fractional {L}aplacian? {A} comparative review with new results},
  journal = {Journal of Computational Physics},
  volume  = {404},
  pages   = {109009},
  year    = {2020},
  doi     = {10.1016/j.jcp.2019.109009}
}

@article{kwasnicki2017ten,
  author  = {Kwa{\'s}nicki, Mateusz},
  title   = {Ten equivalent definitions of the fractional {L}aplace operator},
  journal = {Fractional Calculus and Applied Analysis},
  volume  = {20},
  number  = {1},
  pages   = {7-51},
  year    = {2017},
  doi     = {10.1515/fca-2017-0002}
}

@article{huang2014numerical,
  author  = {Huang, Yanghong and Oberman, Adam},
  title   = {Numerical methods for the fractional {L}aplacian: {A} finite difference-quadrature approach},
  journal = {SIAM Journal on Numerical Analysis},
  volume  = {52},
  number  = {6},
  pages   = {3056-3084},
  year    = {2014},
  doi     = {10.1137/140954040}
}

@article{bonito2018numerical,
  author  = {Bonito, Andrea and Borthagaray, Juan Pablo and Nochetto, Ricardo H. and Ot{\'a}rola, Enrique and Salgado, Abner J.},
  title   = {Numerical methods for fractional diffusion},
  journal = {Computing and Visualization in Science},
  volume  = {19},
  number  = {5-6},
  pages   = {19-46},
  year    = {2018},
  doi     = {10.1007/s00791-018-0289-y}
}

@article{ortigueira2006riesz,
  author  = {Ortigueira, Manuel Duarte},
  title   = {Riesz potential operators and inverses via fractional centred derivatives},
  journal = {International Journal of Mathematics and Mathematical Sciences},
  volume  = {2006},
  pages   = {48391},
  year    = {2006},
  doi     = {10.1155/IJMMS/2006/48391}
}

@article{celik2012crank,
  author  = {{\c{C}}elik, Cem and Duman, Melda},
  title   = {{C}rank-{N}icolson method for the fractional diffusion equation with the {R}iesz fractional derivative},
  journal = {Journal of Computational Physics},
  volume  = {231},
  number  = {4},
  pages   = {1743-1750},
  year    = {2012},
  doi     = {10.1016/j.jcp.2011.11.008}
}

@article{biler1998fractal,
  author  = {Biler, Piotr and Funaki, Tadahisa and Woyczynski, Wojbor A.},
  title   = {Fractal {B}urgers equations},
  journal = {Journal of Differential Equations},
  volume  = {148},
  number  = {1},
  pages   = {9-46},
  year    = {1998},
  doi     = {10.1006/jdeq.1998.3458}
}

@article{kiselev2008blowup,
  author  = {Kiselev, Alexander and Nazarov, Fedor and Shterenberg, Roman},
  title   = {Blow up and regularity for fractal {B}urgers equation},
  journal = {Dynamics of Partial Differential Equations},
  volume  = {5},
  number  = {3},
  pages   = {211-240},
  year    = {2008},
  doi     = {10.4310/DPDE.2008.v5.n3.a2}
}

@article{alibaud2007occurrence,
  author  = {Alibaud, Natha{\"e}l and Droniou, J{\'e}r{\^o}me and Vovelle, Julien},
  title   = {Occurrence and non-appearance of shocks in fractal {B}urgers equations},
  journal = {Journal of Hyperbolic Differential Equations},
  volume  = {4},
  number  = {3},
  pages   = {479-499},
  year    = {2007},
  doi     = {10.1142/S0219891607001227}
}

@article{droniou2003global,
  author  = {Droniou, J{\'e}r{\^o}me and Gallou{\"e}t, Thierry and Vovelle, Julien},
  title   = {Global solution and smoothing effect for a non-local regularization of a hyperbolic equation},
  journal = {Journal of Evolution Equations},
  volume  = {3},
  number  = {3},
  pages   = {499-521},
  year    = {2003},
  doi     = {10.1007/s00028-003-0503-1}
}

@article{droniou2010numerical,
  author  = {Droniou, J{\'e}r{\^o}me},
  title   = {A numerical method for fractal conservation laws},
  journal = {Mathematics of Computation},
  volume  = {79},
  number  = {269},
  pages   = {95-124},
  year    = {2010},
  doi     = {10.1090/S0025-5718-09-02293-5}
}

@article{sulem1983tracing,
  author  = {Sulem, Catherine and Sulem, Pierre-Louis and Frisch, H{\'e}l{\`e}ne},
  title   = {Tracing complex singularities with spectral methods},
  journal = {Journal of Computational Physics},
  volume  = {50},
  number  = {1},
  pages   = {138-161},
  year    = {1983},
  doi     = {10.1016/0021-9991(83)90045-1}
}

@book{leveque2002finite,
  author    = {LeVeque, Randall J.},
  title     = {Finite Volume Methods for Hyperbolic Problems},
  publisher = {Cambridge University Press},
  address   = {Cambridge},
  year      = {2002},
  doi       = {10.1017/CBO9780511791253}
}

@article{harten1983high,
  author  = {Harten, Ami},
  title   = {High resolution schemes for hyperbolic conservation laws},
  journal = {Journal of Computational Physics},
  volume  = {49},
  number  = {3},
  pages   = {357-393},
  year    = {1983},
  doi     = {10.1016/0021-9991(83)90136-5}
}

@article{vanleer1977towards,
  author  = {van Leer, Bram},
  title   = {Towards the ultimate conservative difference scheme. {IV}. {A} new approach to numerical convection},
  journal = {Journal of Computational Physics},
  volume  = {23},
  number  = {3},
  pages   = {276-299},
  year    = {1977},
  doi     = {10.1016/0021-9991(77)90095-X}
}

@article{vanleer1979towards,
  author  = {van Leer, Bram},
  title   = {Towards the ultimate conservative difference scheme. {V}. {A} second-order sequel to {G}odunov's method},
  journal = {Journal of Computational Physics},
  volume  = {32},
  number  = {1},
  pages   = {101-136},
  year    = {1979},
  doi     = {10.1016/0021-9991(79)90145-1}
}

@article{jiang1996efficient,
  author  = {Jiang, Guang-Shan and Shu, Chi-Wang},
  title   = {Efficient implementation of weighted {ENO} schemes},
  journal = {Journal of Computational Physics},
  volume  = {126},
  number  = {1},
  pages   = {202-228},
  year    = {1996},
  doi     = {10.1006/jcph.1996.0130}
}

@article{borges2008improved,
  author  = {Borges, Rafael and Carmona, Monique and Costa, Bruno and Don, Wai Sun},
  title   = {An improved weighted essentially non-oscillatory scheme for hyperbolic conservation laws},
  journal = {Journal of Computational Physics},
  volume  = {227},
  number  = {6},
  pages   = {3191-3211},
  year    = {2008},
  doi     = {10.1016/j.jcp.2007.11.038}
}

@article{shu2009high,
  author  = {Shu, Chi-Wang},
  title   = {High order weighted essentially nonoscillatory schemes for convection dominated problems},
  journal = {SIAM Review},
  volume  = {51},
  number  = {1},
  pages   = {82-126},
  year    = {2009},
  doi     = {10.1137/070679065}
}

@article{shu1988efficient,
  author  = {Shu, Chi-Wang and Osher, Stanley},
  title   = {Efficient implementation of essentially non-oscillatory shock-capturing schemes},
  journal = {Journal of Computational Physics},
  volume  = {77},
  number  = {2},
  pages   = {439-471},
  year    = {1988},
  doi     = {10.1016/0021-9991(88)90177-5}
}

@article{rusanov1962calculation,
  author  = {Rusanov, V. V.},
  title   = {The calculation of the interaction of non-stationary shock waves and obstacles},
  journal = {USSR Computational Mathematics and Mathematical Physics},
  volume  = {1},
  number  = {2},
  pages   = {304-320},
  year    = {1962},
  doi     = {10.1016/0041-5553(62)90062-9}
}

@article{cox2002exponential,
  author  = {Cox, S. M. and Matthews, P. C.},
  title   = {Exponential time differencing for stiff systems},
  journal = {Journal of Computational Physics},
  volume  = {176},
  number  = {2},
  pages   = {430-455},
  year    = {2002},
  doi     = {10.1006/jcph.2002.6995}
}

@article{kassam2005fourth,
  author  = {Kassam, Aly-Khan and Trefethen, Lloyd N.},
  title   = {Fourth-order time-stepping for stiff {PDE}s},
  journal = {SIAM Journal on Scientific Computing},
  volume  = {26},
  number  = {4},
  pages   = {1214-1233},
  year    = {2005},
  doi     = {10.1137/S1064827502410633}
}

@article{hochbruck2010exponential,
  author  = {Hochbruck, Marlis and Ostermann, Alexander},
  title   = {Exponential integrators},
  journal = {Acta Numerica},
  volume  = {19},
  pages   = {209-286},
  year    = {2010},
  doi     = {10.1017/S0962492910000048}
}

@article{barsinai2019learning,
  author  = {Bar-Sinai, Yohai and Hoyer, Stephan and Hickey, Jason and Brenner, Michael P.},
  title   = {Learning data-driven discretizations for partial differential equations},
  journal = {Proceedings of the National Academy of Sciences},
  volume  = {116},
  number  = {31},
  pages   = {15344-15349},
  year    = {2019},
  doi     = {10.1073/pnas.1814058116}
}

@article{kochkov2021machine,
  author  = {Kochkov, Dmitrii and Smith, Jamie A. and Alieva, Ayya and Wang, Qing and Brenner, Michael P. and Hoyer, Stephan},
  title   = {Machine learning-accelerated computational fluid dynamics},
  journal = {Proceedings of the National Academy of Sciences},
  volume  = {118},
  number  = {21},
  pages   = {e2101784118},
  year    = {2021},
  doi     = {10.1073/pnas.2101784118}
}

@article{zhuang2021learned,
  author  = {Zhuang, Jiawei and Kochkov, Dmitrii and Bar-Sinai, Yohai and Brenner, Michael P. and Hoyer, Stephan},
  title   = {Learned discretizations for passive scalar advection in a two-dimensional turbulent flow},
  journal = {Physical Review Fluids},
  volume  = {6},
  number  = {6},
  pages   = {064605},
  year    = {2021},
  doi     = {10.1103/PhysRevFluids.6.064605}
}

@inproceedings{um2020solver,
  author    = {Um, Kiwon and Brand, Robert and Fei, Yun (Raymond) and Holl, Philipp and Thuerey, Nils},
  title     = {Solver-in-the-loop: {L}earning from differentiable physics to interact with iterative {PDE}-solvers},
  booktitle = {Advances in Neural Information Processing Systems},
  editor    = {Larochelle, H. and Ranzato, M. and Hadsell, R. and Balcan, M. F. and Lin, H.},
  volume    = {33},
  pages     = {6111-6122},
  publisher = {Curran Associates, Inc.},
  year      = {2020},
  url       = {https://proceedings.neurips.cc/paper/2020/hash/43e4e6a6f341e00671e123714de019a8-Abstract.html}
}

@article{list2022learned,
  author  = {List, Bj{\"o}rn and Chen, Li-Wei and Thuerey, Nils},
  title   = {Learned turbulence modelling with differentiable fluid solvers: physics-based loss functions and optimisation horizons},
  journal = {Journal of Fluid Mechanics},
  volume  = {949},
  pages   = {A25},
  year    = {2022},
  doi     = {10.1017/jfm.2022.738}
}

@article{sirignano2020dpm,
  author  = {Sirignano, Justin and MacArt, Jonathan F. and Freund, Jonathan B.},
  title   = {{DPM}: {A} deep learning {PDE} augmentation method with application to large-eddy simulation},
  journal = {Journal of Computational Physics},
  volume  = {423},
  pages   = {109811},
  year    = {2020},
  doi     = {10.1016/j.jcp.2020.109811}
}

@article{melchers2023comparison,
  author  = {Melchers, Hugo and Crommelin, Daan and Koren, Barry and Menkovski, Vlado and Sanderse, Benjamin},
  title   = {Comparison of neural closure models for discretised {PDE}s},
  journal = {Computers and Mathematics with Applications},
  volume  = {143},
  pages   = {94-107},
  year    = {2023},
  doi     = {10.1016/j.camwa.2023.04.030}
}

@article{gupta2023generalized,
  author  = {Gupta, Abhinav and Lermusiaux, Pierre F. J.},
  title   = {Generalized neural closure models with interpretability},
  journal = {Scientific Reports},
  volume  = {13},
  pages   = {10634},
  year    = {2023},
  doi     = {10.1038/s41598-023-35319-w}
}

@article{duraisamy2019turbulence,
  author  = {Duraisamy, Karthik and Iaccarino, Gianluca and Xiao, Heng},
  title   = {Turbulence modeling in the age of data},
  journal = {Annual Review of Fluid Mechanics},
  volume  = {51},
  pages   = {357-377},
  year    = {2019},
  doi     = {10.1146/annurev-fluid-010518-040547}
}

@article{ray2018artificial,
  author  = {Ray, Deep and Hesthaven, Jan S.},
  title   = {An artificial neural network as a troubled-cell indicator},
  journal = {Journal of Computational Physics},
  volume  = {367},
  pages   = {166-191},
  year    = {2018},
  doi     = {10.1016/j.jcp.2018.04.029}
}

@inproceedings{stachenfeld2022learned,
  author    = {Stachenfeld, Kimberly and Fielding, Drummond B. and Kochkov, Dmitrii and Cranmer, Miles and Pfaff, Tobias and Godwin, Jonathan and Cui, Can and Ho, Shirley and Battaglia, Peter and Sanchez-Gonzalez, Alvaro},
  title     = {Learned coarse models for efficient turbulence simulation},
  booktitle = {International Conference on Learning Representations (ICLR)},
  year      = {2022},
  eprint    = {2112.15275},
  archivePrefix = {arXiv}
}

@inproceedings{brandstetter2022message,
  author    = {Brandstetter, Johannes and Worrall, Daniel E. and Welling, Max},
  title     = {Message passing neural {PDE} solvers},
  booktitle = {International Conference on Learning Representations (ICLR)},
  year      = {2022},
  eprint    = {2202.03376},
  archivePrefix = {arXiv}
}

@inproceedings{li2021fourier,
  author    = {Li, Zongyi and Kovachki, Nikola and Azizzadenesheli, Kamyar and Liu, Burigede and Bhattacharya, Kaushik and Stuart, Andrew and Anandkumar, Anima},
  title     = {Fourier neural operator for parametric partial differential equations},
  booktitle = {International Conference on Learning Representations (ICLR)},
  year      = {2021},
  eprint    = {2010.08895},
  archivePrefix = {arXiv}
}

@article{lu2021learning,
  author  = {Lu, Lu and Jin, Pengzhan and Pang, Guofei and Zhang, Zhongqiang and Karniadakis, George Em},
  title   = {Learning nonlinear operators via {DeepONet} based on the universal approximation theorem of operators},
  journal = {Nature Machine Intelligence},
  volume  = {3},
  number  = {3},
  pages   = {218-229},
  year    = {2021},
  doi     = {10.1038/s42256-021-00302-5}
}

@article{kovachki2023neural,
  author  = {Kovachki, Nikola and Li, Zongyi and Liu, Burigede and Azizzadenesheli, Kamyar and Bhattacharya, Kaushik and Stuart, Andrew and Anandkumar, Anima},
  title   = {Neural operator: {L}earning maps between function spaces with applications to {PDE}s},
  journal = {Journal of Machine Learning Research},
  volume  = {24},
  number  = {89},
  pages   = {1-97},
  year    = {2023},
  url     = {https://www.jmlr.org/papers/v24/21-1524.html}
}

@article{mcgreivy2024weak,
  author  = {McGreivy, Nick and Hakim, Ammar},
  title   = {Weak baselines and reporting biases lead to overoptimism in machine learning for fluid-related partial differential equations},
  journal = {Nature Machine Intelligence},
  volume  = {6},
  number  = {10},
  pages   = {1256-1269},
  year    = {2024},
  doi     = {10.1038/s42256-024-00897-5}
}

@article{stolz1999approximate,
  author  = {Stolz, S. and Adams, N. A.},
  title   = {An approximate deconvolution procedure for large-eddy simulation},
  journal = {Physics of Fluids},
  volume  = {11},
  number  = {7},
  pages   = {1699-1701},
  year    = {1999},
  doi     = {10.1063/1.869867}
}

@article{maddu2023stencil,
  author  = {Maddu, Suryanarayana and Sturm, Dominik and Cheeseman, Bevan L. and M{\"u}ller, Christian L. and Sbalzarini, Ivo F.},
  title   = {{STENCIL-NET} for equation-free forecasting from data},
  journal = {Scientific Reports},
  volume  = {13},
  pages   = {12787},
  year    = {2023},
  doi     = {10.1038/s41598-023-39418-6}
}

@article{lee2020model,
  author  = {Lee, Kookjin and Carlberg, Kevin T.},
  title   = {Model reduction of dynamical systems on nonlinear manifolds using deep convolutional autoencoders},
  journal = {Journal of Computational Physics},
  volume  = {404},
  pages   = {108973},
  year    = {2020},
  doi     = {10.1016/j.jcp.2019.108973}
}

@article{agdestein2025discretize,
  author  = {Agdestein, Syver D{\o}ving and Sanderse, Benjamin},
  title   = {Discretize first, filter next: {L}earning divergence-consistent closure models for large-eddy simulation},
  journal = {Journal of Computational Physics},
  volume  = {522},
  pages   = {113577},
  year    = {2025},
  doi     = {10.1016/j.jcp.2024.113577}
}

@article{pang2019fpinns,
  author  = {Pang, Guofei and Lu, Lu and Karniadakis, George Em},
  title   = {{fPINNs}: Fractional physics-informed neural networks},
  journal = {SIAM Journal on Scientific Computing},
  volume  = {41},
  number  = {4},
  pages   = {A2603-A2626},
  year    = {2019},
  doi     = {10.1137/18M1229845}
}

@article{gulian2019machine,
  author  = {Gulian, Mamikon and Raissi, Maziar and Perdikaris, Paris and Karniadakis, George Em},
  title   = {Machine learning of space-fractional differential equations},
  journal = {SIAM Journal on Scientific Computing},
  volume  = {41},
  number  = {4},
  pages   = {A2485-A2509},
  year    = {2019},
  doi     = {10.1137/18M1204991}
}

@article{akhavansafaei2021data,
  author  = {Akhavan-Safaei, Ali and Samiee, Mehdi and Zayernouri, Mohsen},
  title   = {Data-driven fractional subgrid-scale modeling for scalar turbulence: {A} nonlocal {LES} approach},
  journal = {Journal of Computational Physics},
  volume  = {446},
  pages   = {110571},
  year    = {2021},
  doi     = {10.1016/j.jcp.2021.110571}
}

@article{chorin2000optimal,
  author  = {Chorin, Alexandre J. and Hald, Ole H. and Kupferman, Raz},
  title   = {Optimal prediction and the {M}ori-{Z}wanzig representation of irreversible processes},
  journal = {Proceedings of the National Academy of Sciences},
  volume  = {97},
  number  = {7},
  pages   = {2968-2973},
  year    = {2000},
  doi     = {10.1073/pnas.97.7.2968}
}

@article{chorin2002optimal,
  author  = {Chorin, Alexandre J. and Hald, Ole H. and Kupferman, Raz},
  title   = {Optimal prediction with memory},
  journal = {Physica D: Nonlinear Phenomena},
  volume  = {166},
  number  = {3-4},
  pages   = {239-257},
  year    = {2002},
  doi     = {10.1016/S0167-2789(02)00446-3}
}

@article{lopezmarcos1988stability,
  author  = {L{\'o}pez-Marcos, J. C. and Sanz-Serna, J. M.},
  title   = {Stability and convergence in numerical analysis {III}: {L}inear investigation of nonlinear stability},
  journal = {IMA Journal of Numerical Analysis},
  volume  = {8},
  number  = {1},
  pages   = {71-84},
  year    = {1988},
  doi     = {10.1093/imanum/8.1.71}
}

@article{newman1964rational,
  author  = {Newman, D. J.},
  title   = {Rational approximation to $|x|$},
  journal = {Michigan Mathematical Journal},
  volume  = {11},
  number  = {1},
  pages   = {11-14},
  year    = {1964},
  doi     = {10.1307/mmj/1028999029}
}

@article{stahl2003best,
  author  = {Stahl, Herbert R.},
  title   = {Best uniform rational approximation of $x^\alpha$ on $[0,1]$},
  journal = {Acta Mathematica},
  volume  = {190},
  number  = {2},
  pages   = {241-306},
  year    = {2003},
  doi     = {10.1007/BF02392691}
}

@article{stenger1986explicit,
  author  = {Stenger, Frank},
  title   = {Explicit nearly optimal linear rational approximation with preassigned poles},
  journal = {Mathematics of Computation},
  volume  = {47},
  number  = {175},
  pages   = {225-252},
  year    = {1986},
  doi     = {10.1090/S0025-5718-1986-0842132-0}
}

@article{gopal2019solving,
  author  = {Gopal, Abinand and Trefethen, Lloyd N.},
  title   = {Solving {L}aplace problems with corner singularities via rational functions},
  journal = {SIAM Journal on Numerical Analysis},
  volume  = {57},
  number  = {5},
  pages   = {2074-2094},
  year    = {2019},
  doi     = {10.1137/19M125947X}
}

@article{trefethen2021exponential,
  author  = {Trefethen, Lloyd N. and Nakatsukasa, Yuji and Weideman, J. A. C.},
  title   = {Exponential node clustering at singularities for rational approximation, quadrature, and {PDE}s},
  journal = {Numerische Mathematik},
  volume  = {147},
  number  = {1},
  pages   = {227-254},
  year    = {2021},
  doi     = {10.1007/s00211-020-01168-2}
}

@article{herremans2023resolution,
  author  = {Herremans, Astrid and Huybrechs, Daan and Trefethen, Lloyd N.},
  title   = {Resolution of singularities by rational functions},
  journal = {SIAM Journal on Numerical Analysis},
  volume  = {61},
  number  = {6},
  pages   = {2580-2600},
  year    = {2023},
  doi     = {10.1137/23M1551821}
}

@article{nakatsukasa2018aaa,
  author  = {Nakatsukasa, Yuji and S{\`e}te, Olivier and Trefethen, Lloyd N.},
  title   = {The {AAA} algorithm for rational approximation},
  journal = {SIAM Journal on Scientific Computing},
  volume  = {40},
  number  = {3},
  pages   = {A1494-A1522},
  year    = {2018},
  doi     = {10.1137/16M1106122}
}

@article{tang2020rational,
  author  = {Tang, Tao and Wang, Li-Lian and Yuan, Huifang and Zhou, Tao},
  title   = {Rational spectral methods for {PDE}s involving fractional {L}aplacian in unbounded domains},
  journal = {SIAM Journal on Scientific Computing},
  volume  = {42},
  number  = {2},
  pages   = {A585-A611},
  year    = {2020},
  doi     = {10.1137/19M1244299}
}

@article{qian2011algorithm,
  author  = {Qian, Tao and Zhang, Liming and Li, Zhixiong},
  title   = {Algorithm of adaptive {F}ourier decomposition},
  journal = {IEEE Transactions on Signal Processing},
  volume  = {59},
  number  = {12},
  pages   = {5899-5906},
  year    = {2011},
  doi     = {10.1109/TSP.2011.2168520}
}

@article{qian2011adaptive,
  author  = {Qian, Tao and Wang, Yan-Bo},
  title   = {Adaptive {F}ourier series---a variation of greedy algorithm},
  journal = {Advances in Computational Mathematics},
  volume  = {34},
  number  = {3},
  pages   = {279-293},
  year    = {2011},
  doi     = {10.1007/s10444-010-9153-4}
}

@article{song2026mp,
  author  = {Song, Zeyuan and Jiang, Zheyu},
  title   = {{MP-FVM}: Enhancing finite volume method for water infiltration modeling in unsaturated soils via message-passing encoder-decoder network},
  journal = {Computers and Geotechnics},
  volume  = {190},
  pages   = {107745},
  year    = {2026},
  doi     = {10.1016/j.compgeo.2025.107745}
}

@article{song2025physics,
  author  = {Song, Zeyuan and Jiang, Zheyu},
  title   = {A physics-based, data-driven numerical framework for anomalous diffusion of water in soil},
  journal = {Systems and Control Transactions},
  volume  = {4},
  pages   = {2391-2397},
  year    = {2025},
  doi     = {10.69997/sct.163304},
  note    = {{P}roceedings of the 35th European Symposium on Computer Aided Process Engineering ({ESCAPE} 35)}
}

@article{metzler2000random,
  author  = {Metzler, Ralf and Klafter, Joseph},
  title   = {The random walk's guide to anomalous diffusion: {A} fractional dynamics approach},
  journal = {Physics Reports},
  volume  = {339},
  number  = {1},
  pages   = {1-77},
  year    = {2000},
  doi     = {10.1016/S0370-1573(00)00070-3}
}

@article{benson2000application,
  author  = {Benson, David A. and Wheatcraft, Stephen W. and Meerschaert, Mark M.},
  title   = {Application of a fractional advection-dispersion equation},
  journal = {Water Resources Research},
  volume  = {36},
  number  = {6},
  pages   = {1403-1412},
  year    = {2000},
  doi     = {10.1029/2000WR900031}
}

@article{langford1999optimal,
  author  = {Langford, Jacob A. and Moser, Robert D.},
  title   = {Optimal {LES} formulations for isotropic turbulence},
  journal = {Journal of Fluid Mechanics},
  volume  = {398},
  pages   = {321-346},
  year    = {1999},
  doi     = {10.1017/S0022112099006369}
}

@article{beck2019deep,
  author  = {Beck, Andrea and Flad, David and Munz, Claus-Dieter},
  title   = {Deep neural networks for data-driven {LES} closure models},
  journal = {Journal of Computational Physics},
  volume  = {398},
  pages   = {108910},
  year    = {2019},
  doi     = {10.1016/j.jcp.2019.108910}
}

@article{shieh2015new,
  author  = {Shieh, Tien-Tsan and Spector, Daniel E.},
  title   = {On a new class of fractional partial differential equations},
  journal = {Advances in Calculus of Variations},
  volume  = {8},
  number  = {4},
  pages   = {321-336},
  year    = {2015},
  doi     = {10.1515/acv-2014-0009}
}

@article{achleitner2018asymptotic,
  author  = {Achleitner, Franz and Ueda, Yoshihiro},
  title   = {Asymptotic stability of traveling wave solutions for nonlocal viscous conservation laws with explicit decay rates},
  journal = {Journal of Evolution Equations},
  volume  = {18},
  number  = {2},
  pages   = {923-946},
  year    = {2018},
  doi     = {10.1007/s00028-018-0426-6}
}

@article{borwein1994chebyshev,
  author  = {Borwein, Peter and Erd{\'e}lyi, Tam{\'a}s and Zhang, John},
  title   = {Chebyshev polynomials and {Markov}-{Bernstein} type inequalities for rational spaces},
  journal = {Journal of the London Mathematical Society},
  volume  = {50},
  number  = {3},
  pages   = {501-519},
  year    = {1994},
  doi     = {10.1112/jlms/50.3.501}
}

@inproceedings{kingma2015adam,
  author        = {Kingma, Diederik P. and Ba, Jimmy},
  title         = {Adam: {A} method for stochastic optimization},
  booktitle     = {International Conference on Learning Representations (ICLR)},
  year          = {2015},
  eprint        = {1412.6980},
  archivePrefix = {arXiv}
}

@misc{hendrycks2016gaussian,
  author        = {Hendrycks, Dan and Gimpel, Kevin},
  title         = {Gaussian error linear units ({GELUs})},
  year          = {2016},
  eprint        = {1606.08415},
  archivePrefix = {arXiv},
  howpublished  = {arXiv preprint}
}

@inproceedings{loshchilov2017sgdr,
  author        = {Loshchilov, Ilya and Hutter, Frank},
  title         = {{SGDR}: Stochastic gradient descent with warm restarts},
  booktitle     = {International Conference on Learning Representations (ICLR)},
  year          = {2017},
  eprint        = {1608.03983},
  archivePrefix = {arXiv}
}

@inproceedings{pascanu2013difficulty,
  author    = {Pascanu, Razvan and Mikolov, Tomas and Bengio, Yoshua},
  title     = {On the difficulty of training recurrent neural networks},
  booktitle = {Proceedings of the 30th International Conference on Machine Learning},
  series    = {Proceedings of Machine Learning Research},
  editor    = {Dasgupta, Sanjoy and McAllester, David},
  volume    = {28},
  pages     = {1310-1318},
  publisher = {PMLR},
  address   = {Atlanta, Georgia, USA},
  year      = {2013},
  url       = {https://proceedings.mlr.press/v28/pascanu13.html}
}

@article{mcshane1934extension,
  author  = {McShane, E. J.},
  title   = {Extension of range of functions},
  journal = {Bulletin of the American Mathematical Society},
  volume  = {40},
  number  = {12},
  pages   = {837-842},
  year    = {1934},
  doi     = {10.1090/S0002-9904-1934-05978-0}
}

@book{kallenberg2002foundations,
  author    = {Kallenberg, Olav},
  title     = {Foundations of Modern Probability},
  edition   = {Second},
  series    = {Probability and Its Applications},
  publisher = {Springer},
  address   = {New York},
  year      = {2002},
  doi       = {10.1007/978-1-4757-4015-8}
}

@book{billingsley1999convergence,
  author    = {Billingsley, Patrick},
  title     = {Convergence of Probability Measures},
  edition   = {Second},
  series    = {Wiley Series in Probability and Statistics},
  publisher = {Wiley},
  address   = {New York},
  year      = {1999},
  doi       = {10.1002/9780470316962}
}

@misc{olver2026dlmfgen,
  author       = {Olver, F. W. J. and Olde Daalhuis, A. B. and Lozier, D. W. and Schneider, B. I. and Boisvert, R. F. and Clark, C. W. and Miller, B. R. and Saunders, B. V. and Cohl, H. S. and McClain, M. A.},
  title        = {{NIST Digital Library of Mathematical Functions}},
  howpublished = {Release 1.2.8 of 2026-09-15},
  url          = {https://dlmf.nist.gov/},
  year         = {2026}
}

@article{trefethen2014exponentially,
  author  = {Trefethen, Lloyd N. and Weideman, J. A. C.},
  title   = {The exponentially convergent trapezoidal rule},
  journal = {SIAM Review},
  volume  = {56},
  number  = {3},
  pages   = {385-458},
  year    = {2014},
  doi     = {10.1137/130932132}
}

\newpage

\setcounter{section}{0}
\renewcommand{\thesection}{S\arabic{section}}
\renewcommand{\thesubsection}{\thesection.\arabic{subsection}}
\makeatletter
\renewcommand{\p@section}{Section~}
\renewcommand{\p@subsection}{Section~}
\makeatother
\numberwithin{equation}{section}
\renewcommand{\theequation}{\thesection.\arabic{equation}}
\renewcommand{\thetable}{S\arabic{table}}
\renewcommand{\thefigure}{S\arabic{figure}}

\begin{center}
{\Large\bfseries Supplementary Material for\\[0.4em]
``A hybrid finite-volume scheme with learned rational reconstructions and exact nonlocal fluxes for the space-fractional Burgers equation''\par}
\vspace{1.2em}
{\large Zeyuan Song and Zheyu Jiang\par}
\vspace{0.4em}
{\itshape School of Chemical Engineering, Oklahoma State University, Stillwater, OK 74078, USA\par}
\end{center}
\vspace{1em}

\noindent{This Supplementary Material accompanies the main text.
Section~S1 proves the results of Section~\ref{sec:theory} whose proofs are deferred, Section~S2 derives the closed forms of the rational atoms and describes their implementation, Section~S3 gives the data, training and evaluation details behind Section~\ref{sec:coupled}, Section~S4 collects additional results, and Section~S5 states supporting classical results.
Sections, equations, results and tables of the Supplementary Material carry the prefix S, and all other numbers refer to the main text.}

\section{Proofs}\label{app:proofs}

This {section} proves the deferred results of Section~\ref{sec:theory} in the order in which they appear there.
We write $\hat u_k=\frac{1}{2\pi}\int_\T u(x)e^{-ikx}\dd x$, so that $u=\sum_k\hat u_ke^{ikx}$ in $L^2$ and $\|u\|_{L^2}^2=2\pi\sum_k|\hat u_k|^2$, and we set $e_k(x)=e^{ikx}$.
Functions on $\T$ are identified with their periodic extensions, and the H\"older seminorm $[u]_\eta=\sup_{x\neq y}|u(x)-u(y)|/|x-y|^\eta$ with $0<\eta\le1$ is taken over all $x,y\in\R$.
Complex values are allowed where convenient, since every operator below is linear and maps real functions to real functions, and special-function identities are quoted from the DLMF~\cite{olver2026dlmfgen}.

\subsection{Proof of Lemma~\ref{lem:riesz}}\label{app:proofs:riesz}

For the constants in Eq.~\eqref{eq:kappa}, we have $C_\alpha/\alpha=\Gamma(\alpha)\sin(\pi\alpha/2)/\pi$ by the identity $\Gamma(1+\alpha)=\alpha\Gamma(\alpha)$.
The reflection formula~\cite[Eq.~5.5.3]{olver2026dlmfgen} $\Gamma(\alpha)\Gamma(1-\alpha)=\pi/\sin(\pi\alpha)$ and the identity $\sin(\pi\alpha)=2\sin(\frac{\pi\alpha}{2})\cos(\frac{\pi\alpha}{2})$ give
\[
\frac{1}{2\Gamma(1-\alpha)\cos(\pi\alpha/2)}=\frac{\Gamma(\alpha)\sin(\pi\alpha)}{2\pi\cos(\pi\alpha/2)}=\frac{\Gamma(\alpha)\sin(\pi\alpha/2)}{\pi}.
\]
The reflection formula at $\alpha/2$ gives $1/\Gamma(1-\frac\alpha2)=\Gamma(\frac\alpha2)\sin(\frac{\pi\alpha}{2})/\pi$, and with the duplication formula $\Gamma(\frac\alpha2)\Gamma(\frac{1+\alpha}{2})=2^{1-\alpha}\sqrt\pi\,\Gamma(\alpha)$~\cite[Eq.~5.5.5]{olver2026dlmfgen} the third expression takes the same value.
All three expressions therefore equal $\Gamma(\alpha)\sin(\pi\alpha/2)/\pi$, which is positive because $\Gamma(\alpha)>0$ and $\pi/2<\pi\alpha/2<\pi$.

The kernel computations rest on the Mellin transform of the sine~\cite[Eq.~5.9.7]{olver2026dlmfgen} with the exponent $1-\alpha\in(-1,0)$,
\begin{equation}\label{eq:A:mellin}
\int_0^\infty\frac{\sin s}{s^{\alpha}}\dd s=\Gamma(1-\alpha)\sin\Big(\frac{\pi(1-\alpha)}{2}\Big)=\Gamma(1-\alpha)\cos\Big(\frac{\pi\alpha}{2}\Big)=\frac{1}{2\kappa_\alpha},
\end{equation}
which converges absolutely because the integrand is bounded by $s^{1-\alpha}$ near the origin and by $s^{-\alpha}$ at infinity.
For $k\neq0$, the identity $e_k(x+z)-e_k(x-z)=2i\,\sgn(k)\sin(|k|z)\,e_k(x)$, the substitution $s=|k|z$ and Eq.~\eqref{eq:A:mellin} yield
\begin{equation}\label{eq:A:expk}
\kappa_\alpha\int_0^\infty\frac{e_k(x+z)-e_k(x-z)}{z^\alpha}\dd z=i\,\sgn(k)|k|^{\alpha-1}e_k(x),
\end{equation}
and both sides vanish for $k=0$.

For item~(i), let $b_\alpha=\int_0^\infty s^{-\alpha}|\sin s|\dd s<\infty$, so that $\int_0^\infty z^{-\alpha}|e_k(x+z)-e_k(x-z)|\dd z=2b_\alpha|k|^{\alpha-1}$ for every integer $k$.
The hypothesis also gives $\sum_k|\hat u_k|<\infty$, so the Fourier series of $u$ converges absolutely and uniformly to $u$, and that of $\Ra u$ converges absolutely.
Fubini's theorem therefore allows us to integrate the series of $u(x+z)-u(x-z)$ termwise against $z^{-\alpha}$, and Eq.~\eqref{eq:A:expk} gives the first representation in Eq.~\eqref{eq:riesz-kernel} with an absolutely convergent integral.
The substitutions $y=x+z$ for $y>x$ and $y=x-z$ for $y<x$ turn the second representation into the difference of the half-line integrals of $(u(x\pm z)-u(x))z^{-\alpha}$, which is the first, and both converge absolutely because $|u(x\pm z)-u(x)|\le\min\{\|u'\|_{L^\infty}z,2\|u\|_{L^\infty}\}$.

For item~(ii), we may assume $\eta\le1$, because a H\"older function of order larger than one is constant and the integral and the claimed coefficients then vanish.
The bounds $|u(x+z)-u(x-z)|\le[u]_\eta(2z)^\eta$ and $|u(x+z)-u(x-z)|\le2\|u\|_{L^\infty}$ dominate the integrand, uniformly in $x$, by $g(z)=\min\{2^\eta[u]_\eta z^{\eta-\alpha},2\|u\|_{L^\infty}z^{-\alpha}\}$, which is integrable because $\eta-\alpha>-1$ and $\alpha>1$, so the integral converges absolutely and is continuous in $x$ by dominated convergence.
The bound $\int_0^\infty\int_\T|u(x+z)-u(x-z)|z^{-\alpha}\dd x\dd z\le2\pi\int_0^\infty g(z)\dd z$ justifies Fubini's theorem for the Fourier coefficients.
With $\frac{1}{2\pi}\int_\T(u(x+z)-u(x-z))e^{-ikx}\dd x=2i\sin(kz)\,\hat u_k$, Eq.~\eqref{eq:A:mellin} and the substitution $s=|k|z$ give the coefficients $\widehat{(\Ra u)}_k=2i\kappa_\alpha\hat u_k\int_0^\infty z^{-\alpha}\sin(kz)\dd z=i\,\sgn(k)|k|^{\alpha-1}\hat u_k$.
Under the hypotheses of item~(i) the integral equals the absolutely convergent Fourier series at every point, so the two definitions of $\Ra u$ agree.

For item~(iii), the symbol of $\partial_x\Ra$ is $(ik)(i\,\sgn(k)|k|^{\alpha-1})=-|k|^\alpha$, which is the symbol of $-(-\Delta)^{\alpha/2}$.
We prove the pointwise statements under $\sum_k|k|^\alpha|\hat u_k|<\infty$, which holds for $u\in H^s(\T)$ with $s>\alpha+\frac12$ by the Cauchy-Schwarz inequality, hence under Assumption~\ref{ass:regularity}.
The Fourier series of $\Ra u$ may then be differentiated termwise, and $\partial_x\Ra u=-(-\Delta)^{\alpha/2}u$ holds pointwise with both sides continuous.
For a solution with $f=0$ under Assumption~\ref{ass:regularity}, the term $u^2u_x=(u^3/3)_x$ has zero integral over $\T$, so Eq.~\eqref{eq:pde} is dissipative in the sense that $\frac12\frac{\mathrm{d}}{\mathrm{d}t}\|u\|_{L^2}^2=-\nu\int_\T u\,(-\Delta)^{\alpha/2}u\dd x=-2\pi\nu\sum_k|k|^\alpha|\hat u_k|^2\le0$.
For the hypersingular representation, an integration by parts and Eq.~\eqref{eq:A:mellin} give
\begin{equation}\label{eq:A:hyper}
\int_0^\infty\frac{1-\cos s}{s^{1+\alpha}}\dd s=\frac1\alpha\int_0^\infty\frac{\sin s}{s^\alpha}\dd s=\frac{1}{2\alpha\kappa_\alpha}=\frac{1}{2C_\alpha},
\end{equation}
where the boundary terms vanish because $(1-\cos s)s^{-\alpha}$ behaves like $s^{2-\alpha}/2$ near the origin and tends to zero at infinity.
For $\varepsilon>0$ and bounded $u$, the truncated integral over $|y-x|>\varepsilon$ converges absolutely and equals $\int_\varepsilon^\infty(u(x+z)+u(x-z)-2u(x))z^{-1-\alpha}\dd z$.
For $u=e_k$ the numerator equals $-2(1-\cos kz)e_k(x)$, and Eq.~\eqref{eq:A:hyper} with $s=|k|z$ gives $-C_\alpha^{-1}|k|^\alpha e_k(x)=-C_\alpha^{-1}(-\Delta)^{\alpha/2}e_k(x)$ for the integral over $(0,\infty)$.
For general $u$, the bound $\sum_k|\hat u_k|\int_0^\infty2(1-\cos kz)z^{-1-\alpha}\dd z=C_\alpha^{-1}\sum_k|k|^\alpha|\hat u_k|<\infty$ allows termwise integration and shows absolute convergence over $(0,\infty)$.
The principal value is therefore the limit of the truncated integrals as $\varepsilon\to0$, and multiplication by $\nu C_\alpha$ gives the representation of item~(iii).
Under the hypotheses of item~(i) alone the pointwise identity can fail, since $u(x)=\sum_{n\ge1}n^{-1}2^{-n\alpha}\cos(2^nx)$ satisfies them, while at $x=0$ the series of $(-\Delta)^{\alpha/2}u$ diverges and the principal value equals $-\infty$.
The constant $C_\alpha=\alpha\kappa_\alpha$ equals the constant $2^\alpha\Gamma(\frac{1+\alpha}{2})/(\sqrt\pi\,|\Gamma(-\frac\alpha2)|)$ of Kwa\'snicki~\cite{kwasnicki2017ten}, by the third expression in Eq.~\eqref{eq:kappa} and $|\Gamma(-\frac\alpha2)|=\frac2\alpha\Gamma(1-\frac\alpha2)$.

For item~(iv), the function $\cos x=\frac12(e_1(x)+e_{-1}(x))$ is mapped to $\frac12(ie_1(x)-ie_{-1}(x))=-\sin x$, and at $\alpha=2$ the symbol $i\,\sgn(k)|k|$ equals $ik$, the symbol of $\partial_x$.

\subsection{Proof of Lemma~\ref{lem:jump}}\label{app:proofs:jump}

We split the integral in Eq.~\eqref{eq:riesz-kernel} at $z=\varepsilon/2$, and since the integrand of the far part is at most $2\|u\|_\infty z^{-\alpha}$, this part is at most $2\|u\|_\infty\kappa_\alpha(\varepsilon/2)^{1-\alpha}/(\alpha-1)$.
For $z<\varepsilon/2$, both points $x\pm z$ lie in $(x_0-\varepsilon,x_0+\varepsilon)$, so the difference $u(x+z)-u(x-z)$ splits into a jump part and a H\"older part.
The H\"older part contributes at most $\kappa_\alpha[r]_\eta2^\eta\int_0^{\varepsilon/2}z^{\eta-\alpha}\dd z$ to the near part of the integral.
For $x\neq x_0$, the difference $H_{x_0}(x+z)-H_{x_0}(x-z)$ equals $\mathbf{1}[z>|x-x_0|]$ up to one point, so the jump part equals $c\,\kappa_\alpha\int_{|x-x_0|}^{\varepsilon/2}z^{-\alpha}\dd z=\frac{c\,\kappa_\alpha}{\alpha-1}(|x-x_0|^{1-\alpha}-(\varepsilon/2)^{1-\alpha})$.
Collecting the far part, the H\"older part and this jump part gives the expansion of Eq.~\eqref{eq:jump}.
For (i), the piecewise-constant function equals $v_j+(v_{j+1}-v_j)H_{\xf{j}}$ near the face $\xf{j}$, so the expansion applies with $c=v_{j+1}-v_j$ and a constant function $r$.
For (ii), the part $z\ge\varepsilon$ of the integral at $x=x_0$ is bounded by $2\|u\|_\infty\kappa_\alpha\varepsilon^{1-\alpha}/(\alpha-1)$, and it vanishes when $u$ is even about $x_0$.
For $0<z<\varepsilon$ the odd cusp gives $u(x_0+z)-u(x_0-z)=2z^\eta$, which is integrable against $z^{-\alpha}$ if and only if $\eta>\alpha-1$, while the even cusp gives zero.

\subsection{Proof of Theorem~\ref{thm:defect}}\label{app:proofs:defect}

Under Assumption~\ref{ass:regularity}, the flux $F(u)(\cdot,t)$ belongs to $C^1(\T)$ for every $t$, and $\partial_tu$ is continuous in $(x,t)$.
We integrate Eq.~\eqref{eq:pde} over $C_j\times[t_n,t_{n+1}]$, exchange the order of integration and apply the fundamental theorem of calculus in $x$, which gives
\begin{align*}
(\Ph u(t_{n+1}))_j-(\Ph u(t_n))_j&=-\frac1h\int_{t_n}^{t_{n+1}}\Big(F(u)(\xf{j},t)-F(u)(x_{j-1/2},t)\Big)\dd t\\
&\quad+\int_{t_n}^{t_{n+1}}(\Ph f)_j\dd t.
\end{align*}
Subtracting the scheme of Eq.~\eqref{eq:fvscheme} evaluated at $v=\Ph u(t_n)$ yields Eq.~\eqref{eq:defect}, and the semi-discrete identity follows from the same computation without the time integral.

\subsection{Proof of Corollary~\ref{cor:flux}}\label{app:proofs:flux}

On the periodic grid, $\Dm g=0$ holds if and only if $g$ is constant, so the kernel of $\Dm$ is spanned by the constant vector and its range has dimension $N-1$.
By telescoping, the range lies in the zero-sum hyperplane, which also has dimension $N-1$, and this proves (i).
Item~(ii) is~(i) applied to $r=\Sh(v)-v-\int_{t_n}^{t_{n+1}}\Ph f\dd t$, and item~(iii) follows from Theorem~\ref{thm:defect} and telescoping.
Item~(iv) follows from Theorem~\ref{thm:defect} and from the definition of closability, since $\widetilde{\mathcal S}_h(\Ph u)=\Ph\Ev u$ gives $\mathcal M^\ast(\Ph u)=\dfh(u)$.
Item~(v) is a Taylor expansion of the time average in Eq.~\eqref{eq:defect}.

\subsection{Proof of Proposition~\ref{prop:nearfar}}\label{app:proofs:nearfar}

For (i), integration by parts gives $\hat u_k=O(|k|^{-4})$ and hence $\sum_k|k|^\alpha|\hat u_k|<\infty$, so item~(iii) of Lemma~\ref{lem:riesz} shows that the full principal value integral equals $(-\Delta)^{\alpha/2}u(x)$.
A Taylor expansion of the near part to fourth order around $x$ gives
\begin{align*}
C_\alpha\int_0^h\frac{2u(x)-u(x+z)-u(x-z)}{z^{1+\alpha}}\dd z&=-C_\alpha\int_0^h\frac{u''(x)z^2+O(z^4)}{z^{1+\alpha}}\dd z\\
&=-\frac{C_\alpha}{2-\alpha}\,u''(x)\,h^{2-\alpha}+O(h^{4-\alpha}),
\end{align*}
and the far field is the full operator minus this near part, which proves (i).
For (ii), we split $G$ into the instantaneous error $F(u(t_n))(\xf{j})-\hat F_{j+1/2}(\Ph u(t_n))$ and the time-averaging error, which is at most $\frac{\Delta t}{2}\sup|\partial_tF(u)|$.
By the consistency assumption, the instantaneous error equals $-\nu(\Ra-\tilde{\mathcal R})u(\xf{j})$ up to the face errors of the two parts of the numerical flux, which proves the expansion of $G$.
For $\mathcal T_h=-\frac1h\Dm G$, the identity $\frac1h\Dm[g(x_{\cdot+1/2})]=\Ph(\partial_xg)$ for $g\in C^1(\T)$ turns the model term into $\nu\Ph[\partial_x(\Ra-\tilde{\mathcal R})u]$, and the face errors contribute $O(h^q)$ by assumption.
The same identity writes the contribution of the time-averaging error as $-\frac{1}{\Delta t}\int_{t_n}^{t_{n+1}}\Ph\big[\partial_xF(u(t))-\partial_xF(u(t_n))\big]\dd t$, which is at most $\frac{\Delta t}{2}\sup|\partial_t\partial_xF(u)|$ in modulus.
For (iii), the integral over $z>0$ in Eq.~\eqref{eq:riesz-kernel} and the bound $|u(x+z)-u(x-z)|\le2z\|u'\|_\infty$ give a near part of at most $2\kappa_\alpha\|u'\|_\infty\int_0^\rho z^{1-\alpha}\dd z$.
For $\rho=Rh$, the periodic images of the $R$ cells on each side of $x$ enter the far part as $\kappa_\alpha\int_0^{Rh}D(s)(u(x+s)-u(x-s))\dd s$, with the odd kernel $D$ of \ref{app:proofs:dipole}.
The bound $|D(s)|\le c|s|$, with a constant $c$ depending only on $\alpha$, then bounds this part by $\frac23\kappa_\alpha c\|u'\|_\infty(Rh)^3$.

\subsection{Proof of Proposition~\ref{prop:closable}}\label{app:proofs:closable}

In items~(i) to~(iii), $|\cdot|$ denotes a fixed norm on $\R^N$, and in items~(iv) and~(v) the norm of the inner product chosen there.
For item~(i), if $\widetilde{\mathcal S}_h$ closes the diagram on $M$ and $u,w\in M$ satisfy $\Ph u=\Ph w$, then $\Ph\Ev u=\widetilde{\mathcal S}_h(\Ph u)=\widetilde{\mathcal S}_h(\Ph w)=\Ph\Ev w$.
Conversely, under the fiber condition we set $\widetilde{\mathcal S}_h(v)=\Ph\Ev u$ for $v\in\Ph(M)$ and any $u\in M$ with $\Ph u=v$, and $\widetilde{\mathcal S}_h(v)=\Sh(v)$ for $v\notin\Ph(M)$.
The fiber condition makes this definition independent of the choice of $u$, and the defining relation shows that every closing map takes the same values on $\Ph(M)$.
For $u\in M$ the ideal closure satisfies $\mathcal M^\ast(\Ph u)=\Ph\Ev u-\Sh\Ph u=\dfh(u)$ by Eq.~\eqref{eq:defect-def}.
For the composition, let each $\mathcal E^{(k)}$ send $M$ into itself and satisfy the fiber condition, and let $\widetilde{\mathcal S}_h^{(k)}$ close the diagram of step $k$ on $M$.
For $u\in M$ the state $u_n=\mathcal E^{(n-1)}\circ\cdots\circ\mathcal E^{(0)}u$ lies in $M$, so $\widetilde{\mathcal S}_h^{(n)}(\Ph u_n)=\Ph\mathcal E^{(n)}u_n=\Ph u_{n+1}$, and induction on $n$ gives the identity.

For item~(ii), injectivity of $\Ph$ on $M$ implies the fiber condition, and $v=\Ph u$ with $u\in M$ gives $\mathcal D^\ast v=u$, hence $\widetilde{\mathcal S}_h(v)=\Ph\Ev\mathcal D^\ast v$.
Since $\Ph\mathcal D^\ast v=v$ on $\Ph(M)$, the definition of the defect gives $\Sh(v)+\dfh(\mathcal D^\ast v)=\Sh(v)+\Ph\Ev\mathcal D^\ast v-\Sh(\Ph\mathcal D^\ast v)=\Ph\Ev\mathcal D^\ast v$.
If $M$ is compact in a topology in which $\Ph$ is continuous, then $\Ph|_M$ is a continuous bijection onto the Hausdorff space $\Ph(M)$.
It maps closed, hence compact, subsets of $M$ to compact, hence closed, sets, so $\mathcal D^\ast$ is continuous.
Under the H\"older hypothesis, let $L_E$ be a Lipschitz constant of $\Ph\Ev$ on $M$ in the $L^2$ norm, and $L_S$ one of $\Sh$ on $\Ph(M)$.
For $v=\Ph u$ and $v'=\Ph w$ in $\Ph(M)$, the H\"older hypothesis gives $|\widetilde{\mathcal S}_h(v)-\widetilde{\mathcal S}_h(v')|=|\Ph\Ev u-\Ph\Ev w|\le L_E\|u-w\|_{L^2}\le L_EC\,|v-v'|^\eta$.
The set $\Ph(M)$ is bounded, since $M$ is bounded in $L^2$ and $\|\Ph u\|_h\le\|u\|_{L^2}$, and with its diameter $d_M$ we have $|\Sh(v)-\Sh(v')|\le L_Sd_M^{1-\eta}|v-v'|^\eta$ on $\Ph(M)$.
Hence $\mathcal M^\ast=\widetilde{\mathcal S}_h-\Sh$ is $\eta$-H\"older on $\Ph(M)$ with a constant of at most $L_EC+L_Sd_M^{1-\eta}$.
For the extension, $\varrho(v,v')=|v-v'|^\eta$ is a metric on $\R^N$ for $\eta\in(0,1]$, because $t\mapsto t^\eta$ is increasing and subadditive on $[0,\infty)$.
An $\eta$-H\"older map with constant $L$ is $L$-Lipschitz for $\varrho$, and so is the McShane extension~\cite{mcshane1934extension} $\tilde g(v)=\inf_{v'\in\Ph(M)}(g(v')+L\varrho(v,v'))$ of each real component $g$, which coincides with $g$ on $\Ph(M)$.
Collecting the components gives $\eta$-H\"older extensions of $\widetilde{\mathcal S}_h$ and $\mathcal M^\ast$ to $\R^N$, with constants that depend only on $L$, $N$ and the norm.

For item~(iii), we regard every $u\in M$ as the value at some time of a solution that satisfies Assumption~\ref{ass:regularity} and stays in $M$, with one common forcing.
By Eq.~\eqref{eq:semidiscrete}, the required relation at such a state reads $\Phi(\Ph u)=-\frac1h\Dm[F(u)(x_{\cdot+1/2})]$, in which the forcing no longer appears.
If such a map exists, then $\Ph u=\Ph w$ gives $\Dm[F(u)(x_{\cdot+1/2})]=-h\Phi(\Ph u)=\Dm[F(w)(x_{\cdot+1/2})]$, which is the stated condition.
Conversely, under this condition the same formula defines $\Phi$ consistently on $\Ph(M)$, and every extension of it to $\R^N$ has the required property.
If $\Ph$ is injective on $M$, then $\mathcal D^\ast(\Ph u)=u$, so $\Phi_{\mathcal D^\ast}$ coincides with this map on $\Ph(M)$ and is exact there.

For item~(iv), since $Y\in L^2(\mu)$, the conditional expectation $\EE[Y\mid\sigma(V)]$ exists, and Doob's functional representation lemma~\cite[Lemma~1.13]{kallenberg2002foundations} provides a Borel map $m^\star$ with $\EE[Y\mid\sigma(V)]=m^\star(V)$ almost surely.
For a Borel map $m$ with $m(V)\in L^2$, the difference $Z=m^\star(V)-m(V)$ is square integrable and $\sigma(V)$-measurable, while $\EE[Y-m^\star(V)\mid V]=0$.
With the dot denoting the chosen inner product, taking out what is known~\cite{kallenberg2002foundations} gives $\EE[(Y-m^\star(V))\cdot Z]=\EE\big[\EE[Y-m^\star(V)\mid V]\cdot Z\big]=0$.
Expanding $|Y-m(V)|^2=|(Y-m^\star(V))+Z|^2$ in this inner product then yields the Pythagorean identity.

For item~(v), we define $\tr\Cov(Y\mid V)=\EE\big[|Y-m^\star(V)|^2\mid V\big]$ in the same inner product, so that $J_{\mathrm{irr}}(Y)=\EE|Y-m^\star(V)|^2$.
By the Pythagorean identity every admissible $m$ has a mean-square error of at least $J_{\mathrm{irr}}(Y)$, with equality for $m=m^\star$, so $J_{\mathrm{irr}}(Y)$ is the minimum and involves no network.
For $Y=Y_0-s(V)$ with a Borel map $s$ and $s(V)\in L^2$ we have $Y-\EE[Y\mid V]=Y_0-\EE[Y_0\mid V]$, so $J_{\mathrm{irr}}(Y)=J_{\mathrm{irr}}(Y_0)$, and schemes that differ by such a function have defects with the same irreducible error.
If $J_{\mathrm{irr}}(Y)=0$, then $Y=m^\star(V)$ almost surely, and conversely a square-integrable $Y=g(V)$ with Borel $g$ equals its own conditional expectation.
For the last statement, we regard $\mu$ as a measure on $U$ with $\mu(\mathrm{dom}\,\Ev)=1$, and $\dfh$ is Borel on $\mathrm{dom}\,\Ev$ because $\Ev$ and $\Sh$ are Borel and $\Ph$ is continuous.
If $J_{\mathrm{irr}}(\dfh)=0$, the Borel set $M_1=\{u\in\mathrm{dom}\,\Ev:\dfh(u)=m^\star(\Ph u)\}$ has full measure, and $\widetilde{\mathcal S}_h=\Sh+m^\star$ closes the diagram on $M_1$.

Conversely, let $\widetilde{\mathcal S}_h$ close the diagram on a Borel set $M_0\subset\mathrm{dom}\,\Ev$ of full measure, and put $g=\widetilde{\mathcal S}_h-\Sh$, so that $\dfh(u)=g(\Ph u)$ on $M_0$, where $g$ need not be Borel.
Every Borel probability measure on the Polish space $U$ is inner regular with respect to closed sets and tight~\cite[Section~1]{billingsley1999convergence}, so every Borel set $E\subset U$ contains compact sets $K_n$ with $\mu(E\setminus K_n)<1/n$.
We fix a Borel set $B$, apply this to $E=M_0\cap\dfh^{-1}(B)$ and put $A=\bigcup_n\Ph(K_n)$, a countable union of compact sets and hence a Borel set.
Every $u\in M_0$ with $\Ph u\in A$ has the cell averages of some $u'\in E$, so $\dfh(u)=g(\Ph u)=g(\Ph u')=\dfh(u')\in B$, and $u\in E$.
Hence $\bigcup_nK_n\subset M_0\cap\Ph^{-1}(A)\subset E$, so the events $\{\dfh\in B\}$ and $\{V\in A\}$ differ only by a $\mu$-null set.
Since $B$ is arbitrary, $\dfh$ is measurable for the $\sigma$-algebra generated by $\sigma(V)$ and the $\mu$-null sets, on which conditional expectation agrees almost surely with $\EE[\,\cdot\mid V]$.
Consequently $\dfh=\EE[\dfh\mid V]=m^\star(V)$ almost surely, and $J_{\mathrm{irr}}(\dfh)=0$, which proves the converse.
Finally, every such closing map agrees with $\Sh+m^\star$ on $\Ph(M_2)$, where $M_2=M_0\cap M_1$ is a Borel set of full measure.
Compact sets $K_n\subset M_2$ with $\mu(M_2\setminus K_n)<1/n$ give a Borel set $A_2=\bigcup_n\Ph(K_n)\subset\Ph(M_2)$ with $\mu(\Ph^{-1}(A_2))=1$, so $\widetilde{\mathcal S}_h=\Sh+m^\star$ holds $\mu\circ\Ph^{-1}$-almost everywhere.

\subsection{Proof of Proposition~\ref{prop:decoders}}\label{app:proofs:decoders}

For item~(i), applying $\Ph$ and using $\Ph\Ih=I$ twice gives $\Ph\Dth v=v+\Ph g_\theta(v)-\Ph g_\theta(v)=v$.
Conversely, for any decoder with $\Ph\mathcal D=I$ the splitting $\mathcal D=\Ih\Ph\mathcal D+(I-\Ih\Ph)\mathcal D=\Ih+(I-\Ih\Ph)\mathcal D$ exhibits the required form.
For item~(ii), the face values of $\mathcal Dv$ and of $\Ra\mathcal Dv$ are finite by Lemma~\ref{lem:riesz} because $\eta>\alpha-1$, and the sum of a backward difference telescopes on the periodic grid.
For item~(iii), the fundamental theorem of calculus on each cell and $\partial_x\Ra=-(-\Delta)^{\alpha/2}$ give the first identity for the smooth function $\Ih v$.
The mode $(\hat v_k/\mu_k)e^{ikx}$ is mapped by $\Ph(-\Delta)^{\alpha/2}$ to $|k|^\alpha\hat v_ke^{ikjh}$, and the Nyquist term gives $(N/2)^\alpha\hat v_{N/2}(-1)^j$ in the same way, which proves the second identity.

\subsection{Proof of Theorem~\ref{thm:rea}}\label{app:proofs:rea}

The proof rests on the following classical aliasing identity, which loses the sign $(-1)^m$ when the origin is shifted to a cell center, as in Eq.~\eqref{eq:E:alias}.

\begin{lemma}[Aliasing of cell averages]\label{lem:A:alias}
Let $u\in L^2(\T)$ and $v=\Ph u$, and let $\hat v_\kappa=N^{-1}\sum_jv_je^{-i\kappa jh}$ and $\mu_k=(e^{ikh}-1)/(ikh)$ be as in Section~\ref{sec:decoder}, with $\mu_0=1$.
Then for every $\kappa\in\Z$ the discrete coefficient $\hat v_\kappa$ is given by the absolutely convergent series
\begin{equation}\label{eq:A:alias}
\hat v_\kappa=\sum_{m\in\Z}\mu_{\kappa+mN}\,\hat u_{\kappa+mN}=e^{i\kappa h/2}\sum_{m\in\Z}(-1)^m\sigma(\kappa+mN)\,\hat u_{\kappa+mN},
\end{equation}
and the weights in this series satisfy $\sum_{m\in\Z}\sigma(\kappa+mN)^2=1$.
\end{lemma}

\begin{proof}
The multiplier factorizes as $\mu_k=e^{ikh/2}\sigma(k)$, so $|\mu_k|\le\min\{1,2/(|k|h)\}$ is square summable, and direct integration over $C_j$ gives $(\Ph e_k)_j=\mu_ke^{ikjh}$.
Since $\Ph$ is continuous on $L^2$, applying it to the Fourier series of $u$ gives $v_j=\sum_k\mu_k\hat u_ke^{ikjh}$, a series that converges absolutely by the Cauchy-Schwarz inequality.
The discrete orthogonality relation $N^{-1}\sum_je^{i(k-\kappa)jh}=\mathbf 1[k\equiv\kappa\bmod N]$ then gives the first expression in Eq.~\eqref{eq:A:alias}.
The second expression follows from $e^{i(\kappa+mN)h/2}=(-1)^me^{i\kappa h/2}$, which holds because $Nh=2\pi$.
With $z=\kappa h/2$ we have $\sigma(\kappa+mN)^2=\sin^2z/(z+m\pi)^2$, and the expansion $\sum_{m\in\Z}(z+m\pi)^{-2}=\csc^2z$~\cite[Eq.~4.22.4]{olver2026dlmfgen} gives the value one whenever $z\notin\pi\Z$.
If $z\in\pi\Z$, exactly one weight equals one and all other weights vanish, so the identity holds as well.
\end{proof}

We prove item~(i) for complex-valued $u\in L^2(\T)$ and write $e=u-\Ih\Ph u$ for the reconstruction error.
For odd $N$, discrete Fourier inversion and $(\Ph e_\kappa)_j=\mu_\kappa e^{i\kappa jh}$ give $\Ih v=\sum_{|\kappa|\le(N-1)/2}(\hat v_\kappa/\mu_\kappa)e_\kappa$.
For both parities we partition $\Z$ into the $N$ alias classes $\kappa+N\Z$ with $-N/2<\kappa\le N/2$, whose exponentials span mutually orthogonal subspaces of $L^2(\T)$.
The histopolant maps $\hat v_\kappa$ to modes of the class of $\kappa$, and $\T_N$ is the sum of its intersections with these subspaces.
Hence $\|e\|_{L^2}^2$ and $\|Q_Nu\|_{L^2}^2$ split over the classes, and it suffices to compare the two contributions of each class.

Let $\kappa$ be a class representative with $e_\kappa\in\T_N$, that is, $|\kappa|\le(N-1)/2$ for odd $N$ and $|\kappa|\le N/2-1$ for even $N$.
The histopolant contributes only $(\hat v_\kappa/\mu_\kappa)e_\kappa$ to this class, so $\hat e_{\kappa+mN}=\hat u_{\kappa+mN}$ for $m\neq0$, and Lemma~\ref{lem:A:alias} gives $\hat e_\kappa=\hat u_\kappa-\hat v_\kappa/\mu_\kappa=-\sigma(\kappa)^{-1}\sum_{m\neq0}(-1)^m\sigma(\kappa+mN)\,\hat u_{\kappa+mN}$.
The Cauchy-Schwarz inequality and the identity $\sum_{m\neq0}\sigma(\kappa+mN)^2=1-\sigma(\kappa)^2$ of Lemma~\ref{lem:A:alias} yield
\[
\sum_{m\in\Z}|\hat e_{\kappa+mN}|^2\le\Big(\frac{1-\sigma(\kappa)^2}{\sigma(\kappa)^2}+1\Big)\sum_{m\neq0}|\hat u_{\kappa+mN}|^2=\frac{1}{\sigma(\kappa)^2}\sum_{m\neq0}|\hat u_{\kappa+mN}|^2,
\]
where the last sum is the contribution of the class to $\|Q_Nu\|_{L^2}^2/(2\pi)$.
Since $\sinc$ is positive and decreasing on $[0,\pi/2]$, the bounds $|\kappa h/2|\le\pi(N-1)/(2N)$ for odd $N$ and $|\kappa h/2|\le\pi(N-2)/(2N)$ for even $N$ give $1/\sigma(\kappa)\le C_N$.

For even $N$ the Nyquist class $\frac N2+N\Z$ remains, which consists of the odd multiples $(2m+1)\frac N2$ with $m\in\Z$.
At these wavenumbers $e^{ikh}=-1$ and hence $\mu_k=2i/((2m+1)\pi)$, and we write $a_m=\hat u_{(2m+1)N/2}$, so that $a_0=\hat u_{N/2}$ and $a_{-1}=\hat u_{-N/2}$.
Since the Nyquist term of Eq.~\eqref{eq:histopolant} equals $\frac{\pi}{4i}\hat v_{N/2}(e_{N/2}-e_{-N/2})$, Lemma~\ref{lem:A:alias} gives
\[
\frac{\pi}{4i}\,\hat v_{N/2}=\sum_{m\in\Z}\frac{a_m}{2(2m+1)}=\frac{a_0-a_{-1}}{2}+\frac{\rho}{2},\qquad\rho=\sum_{m\neq0,-1}\frac{a_m}{2m+1},
\]
and hence the error coefficients $\hat e_{\pm N/2}=\frac12(a_0+a_{-1})\mp\frac12\rho$ and $\hat e_{(2m+1)N/2}=a_m$ for $m\neq0,-1$.
In this class $\T_N$ contains only $\sin(\frac N2x)$, so $Q_N$ keeps the modes with $m\neq0,-1$ and the component $(a_0+a_{-1})\cos(\frac N2x)$, whose squared norm is $\pi|a_0+a_{-1}|^2$.
The expansion of $\csc^2z$ at $z=\pi/2$ gives $\sum_{m\neq0,-1}(2m+1)^{-2}=\pi^2/4-2$, and the parallelogram identity and the Cauchy-Schwarz inequality yield
\begin{align*}
\sum_{m\in\Z}\big|\hat e_{(2m+1)N/2}\big|^2&=\frac{|a_0+a_{-1}|^2}{2}+\frac{|\rho|^2}{2}+\sum_{m\neq0,-1}|a_m|^2\\
&\le\frac{|a_0+a_{-1}|^2}{2}+\frac{\pi^2}{8}\sum_{m\neq0,-1}|a_m|^2.
\end{align*}
The right-hand side is at most $\pi^2/8$ times the contribution of the Nyquist class to $\|Q_Nu\|_{L^2}^2/(2\pi)$, which yields the constant $\pi/(2\sqrt2)$ and completes the bound in item~(i).

Both candidates for $C_N$ are smaller than $\pi/2$, because $\sinc>2/\pi$ on $[0,\pi/2)$, which contains the arguments above, and $\pi/(2\sqrt2)<\pi/2$.
For $N=64$ the two candidates are $1/\sinc(31\pi/64)=1.52354$ and $\pi/(2\sqrt2)=1.11072$, so that $C_{64}=1.52354$.
For sharpness with odd $N\ge3$, we take $\kappa_*=(N-1)/2$, for which $\sigma(\kappa_*)=\sigma_{\min}$, and let all coefficients vanish except $\hat u_{\kappa_*+mN}=(-1)^m\sigma(\kappa_*+mN)$ for $m\neq0$.
This defines a function in $L^2(\T)$ by Lemma~\ref{lem:A:alias}, the Cauchy-Schwarz inequality above becomes an equality, and $\|Q_Nu\|_{L^2}^2=2\pi(1-\sigma_{\min}^2)$ and $\|e\|_{L^2}^2=2\pi(1-\sigma_{\min}^2)/\sigma_{\min}^2$ give the ratio $C_N$.
Adding the conjugate coefficients at $-\kappa_*-mN$, a different class for odd $N$, yields a real function with the same ratio.

For item~(ii), the definition $\mathcal S_{\mathrm{REA}}=\Ph\circ\Ev\circ\Ih$ gives $\mathcal S_{\mathrm{REA}}(\Ph u)-\Ph\Ev u=\Ph\Ev\Ih\Ph u-\Ph\Ev u$, and the hypothesis and item~(i) bound its norm by $L_E\|u-\Ih\Ph u\|_{L^2}\le L_EC_N\|Q_Nu\|_{L^2}$.
For item~(iii), let $u,w\in M$ with $\Ph u=\Ph w$, so that $\Ih\Ph(u-w)=0$ and item~(i) applied to $u-w$ gives $\|u-w\|_{L^2}\le C_N\|Q_N(u-w)\|_{L^2}$.
The Lipschitz bound and the triangle inequality then give $\|\Ph\Ev u-\Ph\Ev w\|_h\le L_EC_N(\|Q_Nu\|_{L^2}+\|Q_Nw\|_{L^2})$, which is at most the right-hand side of item~(iii).
For item~(iv), all modes with $|k|<N/2$ belong to $\T_N$, and for even $N$ the cosine component above has the squared norm $\pi|a_0+a_{-1}|^2\le2\pi(|a_0|^2+|a_{-1}|^2)$, so that $\|Q_Nu\|_{L^2}^2\le2\pi\sum_{|k|\ge\lceil N/2\rceil}|\hat u_k|^2$.
For $s\ge0$, as is implicit for a Sobolev ball, the weight $(1+k^2)^s\ge\lceil N/2\rceil^{2s}$ on these modes gives $\|Q_Nu\|_{L^2}\le\lceil N/2\rceil^{-s}R$, and the weight $e^{2\tau|k|}\ge e^{2\tau\lceil N/2\rceil}$ gives $\|Q_Nu\|_{L^2}\le e^{-\tau\lceil N/2\rceil}R$ on the analytic class.
Multiplying these bounds by $L_EC_N$ as in item~(ii) gives the two estimates of item~(iv).

We record two consequences that make the constant $L_E$ and the link with Proposition~\ref{prop:closable} concrete.
First, if $u$ and $w$ solve Eq.~\eqref{eq:pde} with the same forcing under Assumption~\ref{ass:regularity}, then $\omega=u-w$ and $\varsigma=u+w$ satisfy $\omega_t+\frac12(\omega\varsigma)_x=-\nu(-\Delta)^{\alpha/2}\omega$.
Since $\int_\T\omega(\omega\varsigma)_x\dd x=\frac12\int_\T\omega^2\varsigma_x\dd x$ and $\int_\T\omega(-\Delta)^{\alpha/2}\omega\dd x\ge0$, we obtain $\frac{\mathrm{d}}{\mathrm{d}t}\|\omega\|_{L^2}^2\le\frac12\|\partial_x\varsigma\|_{L^\infty}\|\omega\|_{L^2}^2$, and Gronwall's inequality with $\|\Ph g\|_h\le\|g\|_{L^2}$ gives $\|\Ph\Ev u-\Ph\Ev w\|_h\le e^{\Lambda\Delta t/4}\|u-w\|_{L^2}$ with $\Lambda=\sup_t\|\partial_x(u+w)\|_{L^\infty}$.
Second, let $\mu$ be concentrated on one value of $\xi$ with $V=\Ph u$ and $\dfh\in L^2$, let the hypothesis of item~(ii) hold $\mu$-almost surely, and let $m=\mathcal S_{\mathrm{REA}}-\Sh$ be Borel.
Item~(ii) then bounds $\|\dfh(u)-m(\Ph u)\|_h=\|\Ph\Ev u-\Ph\Ev\Ih\Ph u\|_h$ by $L_EC_N\|Q_Nu\|_{L^2}$ almost surely, so $m(V)\in L^2$ when $\EE\|Q_Nu\|_{L^2}^2<\infty$.
Item~(v) of Proposition~\ref{prop:closable} in the norm $\|\cdot\|_h$ then gives the bound $L_E^2C_N^2\,\EE\|Q_Nu\|_{L^2}^2$ on the irreducible error.

\subsection{Proof of Proposition~\ref{prop:mass}}\label{app:proofs:mass}

In the coarse discrete Fourier variables, with $z_k=\Delta t\lambda_k$, $\varphi_1(z)=(e^z-1)/z$ and $\varphi_2(z)=(e^z-1-z)/z^2$, the ETDRK2 step reads
\[
\hat{\tilde v}_k=e^{z_k}\hat v^n_k+\Delta t\,\varphi_1(z_k)\hat{\mathcal N}^0_k,\qquad \hat v^{n+1}_k=\hat{\tilde v}_k+\Delta t\,\varphi_2(z_k)\big(\hat{\mathcal N}^1_k-\hat{\mathcal N}^0_k\big),
\]
where $\mathcal N^0=\mathcal N_\theta(v^n,t_n)$ and $\mathcal N^1=\mathcal N_\theta(\tilde v,t_{n+1})$.
Since $\lambda_0=0$, the limits $\varphi_1(0)=1$ and $\varphi_2(0)=\frac12$ give $\hat v^{n+1}_0=\hat v^n_0+\frac{\Delta t}2(\hat{\mathcal N}^0_0+\hat{\mathcal N}^1_0)$ for the mean.
The flux divergence has zero sum by telescoping, so $\hat{\mathcal N}_0$ is the mean of $\Ph f$ at the stage time, and $m_h=2\pi\hat v_0$ gives the claim.

\subsection{Proof of Proposition~\ref{prop:gap}}\label{app:proofs:gap}

The geometric series in Eq.~\eqref{eq:atom} gives $\hat u_n=\gamma e^{-n\delta}e^{-in\phi}$ for $n\ge1$, $\hat u_{-n}=\overline{\hat u_n}$ and $\hat u_0=0$.
By orthogonality, the error $u-s$ of every trigonometric polynomial $s$ of degree at most $N/2$ contains at least the modes of $u$ with $|n|\ge n_0=\lfloor N/2\rfloor+1$.
Summing the geometric tails therefore gives $\|u-s\|_{L^2}^2/\|u\|_{L^2}^2\ge\sum_{n\ge n_0}e^{-2n\delta}/\sum_{n\ge1}e^{-2n\delta}=e^{-2\delta\lfloor N/2\rfloor}\ge e^{-N\delta}$ for every such $s$.
Taking square roots and using $N\delta/2=\pi\delta/h$ then proves the claimed lower bound on the relative error.

\subsection{Proof of Lemma~\ref{lem:atoms}}\label{app:proofs:atoms}

Since $|z|=e^{-\delta}<1$, the series of Eq.~\eqref{eq:atom} converges absolutely and uniformly on $\T$, so it may be differentiated and integrated termwise, and Fourier multipliers act termwise.
Item~(i) is direct, and item~(ii) follows from $\frac{\mathrm{d}}{\mathrm{d}x}\log(1-z)=-iz/(1-z)$.
For item~(iii), the multipliers $i\,\sgn(n)|n|^{\alpha-1}$ and $|n|^\alpha$ act on $z^n=\bar a^ne^{inx}$ with $n\ge1$, and $\Ra$ maps real functions to real functions.
Item~(iv) is DLMF 25.12.12 with $s=1-\alpha$, valid because $|\mu|^2\le\delta^2+\pi^2<4\pi^2$, and item~(v) follows because the power series in Eq.~\eqref{eq:dlmf} is bounded on the disc $|\mu|\le\sqrt{1+\pi^2}$.

\subsection{Proof of Lemma~\ref{lem:dipole}}\label{app:proofs:dipole}

We write $K(z)=\sgn(z)|z|^{-\alpha}$, which is smooth for $z\neq0$ with $|K'(z)|=\alpha|z|^{-1-\alpha}$, and use the kernel form of $\Ra$ in Eq.~\eqref{eq:riesz-kernel}.
All three items rest on one estimate for a cell $J$ of length $h$ with center $y_J$ whose nearest point lies at the distance $d_J>0$ from $x$.
The mean value theorem gives $|K(y-x)-K(y_J-x)|\le\frac h2\,\alpha d_J^{-1-\alpha}$ for $y\in J$, so that every bounded function $w$ satisfies
\begin{equation}\label{eq:A:cell}
\Big|\int_JK(y-x)w(y)\dd y-K(y_J-x)\int_Jw(y)\dd y\Big|\le\frac\alpha2\,h^2d_J^{-1-\alpha}\|w\|_{L^\infty(J)},
\end{equation}
in which the second term on the left vanishes when $w$ has zero mean on $J$.

For item~(i) on $\R$, we split the kernel integral at $z=h$, so that the near part involves only the two cells adjacent to the face $x$.
There $|w(x+z)-w(x-z)|\le2z\|w'\|_{L^\infty}$ with the norm over these two cells, which bounds the near part by $2\kappa_\alpha\|w'\|_{L^\infty}\int_0^hz^{1-\alpha}\dd z=\frac{2\kappa_\alpha h^{2-\alpha}}{2-\alpha}\|w'\|_{L^\infty}$.
Since $w$ is bounded with compact support, the far part equals $\kappa_\alpha\int_{|y-x|\ge h}K(y-x)w(y)\dd y$ with an absolutely convergent integral, and since $x$ is a face its domain is a union of cells.
For each $n\ge1$ two of these cells have $d_J=nh$, and Eq.~\eqref{eq:A:cell} bounds their joint contribution by $\kappa_\alpha\alpha n^{-1-\alpha}h^{1-\alpha}\|w\|_{L^\infty}$, which sums to the far part of item~(i) because $\sum_{n\ge1}n^{-1-\alpha}=\zeta(1+\alpha)$.

For periodic $w$ and even $N$, the near part is unchanged, and we fold the far part $\kappa_\alpha\int_{|z|\ge h}K(z)w(x+z)\dd z$ onto one period by writing $z=s+2\pi m$ with $s\in(-\pi,\pi]$.
Periodicity of $w$ then gives
\[
\int_{|z|\ge h}K(z)w(x+z)\dd z=\int_{h\le|s|\le\pi}K(s)w(x+s)\dd s+\int_{-\pi}^{\pi}D(s)w(x+s)\dd s,
\]
where $D(s)=\sum_{m\ge1}\big[(2\pi m+s)^{-\alpha}-(2\pi m-s)^{-\alpha}\big]$ collects the periodic images, and $\int_{|z|\ge\pi}|z|^{-\alpha}\dd z<\infty$ justifies the exchange of sum and integral.
Since $x+\pi$ is a face for even $N$, the first integral runs over the cells with $d_J=nh$ and $1\le n\le N/2-1$, so it obeys the bound of the line case.
The function $D$ is smooth on $[-\pi,\pi]$, and $|D'(s)|=\alpha\sum_{m\ge1}\big[(2\pi m+s)^{-1-\alpha}+(2\pi m-s)^{-1-\alpha}\big]$ is even and convex, hence maximal at $s=\pm\pi$.
Its maximum equals $\alpha\big(2\sum_{n\ge1}(\pi(2n-1))^{-1-\alpha}-\pi^{-1-\alpha}\big)$, which is smaller than the constant $c_{\mathrm{per}}$ of item~(i).
Subtracting the value of $D$ at the center of each of the $N$ cells that cover $(-\pi,\pi]$ relative to $x$ bounds the second integral by $N\cdot\frac h2c_{\mathrm{per}}\cdot h\|w\|_{L^\infty}=\pi c_{\mathrm{per}}h\|w\|_{L^\infty}$, which gives the periodic version after multiplication by $\kappa_\alpha$.

For item~(ii), the function $u-v$ is bounded with zero mean on every cell, also when $u$ and $v$ are periodic and are viewed on the unfolded line.
The integral in item~(ii) is the far part of the second kernel form in Eq.~\eqref{eq:riesz-kernel}, because the constant $u(x)-v(x)$ integrates to zero against the odd kernel on $|y-x|\ge\rho$.
This integral converges absolutely, and since $x$ is a face and $\rho=m_0h$, the region $|y-x|\ge\rho$ is the union of the cells with $d_J=nh$ and $n\ge m_0$, two for each $n$.
Applying Eq.~\eqref{eq:A:cell} with zero cell means and the comparison $h(nh)^{-1-\alpha}\le\int_{(n-1)h}^{nh}z^{-1-\alpha}\dd z$, we obtain
\[
\begin{aligned}
\Big|\kappa_\alpha\int_{|y-x|\ge\rho}K(y-x)\,(u-v)(y)\dd y\Big|&\le\kappa_\alpha\alpha h\|u-v\|_{L^\infty}\int_{\rho-h}^\infty z^{-1-\alpha}\dd z\\
&=\kappa_\alpha h\|u-v\|_{L^\infty}(\rho-h)^{-\alpha},
\end{aligned}
\]
where $m_0\ge2$ ensures that $\rho-h$ is positive.

For item~(iii), let $J$ be a cell that is not adjacent to $x$, let $d=y_J-x$ be the signed distance from $x$ to its center, so that $|d|\ge3h/2$, and let $\bar w$ be the mean of $w$ on $J$.
Then $K(y_J-x)\int_Jw\dd y=h\bar w\,\sgn(d)|d|^{-\alpha}$, and since $d_J=|d|-h/2\ge\frac23|d|$, Eq.~\eqref{eq:A:cell} bounds the remainder by $\frac\alpha2(\frac32)^{1+\alpha}h^2|d|^{-1-\alpha}\|w\|_{L^\infty}$, and multiplication by $\kappa_\alpha$ gives item~(iii).

\subsection{Proof of Proposition~\ref{prop:poles}}\label{app:proofs:poles}

For item~(i), the once-integrated stationary equation is $\nu u'=(u^2-U^2)/2$, which the profile satisfies because $\tanh'=1-\tanh^2$.
Its poles are those of $\tanh\xi$ at $\xi=i\pi(k+\frac12)$, where $\tanh$ has residue one, which gives the residue $-U\cdot2\nu/U=-2\nu$ in the variable $x$.
Since $\max|u_x|=U^2/(2\nu)$ at the center of the front and $\Delta U=2U$, the front width defined in item~(i) equals $\ell=2\nu/U$.
For item~(ii), $(-\Delta)^{\alpha/2}$ is homogeneous of degree $\alpha$ under dilations, so $u=U\chi((x-x_0)/\ell_\alpha)$ solves the equation with parameters $(U,\nu)$ if and only if $\chi$ solves it with $(1,1)$ and $\ell_\alpha^{\alpha-1}=\nu/U$.
Uniqueness up to translation gives the first claim, and the dilation maps the strip of $\chi_\alpha$ onto that of $u$.

\subsection{Proof of Theorem~\ref{thm:consistency}}\label{app:proofs:consistency}

We may assume $\eta\le1$, as in item~(ii) of Lemma~\ref{lem:riesz}, and $u(t)\in C^1(\T)$ under Assumption~\ref{ass:regularity}, so $w=u-\mathcal Dv$ is bounded and H\"older continuous of order $\eta$.
By item~(ii) of Lemma~\ref{lem:riesz} the face values of $\Ra u$, $\Ra\mathcal Dv$ and $\Ra w$ are finite, and linearity of the kernel integral gives $\Ra w=\Ra u-\Ra\mathcal Dv$.
For item~(i), the identity $\Ph w=\Ph u-\Ph\mathcal Dv=v-v=0$ shows that the reconstruction error $w$ has zero mean on every cell.
Subtracting $\Phi_{\mathcal D}(v)+\Ph f$ from Eq.~\eqref{eq:semidiscrete} gives $r=-\frac1h\Dm\Delta F$ with $\Delta F$ evaluated at the faces, and the triangle inequality gives $\|r\|_\infty\le\frac2h\max_j|\Delta F_{j+1/2}|$.
For item~(ii), the advective part of $\Delta F$ is $\frac12(u^2-(u-w)^2)=w(u-\frac12w)$, whose modulus is at most $(\|u\|_{L^\infty}+\frac12\|w\|_{L^\infty})\|w\|_{L^\infty}$, and its fractional part is $-\nu\Ra w$.

For item~(iii), the bound $B_1$ holds when $\sum_k|k|^{\alpha-1}|\hat w_k|<\infty$, since $\Ra w$ is then continuous with coefficients $i\,\sgn(k)|k|^{\alpha-1}\hat w_k$ by Lemma~\ref{lem:riesz} and equals its absolutely convergent Fourier series.
The bound $B_2$ is item~(i) of Lemma~\ref{lem:dipole} for periodic functions and even $N$, which applies because $w$ is Lipschitz when $w'$ is bounded and has zero cell means by item~(i).
The bound $B_3$ is the interpolation inequality of Lemma~\ref{lem:interp} applied to the periodic extension of $w$, and each bound holds trivially when its right-hand side is infinite.

For item~(iv), let $\mathcal D=\Ih$, and let $\kappa$ run over the class representatives with $e_\kappa\in\T_N$ as in the proof of Theorem~\ref{thm:rea}.
That proof gives $\hat w_{\kappa+mN}=\hat u_{\kappa+mN}$ for $m\neq0$ and, since $|\sigma|\le1$, the bound $|\hat w_\kappa|\le\sigma(\kappa)^{-1}\sum_{m\neq0}|\hat u_{\kappa+mN}|\le C_N\sum_{m\neq0}|\hat u_{\kappa+mN}|$.
For $m\neq0$ we have $|\kappa+mN|\ge N-|\kappa|>|\kappa|$, so the weight $|\kappa|^{\alpha-1}$ is at most $|\kappa+mN|^{\alpha-1}$, and these wavenumbers satisfy $|\kappa+mN|\ge N/2$.
Summing over these classes, and adding the modes with $m\neq0$ once more, bounds their weighted sum of $|\hat w_k|$ by $1+C_N$ times the weighted sum of $|\hat u_k|$ over $|k|\ge N/2$.
For odd $N$ these classes cover all wavenumbers, and $C_N<\pi/2$ gives the first bound with the factor $1+\frac\pi2$.
For even $N$ the Nyquist class remains, in which $\hat w_{\pm N/2}=\frac12(a_0+a_{-1})\mp\frac12\rho$ and $\hat w_{(2m+1)N/2}=a_m$ for $m\neq0,-1$ as in the proof of Theorem~\ref{thm:rea}.
Since $|2m+1|\ge3$ in the definition of $\rho$, we have $|\rho|\le\frac13\sum_{m\neq0,-1}|a_m|$, while every wavenumber of this class has modulus at least $N/2$.
The weighted sum of $|\hat w_k|$ over the Nyquist class is therefore at most $\frac43$ times that of $|\hat u_k|$, and $\frac43<1+\frac\pi2$ completes the first bound for even $N$.
The same argument with the weight one bounds $\sum_k|\hat w_k|$, which dominates $\|w\|_{L^\infty}$ because a continuous function with an absolutely summable Fourier series equals that series.

For the last statement of item~(iv), let $u$ extend analytically to the strip $|\Imag x|<\tau$ with $\|u\|_\tau=\sup_{|\Imag x|<\tau}|u(x)|<\infty$.
For $|y|<\tau$, periodicity and Cauchy's theorem move the integral for $\hat u_k$ to the line $\Imag x=y$, which gives $|\hat u_k|\le\|u\|_\tau e^{ky}$ and, as $y\to-\sgn(k)\tau$, the estimate $|\hat u_k|\le\|u\|_\tau e^{-\tau|k|}$.
With $K=\lceil N/2\rceil$, a lower bound for $|k|$ in these sums, and $(K+j)^{\alpha-1}\le K^{\alpha-1}(1+j)^{\alpha-1}$ for integers $j\ge0$, we obtain
\[
\sum_{|k|\ge N/2}|\hat u_k|\le\frac{2\|u\|_\tau e^{-\tau K}}{1-e^{-\tau}},\qquad\sum_{|k|\ge N/2}|k|^{\alpha-1}|\hat u_k|\le2c_{\alpha,\tau}\|u\|_\tau K^{\alpha-1}e^{-\tau K},
\]
with $c_{\alpha,\tau}=\sum_{j\ge0}(1+j)^{\alpha-1}e^{-\tau j}$.
The unweighted sum, and with it $\|w\|_{L^\infty}$, is therefore of order $e^{-\tau N/2}$, and the weighted sum is of order $N^{\alpha-1}e^{-\tau N/2}$, hence of order $e^{-\tau'N/2}$ for every $\tau'<\tau$.
On the analytic class of Theorem~\ref{thm:rea}, the Cauchy-Schwarz inequality gives the same two orders.
\section{Closed forms and implementation of the rational atoms}\label{app:atoms}

This {section} derives the closed forms of Lemma~\ref{lem:atoms} and describes their implementation.

\subsection{Derivation of the closed forms}\label{app:atoms:derivation}

Throughout, the pole parameter is $a=e^{-\delta+i\phi}$ with depth $\delta>0$ and position $\phi\in\T$, and $z(x)=\bar a\,e^{ix}=e^{\mu(x)}$ with $\mu(x)=-\delta+i(x-\phi)$.
Since $|z|=e^{-\delta}<1$ on $\T$, the mean-free atom has the absolutely and uniformly convergent expansion
\begin{equation}\label{eq:B:series}
\psia(x)=\frac{z}{1-z}=\sum_{n\ge1}z^n=\sum_{n\ge1}e^{-n\delta}e^{in(x-\phi)},
\end{equation}
which justifies termwise differentiation, integration and Fourier multipliers.

\paragraph{Values, slopes and cell averages}
We have $\partial_xz=iz$ and $\partial_x\log(1-z)=-i\psia$, where $\Real(1-z)\ge1-e^{-\delta}>0$ keeps the principal logarithm holomorphic along $1-z(x)$.
The chain rule and the fundamental theorem of calculus therefore give the slope and the exact cell averages
\begin{align}
\psia'(x)&=\frac{i\,z(x)}{\big(1-z(x)\big)^2},\label{eq:B:slope}\\
(\Ph\psia)_j&=\frac1h\int_{x_{j-1/2}}^{\xf{j}}\psia\dd x=\frac{1}{ih}\Big[\log\big(1-z(x_{j-1/2})\big)-\log\big(1-z(\xf{j})\big)\Big].\label{eq:B:cellavg}
\end{align}

\paragraph{Riesz flux and fractional Laplacian}
The Riesz flux operator acts on $z^n=\bar a^ne^{inx}$, $n\ge1$, through the symbol $i\sgn(n)|n|^{\alpha-1}=in^{\alpha-1}$, so with $\Li_s(w)=\sum_{n\ge1}w^nn^{-s}$, absolutely convergent for $|w|<1$, Eq.~\eqref{eq:B:series} yields
\begin{equation}\label{eq:B:riesz}
\Ra\psia=i\sum_{n\ge1}n^{\alpha-1}z^n=i\Li_{1-\alpha}(z),\qquad(-\Delta)^{\alpha/2}\psia=\sum_{n\ge1}n^{\alpha}z^n=\Li_{-\alpha}(z).
\end{equation}
With $w\,\partial_w\Li_s(w)=\Li_{1}(w)$ and $\partial_xz=iz$ this gives $\partial_x\Ra\psia=-(-\Delta)^{\alpha/2}\psia$, as in Lemma~\ref{lem:riesz}, so cell averages of $-(-\Delta)^{\alpha/2}\psia$ are face differences of $\Ra\psia$ over $h$.
The symbol of $\Ra$ is odd and purely imaginary, so $\Ra$ commutes with taking real parts, and the face flux of $g_\theta=2\Real\sum_{j'}\sum_{m}\gamma_{j'm}\psi_{a_{j'm}}$ is
\begin{equation}\label{eq:B:gflux}
-\nu\,\Ra g_\theta(\xf{j})=2\nu\,\Real\sum_{j'=0}^{N-1}\sum_{m=1}^{p}\gamma_{j'm}\,(-i)\,\Li_{1-\alpha}\big(e^{\mu_{j'm}(\xf{j})}\big),
\end{equation}
where $\mu_{j'm}(x)=-\delta_{j'm}+i(x-\phi_{j'm})$ belongs to the atom $m$ of cell $j'$.

\paragraph{Sub-grid flux and operation count}
The sub-grid flux of Eq.~\eqref{eq:nrc_sg} is Eq.~\eqref{eq:B:gflux} minus $-\nu\Ra\Ih(\Ph g_\theta)$ at the face, which applies the face-flux multiplier of $\Ih$ to the cell averages of Eq.~\eqref{eq:B:cellavg}.
With $\hat v_k=N^{-1}\sum_{j}v_je^{-ikjh}$ and the Nyquist mode of $\Ih$ carried by $\sin(Nx/2)$, this multiplier reads
\begin{equation}\label{eq:B:ihflux}
-\Ra(\Ih v)(\xf{j})=\sum_{0<|k|<N/2}\hat v_k\,\frac{h|k|^\alpha}{1-e^{-ikh}}\,e^{ikjh}+\frac h2\Big(\frac N2\Big)^{\alpha}\hat v_{N/2}\,(-1)^j .
\end{equation}
The face values and face slopes of $\Ih v$ have the multipliers $ikh/(1-e^{-ikh})$ and $ik\cdot ikh/(1-e^{-ikh})$, and the Nyquist mode vanishes at the faces but adds $-\frac{\pi N}{4}\hat v_{N/2}(-1)^j$ to the face slope.
These multipliers agree to round-off with a fast Fourier transform of $\Ih v$ on 4096 points.
Each NRC step evaluates all $pN^2$ atom-face pairs twice, and a profile at batch size one attributes about one third of the time to polylogarithm sums, one quarter to other atom quantities and $29\%$ to the network.

\subsection{Evaluation of the polylogarithm}\label{app:atoms:polylog}

We evaluate $\Li_s(e^\mu)$ with the expansion~\cite[Eq.~25.12.12]{olver2026dlmfgen}, valid for $s\notin\{1,2,\dots\}$ and $|\mu|<2\pi$,
\begin{equation}\label{eq:B:dlmf}
\Li_s(e^\mu)=\Gamma(1-s)\,(-\mu)^{1}+\sum_{k=0}^{\infty}\frac{\zeta(s-k)}{k!}\,\mu^k .
\end{equation}
For the flux, $s=1-\alpha\in(-1,0)$ and the term $\Gamma(\alpha)(-\mu)^{-\alpha}$ carries the whole singularity, with the principal logarithm valid because $\Real(-\mu)=\delta>0$.

\paragraph{Convergence and truncation}
Wrapping $x-\phi$ to $[-\pi,\pi)$ leaves $z$ unchanged and gives $|\mu|^2\le\delta^2+\pi^2$, so $|\mu|<2\pi$ for $\delta<\sqrt3\,\pi$, and the head bound $\delta\le3.02h\approx0.296$ gives $|\mu|/(2\pi)\le0.503$.
By the functional equation of the zeta function, $|\zeta(s-k)|/k!$ behaves like $(2\pi)^{-k}$ times a power of $k$, so the terms decay geometrically with ratio $|\mu|/(2\pi)$.
We keep $K=32$ terms, precomputed to 30 digits in mpmath and summed by Horner's rule, and the truncation error against mpmath is at most $1.5\times10^{-9}$ relative for $\alpha\in\{1.2,1.5,1.8\}$.
This maximum occurs near $|x-\phi|=\pi$ for $\alpha=1.8$, and on random points the double-precision implementation has a maximal relative error of $4.0\times10^{-10}$.

\paragraph{Exact gradients}
Since $\partial_\mu w=w$ for $w=e^\mu$, the same identity gives $\partial_\mu\Li_s(e^\mu)=\Li_{1}(e^\mu)$, which reuses the expansion with $s-1=-\alpha$, again not a positive integer.
The custom backward function multiplies the adjoint by the conjugate of this derivative, following the conjugate Wirtinger convention of PyTorch, and $\partial\mu/\partial\delta=-1$, $\partial\mu/\partial\phi=-i$ make the gradients exact up to evaluation error.

\paragraph{Precision}
All arrays are in double precision, except the reference data in single precision and the Horner sums in complex64, whose relative error of up to $2.9\times10^{-7}$ against mpmath on random points dominates the truncation error.
Single-atom Riesz fluxes agree with a 40,000-term series to between $1\times10^{-7}$ and $5\times10^{-7}$.
In situ, we compare the closed-form sub-grid flux of NRC (seed 0) with a fast Fourier transform of $\Dth v$ on $2^{16}$ points, on test states at $t=10$ of eight trajectories per $\alpha$.
The maximal relative difference is $6\times10^{-6}$ to $1.5\times10^{-5}$ over the three orders, and Simpson quadrature gives $\max_j|(\Ph\Dth v)_j-v_j|\le2.9\times10^{-15}$ on the same states.
These deviations lie more than an order of magnitude below the rollout errors of about $3\times10^{-4}$ in Section~\ref{sec:accuracy}, so single precision limits the flux accuracy but not the reported comparisons.
The analytic gradient matches finite differences for double-precision sums, whereas in complex64 rounding dominates the quotients.

\paragraph{Cancellation for poles close to the axis}
Exact in this paper means exact up to the evaluation of the polylogarithm, which loses digits when near-axis poles with compensating amplitudes make the sub-grid flux much smaller than its terms.
By Eq.~\eqref{eq:B:dlmf} one atom has a flux of about $\Gamma(\alpha)\delta^{-\alpha}$ next to its pole, or $\delta^{1-\alpha}$ under the scaling $\gamma\propto\delta$, and the bound $\delta>0.02h$ limits the amplification.
The trained model of seed 0 keeps all poles at $\delta\ge0.287h$ on the test states, and we read the in-situ deviation of about $10^{-5}$ as single-precision rounding with moderate amplification.
We have not separated the rounding from the amplification in this deviation.

\paragraph{Takenaka-Malmquist product form}
The Takenaka-Malmquist system~\cite{qian2011algorithm,qian2011adaptive} multiplies normalized Szeg\H{o} kernels $(1-\bar a_ne^{ix})^{-1}$ by Blaschke products, and an NRC atom is such a kernel up to normalization and the constant term.
Our closed forms need simple poles, and converting the products to partial fractions, whose coefficients contain reciprocal pole differences, amplified rounding errors by $2.7\times10^{7}$ with 32 random poles and failed with 64.
The NRC head therefore outputs partial-fraction amplitudes directly, at the price of a non-unique representation by non-orthogonal atoms, which matters for interpreting single poles but not for the fluxes.

\subsection{Initialization of the head and inference pruning}\label{app:atoms:init}\label{app:atoms:pruning}

The head weights of Eq.~\eqref{eq:nrc_head} are scaled by $0.1$ after the default PyTorch initialization, so the initial decoder is a small perturbation of $\Ih$.
The only nonzero biases are the depth bias $-1.93$, giving $\delta\approx0.40h$ near the minimal front width of $0.38h$ to $0.48h$ in the training data, and the position biases $0.8\,(m-(p+1)/2)$, giving poles at $x_j\pm0.19h$.

Pruning atoms with peak amplitude $|\gamma|/\delta$ below $10^{-3}$ keeps a fraction $0.986$ to $0.998$ of them on validation states, against $0.59$, $0.73$ and $0.91$ above $10^{-2}$ for increasing $\alpha$ (Table~\ref{tab:extra}).
It changes the explicit face flux by at most $8.3\times10^{-5}$ relative and leaves $\errA$ of seed 0 essentially unchanged, $0.0329\%$ against $0.0328\%$.
At batch size one, pruning makes one explicit flux evaluation slower, $6.87$~ms against $5.82$~ms without pruning, and at batch size 24 the timings agree within two percent.
Almost all atoms stay active at thresholds that preserve the accuracy, and the dense batched evaluation does not profit from an irregular sparsity pattern.
A sparsity penalty during training is examined among the ablations of Section~\ref{sec:ablation}, and cheaper atom evaluations are discussed as future work in Section~\ref{sec:conclusions}.
\section{Training and evaluation details}\label{app:training}

This {section} collects the data generation, training settings, baselines and metrics behind Section~\ref{sec:coupled}, on $N=64$ cells with the coarse step $\Delta t=0.02$ unless stated otherwise.

\subsection{Data generation}\label{app:training:data}

\paragraph{Reference solutions}
The reference solver uses $N_f=1024$ Fourier points with the dealiasing mask $\chi_k=\mathbf{1}[|k|\le N_f/3]$ and integrates the linear part exactly through its symbol $-\nu|k|^\alpha$.
Its explicit part is $-\tfrac{ik}{2}\chi_k\widehat{(u^2)}_k+\hat f_k(t)$ with $u=\mathcal F^{-1}(\chi\hat u)$, where $\mathcal F$ denotes the discrete Fourier transform on the fine grid.
ETDRK4~\cite{cox2002exponential} advances this system with the step $\Delta t/16$, coefficients from contour integrals on 64 points of a circle~\cite{kassam2005fourth} and the mask applied after every step.
The coarse data are the exact cell averages of the fine trigonometric polynomial, obtained with the multiplier $(e^{ikh}-1)/(ikh)$ and sampling at every sixteenth point.
At the out-of-distribution orders, cell averages with $N_f=1024$ and $2048$ differ by at most $2.7\times10^{-7}$, the fine spectral tail stays below $1.9\times10^{-6}$ and the Courant number $\max|u|\Delta t/h$ lies between $0.24$ and $0.31$.
Over all orders used, the minimal front width lies between $0.38$ and $0.53$ coarse cells.
Initial data and forcing have zero mean as in Section~\ref{sec:problem}, so the total mass vanishes for all times.
NRC and the classical schemes evaluate the exact cell averages of $f$ at their stage times, whereas \texttt{c5\_fno} receives the exact space and time integral of $f$ over each step.

\paragraph{Coefficients and splits}
Orders outside the training set take $\nu$ from a rule that was fixed before any model result was available.
It interpolates $\log\nu$ piecewise linearly between the training pairs of Section~\ref{sec:setup} and extrapolates it linearly,
\begin{equation}\label{eq:C:nurule}
\log\nu(\alpha)=
\begin{cases}
\log0.3+\dfrac{\log0.12-\log0.3}{0.3}\,(\alpha-1.2), & \alpha\le1.5,\\[6pt]
\log0.12+\dfrac{\log0.05-\log0.12}{0.3}\,(\alpha-1.5), & \alpha>1.5,
\end{cases}
\end{equation}
which gives $\nu=0.407$, $0.221$, $0.0669$ and $0.0373$ at $\alpha=1.1$, $1.3$, $1.7$ and $1.9$.
Table~\ref{tab:splits} lists all splits, whose seeds are disjoint across splits and shared across orders within a split, so that comparisons across $\alpha$ within one split are paired.
The fresh seed-88 factors lie below, inside and above its training range.
Called with seed 66, the generator of the fresh splits reproduces the original split bit for bit.
Generating one training split takes from 76 to 79~s per order on the CPU, a validation split about 30~s and a test split about 83~s.

\begin{table}[!tbp]
\centering
\caption{Data splits of the benchmark, stored as exact cell averages on 64 cells after every coarse step. The size is the number of trajectories per order, and splits that share a data seed contain identical initial data and forcing.}
\label{tab:splits}
\footnotesize
\setlength{\tabcolsep}{4.5pt}
\begin{tabular}{lrrlll}
\toprule
Split & Seed & Size & Orders $\alpha$ & Diffusivity & Horizon \\
\midrule
train & 11 & 64 & 1.2, 1.5, 1.8 & $\nu(\alpha)$ & 25 \\
validation & 22 & 16 & 1.2, 1.5, 1.8 & $\nu(\alpha)$ & 25 \\
test & 33 & 24 & 1.2, 1.5, 1.8 & $\nu(\alpha)$ & 45 \\
order shift & 44 & 24 & 1.1, 1.3, 1.7, 1.9 & $\nu(\alpha)$ from Eq.~\eqref{eq:C:nurule} & 25 \\
forcing shift, $l_i\in\{7,8\}$ & 55 & 24 & 1.5 & $\nu(\alpha)$ & 25 \\
coefficient shift & 66 & 24 & 1.5 & $\nu\times\{0.7,2\}$ & 25 \\
diffusivity-diverse training & 77 & 64 & 1.2, 1.5, 1.8 & $\nu\times\{0.8,1.5\}$ & 25 \\
fresh coefficient shift & 88 & 24 & 1.2, 1.5, 1.8 & $\nu\times\{0.6,1.25,2.5\}$ & 25 \\
\bottomrule
\end{tabular}
\end{table}

\subsection{Classical baselines and their selection}\label{app:training:classical}

The classical baselines are selected from seven reconstructions, two Riemann solvers and three time integrators, always with the exact global fractional operator.
The reconstructions are piecewise constant states, MUSCL~\cite{vanleer1979towards} with the minmod, van Leer, monotonized central and superbee limiters~\cite{vanleer1977towards}, WENO5-Z~\cite{jiang1996efficient,borges2008improved} and the face values of $\Ih v$ with a central flux.
The Riemann solvers are the Rusanov~\cite{rusanov1962calculation} and exact Godunov~\cite{leveque2002finite} fluxes, and the integrators are ETDRK2~\cite{cox2002exponential} with the exact symbol, SSPRK3~\cite{shu1988efficient} and a second-order Heun method.
The validation score of a candidate is the median over the 48 validation trajectories of the time-mean relative $L^2$ error over 1000 steps from $t=5$.
C0 and C0+ are the MUSCL and WENO5-Z entries that minimize this score, with validation errors of $1.50\%$ and $0.407\%$ and the components listed in Section~\ref{sec:setup}.
SSPRK3 applies the exact operator explicitly in flux form, which is stable since $\Delta t\,\nu\,(N/2)^\alpha=0.38$, $0.43$ and $0.51$ at the training orders lie well inside its real stability interval of length about $2.51$.
The Heun method uses the exact space and time integral of the forcing over each step, and the NRC base with $g=0$, the central flux on $\Ih v$ with ETDRK2, has a validation error of $2.10\%$.
C0 with the local operator replaces the exact operator by a surrogate of half-width $r=4$ cells from truncated fractional centered differences~\cite{ortigueira2006riesz,celik2012crank}, with the symbol
\begin{equation}\label{eq:C:s2}
\lambda^{\mathrm{loc}}_k=-\nu h^{-\alpha}\sum_{m=1}^{4}|c_m|\,\big(2-2\cos(kmh)\big),\qquad c_m=\frac{(-1)^m\,\Gamma(\alpha+1)}{\Gamma(\frac\alpha2-m+1)\,\Gamma(\frac\alpha2+m+1)},
\end{equation}
which keeps constants in its kernel and reaches $39.1\%$ on validation.

\subsection{Training of NRC}\label{app:training:nrc}

Table~\ref{tab:hyper} {of the main text} lists the optimization settings of all learned models, and we add only what it and Section~\ref{sec:training} leave open.
The stage-A face flux $F^*_{j+1/2}=\big[\tfrac12\mathcal F^{-1}\big(\chi\,\widehat{(u^2)}\big)-\nu\Ra u\big](\xf{j})$ of Eq.~\eqref{eq:nrc_lossA} is computed from the dealiased fine state.
The normalizers are $\sigma_F=6.82\times10^{-3}$, $1.46\times10^{-2}$ and $2.07\times10^{-2}$ and $\sigma_U=1.20\times10^{-2}$, $2.15\times10^{-2}$ and $2.71\times10^{-2}$ for $\alpha=1.2$, $1.5$ and $1.8$.
For seed 0 the stage-A loss falls from $92.9$ to $4.2\times10^{-3}$ over the 1500 stage-A iterations.
Stage-B windows start at a step $n_0$ between 50 and $1250-H$, and the learning rate is $3\times10^{-4}$ times $0.1+0.45\big(1+\cos(\pi\min(1,i/I))\big)$ at stage-B iteration $i$ of $I=950$.
Checkpoints are scored as the classical candidates of \ref{app:training:classical}, but over a 200-step rollout and with non-finite values or values above 10 set to 10.
No logged run had a step with a non-finite loss or reached the wall-clock cap of 55 minutes, and test data never entered model selection.

The fourth input channel $(\ln(\Delta t\,\nu/h^\alpha)-1.2)/0.5$ is affine in the order and diffusivity channels, since $\ln(\Delta t\,\nu/h^\alpha)=\ln\Delta t+\ln\nu-\alpha\ln h$.
It is not centered either, with values of about $-7.1$, $-7.5$ and $-7.9$ at the three training orders.
Since $\nu$ is a function of $\alpha$ in the training data, only the diffusivity-diverse set, with $\nu\times\{0.8,1,1.5\}$ at every order, varies the two parameters independently.

\subsection{Learned baselines}\label{app:training:baselines}

The face-flux models write $F_{j+1/2}=F^{\mathrm{base}}_{j+1/2}(v)+\sigma^{\mathrm{bb}}_F\,\mathrm{head}([e_j,e_{j+1}])$, where $e$ is the backbone output, so each CNN face correction sees 18 cells.
Their head is a pointwise convolution with 81 parameters, and $\sigma^{\mathrm{bb}}_F$ is the mean over orders of the root-mean-square face-flux error of the C0 base.
All face-flux models are integrated by ETDRK2 with their linear operator treated exponentially, whatever integrator their base scheme uses on its own.
The one-step stage A minimizes $\langle((S_\theta(v^n)-v^{n+1})/\sigma_\delta(\alpha))^2\rangle$ over random orders, trajectories and steps, with $\sigma_\delta$ the root-mean-square one-step defect of C0, while the flux stage A uses the NRC target and $\sigma_F$.
With $S_{\mathrm{C0}}$ one step of C0 and scalars $\sigma$ fixed from the training data, \texttt{c1\_point} sets $v^{n+1}=S_{\mathrm{C0}}(v^n)+\sigma\,\mathrm{CNN}(v^n)$.
Trained on the ratio target of \ref{app:ratio}, \texttt{c1\_mult} sets $v^{n+1}=S_{\mathrm{C0}}(v^n)\odot\big(1+\sigma\,\mathrm{CNN}(v^n)\big)$ instead, a multiplicative correction of the same step.
The surrogate \texttt{c5\_fno} updates $v^{n+1}=v^n+F^n+\mathcal Z(\sigma\,\mathrm{FNO})$, where the network sees $v^n$, $\alpha$, $\nu$ and the exact forcing integral $F^n$, and $\mathcal Z$ removes the mean.

In block (B) of Table~\ref{tab:barsinai} NRC trains by flux matching only, and the face stencil of its last line trains for 2500 one-step and 200 unrolled iterations.
Evaluating on a GPU reproduces the stored CPU values of Table~\ref{tab:ood} to a relative $3\times10^{-5}$.

\paragraph{NRC ablations}
The ablations of Section~\ref{sec:ablation} change one component of NRC at a time under the short curriculum, with \texttt{nrc\_short} as the reference.
The local-surrogate variant uses Eq.~\eqref{eq:C:s2} in the integrator and in the stage-A flux but keeps the sub-grid Riesz flux exact, and the variant without projection sets $\mathcal D v=\Ih v+g_\theta$.
Others remove the sub-grid Riesz flux, zero the three parameter channels, use $p=1$ pole per cell with 25,165 parameters, or add the stage-A penalty $0.02\langle|\gamma|/\delta\rangle/0.6$.
The sub-grid field variant replaces the atoms by 16 nodal values per cell, scaled by $0.06$ and interpolated trigonometrically on the fine grid, with face quantities from fast Fourier transforms and 25,616 parameters.
The receptive-field sweep uses kernels $(3,3,1,1)$ for 5 cells with 9,009 parameters, and dilations $(1,1,2,4)$ and $(2,2,4,8)$ for 33 and 65 cells.

\subsection{Evaluation protocol and metrics}\label{app:training:metrics}

\paragraph{Rollouts and errors}
Besides the 2000-step test rollouts from $t=5$, models are rolled out for 200 steps from $t_0=10$ and $15$ on the test split and for 1000 steps from $t=5$ on the shifted splits.
Earlier single-seed ablations used a reduced protocol with the first 12 of the 24 trajectories per order, the window $t\in(5,25]$ only and a single 200-step start.
Non-finite values of the relative error $e(t_n)$ of Section~\ref{sec:setup} are set to $+\infty$, and $\errA$ and $\errB$ average it over $251\le n\le1250$ and $1251\le n\le2250$.
The error $\mathrm{err}_{200}$ is $e$ after 200 steps from the starts $t_0=5$, $10$ and $15$, which gives 216 values per seed over the three orders.

\paragraph{Observables}
All observables are evaluated from $t=5$ to $t=25$ on the same rollouts as $\errA$.
The spectrum error is $\frac1{32}\sum_{k=1}^{32}|\log_{10}E_V(k)-\log_{10}E_R(k)|$, where $E(k)$ is the time mean of $|\hat V_k/64|^2$, doubled for $1\le k\le31$ to account for the negative wavenumbers.
The energy error is the relative difference of the time and space means of $\frac12V^2$ and $\frac12R^2$, and the front amplitude error is the relative difference of the time means of $J(t)=-\min_j(v_{j+1}-v_j)$.
The front position error is the time mean of the circular distance $\min(|p_V-p_R|,64-|p_V-p_R|)$ in cells, where $p(t)=\arg\min_j(v_{j+1}-v_j)$ locates the descending jump of maximal size.
The mass drift is $\max_n|\sum_jV^n_j-\sum_jR^n_j|\,h$ over the full test rollout, maximized over trajectories, orders and seeds.
A blow-up is a trajectory with a non-finite state in a window, and the count adds the infinite values of $\errB$ to the non-finite values in the first window, which never coincided on the test split.

\paragraph{Aggregation and timing}
The interpolation and extrapolation errors pool the 48 trajectories at $\alpha\in\{1.3,1.7\}$ and $\alpha\in\{1.1,1.9\}$, and seeds and trajectories are otherwise aggregated as in Section~\ref{sec:setup}.
The {GPU timings quoted in Section~\ref{sec:future}} were taken on an otherwise idle machine, with each group of configurations in a fresh process and one process at a time.
From the test states at $\alpha=1.5$ and $t=5$, with batches of 1 and 24 trajectories, we time only the step, without gradient tracking and without collecting outputs.
Each time is the median of five repeats of 50 coarse steps after two warm-up steps, and on the GPU the device is synchronized before the clock starts and stops.
A classical solver with a smaller step takes $\mathrm{round}(0.02/\Delta t)$ substeps per coarse step, so that every repeat covers one unit of simulated time.
{The GPU runs keep two host threads, all timings use PyTorch 2.8.0 and NumPy 2.1.3 on the workstation of \ref{app:training:hardware}, an Intel Core Ultra 9 285K with an NVIDIA GeForce RTX 5090 GPU, and no implementation uses CUDA graphs or kernel fusion.}

\subsection{Hardware and software}\label{app:training:hardware}

All runs of the main comparison used PyTorch on the CPU only, with two threads per process and two training processes running in parallel.
The logs do not record the processor, but the runs were made on the workstation holding the code, an Intel Core Ultra 9 285K with 24 cores shared with other jobs.
The three-seed ablations, the receptive-field sweep and the diffusivity-diverse seeds 0, 1 and 2 were trained on the NVIDIA GeForce RTX 5090 of this workstation, the last with the full curriculum.
Each GPU run builds the model on the CPU with the initial weights of the CPU run of the same seed but is not bit-identical to it.
Re-evaluating the stored checkpoints with the current code reproduces the validation score of NRC seed 0 to a relative difference of $7\times10^{-13}$ and the three NRC test medians exactly.
\section{Additional results}\label{app:additional}

This {section} gives five tables that resolve the pooled results of Section~\ref{sec:coupled}, and each subsection states only what the full table adds.

\subsection{Per-order errors}\label{app:add:peralpha}

\begin{table}[!tbp]
\centering
\caption{Test errors of NRC and of the baselines of Table~\ref{tab:main} for each fractional order, with 24 trajectories per order and the errors $\errA$ and $\errB$ of that table. Single runs report the median with the 10th and 90th percentiles, and three-seed rows report the median over seeds with the seed range. The lower block was trained with four times as many iterations.}
\label{tab:peralpha}
\footnotesize
\setlength{\tabcolsep}{3pt}
\resizebox{\textwidth}{!}{%
\begin{tabular}{lrcccccc}
\toprule
 & & \multicolumn{3}{c}{$\errA$ (\%), $t\in(5,25]$} & \multicolumn{3}{c}{$\errB$ (\%), $t\in(25,45]$} \\
\cmidrule(lr){3-5}\cmidrule(lr){6-8}
Method & Seeds & $\alpha=1.2$ & $\alpha=1.5$ & $\alpha=1.8$ & $\alpha=1.2$ & $\alpha=1.5$ & $\alpha=1.8$ \\
\midrule
\multicolumn{8}{l}{\emph{Classical schemes}} \\
C0 & n/a & 1.22 & 1.55 & 1.92 & 1.24 & 1.48 & 1.68 \\
\hspace{0.5em}{\scriptsize MUSCL-MC, Godunov} & & {\scriptsize(1.00 to 1.51)} & {\scriptsize(1.22 to 1.97)} & {\scriptsize(1.35 to 2.41)} & {\scriptsize(0.943 to 1.50)} & {\scriptsize(1.15 to 1.83)} & {\scriptsize(1.20 to 2.26)} \\
C0+ & n/a & 0.365 & 0.439 & 0.600 & 0.310 & 0.381 & 0.477 \\
\hspace{0.5em}{\scriptsize WENO5-Z, Rusanov} & & {\scriptsize(0.265 to 0.430)} & {\scriptsize(0.326 to 0.503)} & {\scriptsize(0.421 to 0.723)} & {\scriptsize(0.195 to 0.411)} & {\scriptsize(0.297 to 0.487)} & {\scriptsize(0.365 to 0.627)} \\
NRC base, $g=0$ & n/a & 1.55 & 4.16 & 6.93 & 1.64 & 3.26 & 4.59 \\
\hspace{0.5em}{\scriptsize central flux on $\Ih v$} & & {\scriptsize(0.773 to 3.80)} & {\scriptsize(1.74 to 7.19)} & {\scriptsize(2.05 to 16.6)} & {\scriptsize(0.367 to 2.70)} & {\scriptsize(1.48 to 7.27)} & {\scriptsize(1.99 to 19.0)} \\
C0, local operator & n/a & 81.3 & 39.5 & 9.71 & 95.4 & 44.8 & 11.6 \\
\hspace{0.5em}{\scriptsize Eq.~\eqref{eq:C:s2}} & & {\scriptsize(64.2 to 98.7)} & {\scriptsize(27.6 to 47.9)} & {\scriptsize(7.37 to 12.7)} & {\scriptsize(69.7 to 119)} & {\scriptsize(33.3 to 61.3)} & {\scriptsize(7.76 to 18.2)} \\
\midrule
\multicolumn{8}{l}{\emph{Learned baselines}} \\
\texttt{c1\_point} & 1 & 2.82 & 3.76 & 5.92 & 7.72 & 9.94 & 10.2 \\
\hspace{0.5em}{\scriptsize pointwise additive} & & {\scriptsize(1.61 to 4.68)} & {\scriptsize(2.18 to 7.87)} & {\scriptsize(3.47 to 8.73)} & {\scriptsize(4.36 to 13.8)} & {\scriptsize(3.25 to 17.3)} & {\scriptsize(5.52 to 20.0)} \\
\texttt{c1\_mult} & 1 & 7.60 & 11.2 & 18.9 & 26.2 & 19.7 & 32.0 \\
\hspace{0.5em}{\scriptsize pointwise multiplicative} & & {\scriptsize(4.23 to 18.4)} & {\scriptsize(7.83 to 18.9)} & {\scriptsize(13.0 to 24.8)} & {\scriptsize(8.01 to 41.7)} & {\scriptsize(11.6 to 34.1)} & {\scriptsize(17.5 to 45.4)} \\
\texttt{c2\_bb\_s2} & 1 & 10.4 & 9.13 & 2.00 & 12.5 & 11.4 & 2.18 \\
\hspace{0.5em}{\scriptsize flux CNN, local operator} & & {\scriptsize(7.03 to 15.5)} & {\scriptsize(5.30 to 11.9)} & {\scriptsize(1.19 to 2.95)} & {\scriptsize(6.97 to 21.8)} & {\scriptsize(7.45 to 16.8)} & {\scriptsize(1.56 to 5.35)} \\
\texttt{c4\_fno\_s0} & 3 & 0.478 & 0.697 & 1.04 & 0.470 & 0.691 & 1.11 \\
\hspace{0.5em}{\scriptsize flux FNO on C0} & & {\scriptsize[0.474, 0.531]} & {\scriptsize[0.680, 0.762]} & {\scriptsize[0.994, 1.06]} & {\scriptsize[0.457, 0.505]} & {\scriptsize[0.688, 0.719]} & {\scriptsize[0.975, 1.13]} \\
\texttt{c5\_fno} & 1 & 4.89 & 7.67 & 10.1 & 4.92 & 8.35 & 11.2 \\
\hspace{0.5em}{\scriptsize FNO surrogate} & & {\scriptsize(4.02 to 5.73)} & {\scriptsize(5.74 to 9.85)} & {\scriptsize(7.20 to 13.0)} & {\scriptsize(4.28 to 6.17)} & {\scriptsize(6.49 to 12.8)} & {\scriptsize(8.76 to 16.2)} \\
\texttt{bb\_s0} & 1 & 0.338 & 0.452 & 0.614 & 0.340 & 0.485 & 0.742 \\
\hspace{0.5em}{\scriptsize flux CNN on C0} & & {\scriptsize(0.298 to 0.369)} & {\scriptsize(0.374 to 0.504)} & {\scriptsize(0.481 to 0.804)} & {\scriptsize(0.297 to 0.377)} & {\scriptsize(0.377 to 0.565)} & {\scriptsize(0.480 to 1.02)} \\
\texttt{bb\_weno} & 1 & 0.135 & 0.224 & 0.345 & 0.121 & 0.236 & 0.340 \\
\hspace{0.5em}{\scriptsize flux CNN on C0+} & & {\scriptsize(0.106 to 0.159)} & {\scriptsize(0.177 to 0.269)} & {\scriptsize(0.242 to 0.469)} & {\scriptsize(0.0977 to 0.171)} & {\scriptsize(0.195 to 0.311)} & {\scriptsize(0.271 to 0.389)} \\
\texttt{bb\_weno\_fluxA} & 3 & 0.0910 & 0.153 & 0.275 & 0.0899 & 0.149 & 0.301 \\
\hspace{0.5em}{\scriptsize on C0+, NRC flux target} & & {\scriptsize[0.0885, 0.102]} & {\scriptsize[0.153, 0.168]} & {\scriptsize[0.252, 0.286]} & {\scriptsize[0.0870, 0.0956]} & {\scriptsize[0.146, 0.167]} & {\scriptsize[0.269, 0.312]} \\
\texttt{bb\_central\_fluxA} & 1 & 0.132 & 0.228 & 0.297 & 0.126 & 0.195 & 0.272 \\
\hspace{0.5em}{\scriptsize on NRC base, NRC target} & & {\scriptsize(0.0951 to 0.189)} & {\scriptsize(0.154 to 0.299)} & {\scriptsize(0.182 to 0.456)} & {\scriptsize(0.0825 to 0.178)} & {\scriptsize(0.149 to 0.268)} & {\scriptsize(0.194 to 0.403)} \\
\midrule
\multicolumn{8}{l}{\emph{Proposed method}} \\
NRC & 3 & 0.0189 & 0.0304 & 0.0439 & 0.0168 & 0.0281 & 0.0418 \\
\hspace{0.5em}{\scriptsize rational atoms} & & {\scriptsize[0.0172, 0.0204]} & {\scriptsize[0.0226, 0.0349]} & {\scriptsize[0.0299, 0.0491]} & {\scriptsize[0.0167, 0.0201]} & {\scriptsize[0.0214, 0.0350]} & {\scriptsize[0.0279, 0.0493]} \\
\midrule
\multicolumn{8}{l}{\emph{Four times the training iterations, single seed}} \\
\texttt{c5\_fno\_long} & 1 & 0.843 & 1.61 & 2.74 & 0.824 & 1.57 & 2.77 \\
\hspace{0.5em}{\scriptsize 4 times the iterations} & & {\scriptsize(0.662 to 1.06)} & {\scriptsize(1.20 to 1.95)} & {\scriptsize(1.73 to 3.87)} & {\scriptsize(0.635 to 1.40)} & {\scriptsize(1.29 to 2.30)} & {\scriptsize(2.00 to 4.54)} \\
\texttt{bb\_weno\_fluxA\_long} & 1 & 0.0491 & 0.0906 & 0.157 & 0.0488 & 0.0892 & 0.174 \\
\hspace{0.5em}{\scriptsize 4 times the iterations} & & {\scriptsize(0.0414 to 0.0679)} & {\scriptsize(0.0762 to 0.109)} & {\scriptsize(0.125 to 0.205)} & {\scriptsize(0.0373 to 0.0647)} & {\scriptsize(0.0767 to 0.111)} & {\scriptsize(0.117 to 0.255)} \\
\bottomrule
\end{tabular}}
\end{table}
Table~\ref{tab:peralpha} splits the test errors of Table~\ref{tab:main} by order and adds $\errB$, and only the two rows with the local operator have a lower median $\errA$ at $\alpha=1.8$ than at $\alpha=1.2$.
The gap of NRC to C0+ shrinks with the order, whereas its gap in median $\errA$ to \texttt{bb\_weno\_fluxA} widens from 4.8 to 6.3 times over three seeds.
For the single longer run of this baseline, the factor grows from 2.6 to 3.6.

\subsection{Ablations per order}\label{app:add:ablation}

\begin{table}[!tbp]
\centering
\caption{Errors $\errA$ in percent of the GPU-trained three-seed ablations and of the receptive-field sweep for each order, under the short curriculum. Entries are medians over seeds of the per-seed median over 24 trajectories, and daggers mark the extrapolated orders $1.1$ and $1.9$. The spread is the largest ratio between the highest and the lowest seed value over the seven columns of a row. The variants follow Table~\ref{tab:ablation}, the stage A rows evaluate the end of stage A, and no trajectory blew up.}
\label{tab:ablation_full}
\footnotesize
\setlength{\tabcolsep}{3.2pt}
\begin{tabular}{lcccccccc}
\toprule
 & \multicolumn{3}{c}{Test orders $\alpha$} & \multicolumn{4}{c}{Shifted orders $\alpha$} & \\
\cmidrule(lr){2-4}\cmidrule(lr){5-8}
Configuration & $1.2$ & $1.5$ & $1.8$ & $1.1^{\dagger}$ & $1.3$ & $1.7$ & $1.9^{\dagger}$ & Spread \\
\midrule
\multicolumn{9}{l}{\emph{NRC reference and ablations, short curriculum}} \\
NRC reference & 0.0219 & 0.0365 & 0.0511 & 0.0252 & 0.0270 & 0.0413 & 0.0561 & 1.32 \\
Stage A only & 0.0290 & 0.0506 & 0.0721 & 0.0325 & 0.0370 & 0.0555 & 0.0811 & 1.26 \\
(i) local surrogate $s_2$ & 12.0 & 11.2 & 3.03 & 10.4 & 9.17 & 4.94 & 2.23 & 2.07 \\
(ii) no projection $\Ph$ & 0.224 & 0.384 & 0.559 & 0.202 & 0.249 & 0.470 & 0.652 & 1.54 \\
(iii) no sub-grid Riesz flux & 0.316 & 0.502 & 0.638 & 0.254 & 0.349 & 0.594 & 0.727 & 1.28 \\
(iv) black-box face flux & 0.166 & 0.325 & 0.444 & 0.130 & 0.189 & 0.403 & 0.533 & 1.48 \\
(v) sub-grid field head & 0.105 & 0.240 & 0.354 & 0.111 & 0.130 & 0.316 & 0.518 & 1.58 \\
(vi) one pole per cell & 0.0334 & 0.0539 & 0.0876 & 0.0309 & 0.0378 & 0.0749 & 0.0966 & 1.37 \\
(vii) network sees only $v$ & 0.0627 & 0.0743 & 0.123 & 0.0613 & 0.0688 & 0.103 & 0.156 & 1.30 \\
(viii) sparsity penalty & 0.0235 & 0.0348 & 0.0555 & 0.0293 & 0.0284 & 0.0424 & 0.0606 & 2.44 \\
\midrule
\multicolumn{9}{l}{\emph{NRC, receptive field in cells}} \\
\hspace{0.5em}5 & 0.0775 & 0.110 & 0.131 & 0.0723 & 0.0888 & 0.114 & 0.145 & 1.85 \\
\hspace{0.5em}17, reference & 0.0219 & 0.0365 & 0.0511 & 0.0252 & 0.0270 & 0.0413 & 0.0561 & 1.32 \\
\hspace{0.5em}33 & 0.0427 & 0.0665 & 0.0990 & 0.0405 & 0.0502 & 0.0838 & 0.105 & 2.34 \\
\hspace{0.5em}65 & 0.150 & 0.290 & 0.350 & 0.120 & 0.206 & 0.338 & 0.416 & 1.42 \\
\midrule
\multicolumn{9}{l}{\emph{\texttt{c2\_bb\_s2}, flux CNN on C0 with the local operator, receptive field in cells}} \\
\hspace{0.5em}5 & 48.5 & 21.2 & 5.70 & 52.2 & 34.2 & 11.6 & 3.63 & 1.51 \\
\hspace{0.5em}17 & 11.8 & 10.8 & 2.28 & 11.7 & 8.30 & 4.40 & 3.81 & 1.66 \\
\hspace{0.5em}33 & 5.03 & 4.05 & 2.67 & 9.03 & 5.26 & 2.72 & 3.40 & 1.64 \\
\hspace{0.5em}65 & 0.878 & 1.87 & 2.23 & 7.85 & 5.68 & 2.18 & 2.99 & 1.29 \\
\hspace{0.5em}17, stage A state & 12.9 & 11.6 & 2.97 & 12.4 & 9.90 & 4.21 & 4.79 & 1.69 \\
\bottomrule
\end{tabular}
\end{table}
Table~\ref{tab:ablation_full} resolves the three-seed ablations of Section~\ref{sec:ablation} by order, including the four shifted orders.
Per order, the median $\errA$ of the variants (i) to (v) stays at least 4.4 times above the reference, a smaller margin than the pooled factor of {6.5}.
A field of 65 instead of 17 cells lowers the median $\errA$ of \texttt{c2\_bb\_s2} 13.4 times at $\alpha=1.2$, but only 1.5 times at the extrapolated $\alpha=1.1$.
{For NRC, the sweep of Section~\ref{sec:ablation} gives a lower median $\errA$ at the default of 17 cells than at 33 and 65 cells, and we have two hypotheses for this, neither of which we have tested.
Optimization may be harder for the dilated networks, since at 65 cells validation selected the stage-A state in all three seeds and at 33 cells one seed reaches 0.123\% against 0.0604\% and 0.0683\% for the other two.
In addition, the dilations $(2,2,4,8)$ at 65 cells connect each cell only with cells at even offsets, so its poles are placed without the averages of its two neighbors.}

\subsection{Paired comparisons}\label{app:add:paired}

\begin{table}[!tbp]
\centering
\caption{Paired ratios $r=\errA^{\mathrm{NRC}}/\errA^{\mathrm{base}}$ of NRC over each baseline, on the 72 test trajectories and on the 96 trajectories at $\alpha\in\{1.1,1.3,1.7,1.9\}$. The median is taken over seed pairs of the per-pair median, and the 10th and 90th percentiles over all pairs pooled. Blow-ups of the NRC base with $g=0$ count as $r=0$. NRC is better on every trajectory of every pair, except against its diffusivity-diverse variant, where the median fractions with $r<1$ are 0.778 on the test split and 0.510 on the order shift.}
\label{tab:paired}
\footnotesize
\setlength{\tabcolsep}{3.5pt}
\begin{tabular}{lrcccccc}
\toprule
 & & \multicolumn{3}{c}{Test, 72 per pair} & \multicolumn{3}{c}{Order shift, 96 per pair} \\
\cmidrule(lr){3-5}\cmidrule(lr){6-8}
Baseline & Pairs & Median & 10th & 90th & Median & 10th & 90th \\
\midrule
\multicolumn{8}{l}{\emph{Classical schemes}} \\
C0 & 3 & 0.0210 & 0.0121 & 0.0278 & 0.0216 & 0.0133 & 0.0290 \\
C0+ & 3 & 0.0710 & 0.0470 & 0.0910 & 0.0772 & 0.0486 & 0.101 \\
NRC base, $g=0$ & 3 & 0.00874 & 0.00425 & 0.0188 & 0.0122 & 0.00439 & 0.0272 \\
C0, local operator & 3 & 0.000803 & 0.000198 & 0.00493 & 0.000889 & 0.000231 & 0.0119 \\
\midrule
\multicolumn{8}{l}{\emph{Learned baselines}} \\
\texttt{c1\_point} & 3 & 0.00723 & 0.00372 & 0.0143 & 0.00914 & 0.00455 & 0.0154 \\
\texttt{c1\_mult} & 3 & 0.00277 & 0.00118 & 0.00443 & 0.00258 & 0.00121 & 0.00458 \\
\texttt{c2\_bb\_s2} & 3 & 0.00385 & 0.00159 & 0.0250 & 0.00545 & 0.00170 & 0.0183 \\
\texttt{c4\_fno\_s0} & 9 & 0.0414 & 0.0269 & 0.0573 & 0.0458 & 0.0286 & 0.0645 \\
\texttt{c5\_fno} & 3 & 0.00416 & 0.00277 & 0.00542 & 0.00422 & 0.00289 & 0.00558 \\
\texttt{bb\_s0} & 3 & 0.0667 & 0.0422 & 0.0896 & 0.0675 & 0.0439 & 0.0937 \\
\texttt{bb\_weno} & 3 & 0.139 & 0.0888 & 0.176 & 0.134 & 0.0824 & 0.217 \\
\texttt{bb\_weno\_fluxA} & 9 & 0.193 & 0.122 & 0.248 & 0.173 & 0.107 & 0.283 \\
\texttt{bb\_central} & 3 & 0.0884 & 0.00867 & 0.134 & 0.104 & 0.0294 & 0.171 \\
\texttt{bb\_central\_fluxA} & 3 & 0.151 & 0.102 & 0.193 & 0.172 & 0.110 & 0.265 \\
\midrule
\multicolumn{8}{l}{\emph{Learned baselines with four times the iterations}} \\
\texttt{c5\_fno\_long} & 3 & 0.0194 & 0.0117 & 0.0280 & 0.0104 & 0.00474 & 0.0160 \\
\texttt{bb\_weno\_fluxA\_long} & 3 & 0.337 & 0.208 & 0.452 & 0.330 & 0.189 & 0.548 \\
\midrule
\multicolumn{8}{l}{\emph{NRC trained at $\nu\times\{0.8,1,1.5\}$}} \\
NRC, $\nu$-diverse & 3 & 0.845 & 0.640 & 1.12 & 0.986 & 0.657 & 1.27 \\
\bottomrule
\end{tabular}
\end{table}
Table~\ref{tab:paired} gives the paired ratios of NRC to every baseline, of which Section~\ref{sec:accuracy} quotes three, and the 90th percentile of the ratio stays at or below 0.548 on both splits.
The exception is the diffusivity-diverse model, a single seed trained {on three diffusivities per order}, against which NRC has median ratios of 0.845 on the test split and 0.986 on the order shift.

\subsection{Pole statistics for every seed}\label{app:add:interp}

\begin{table}[!tbp]
\centering
\caption{Position and depth of the learned poles at the sub-cell fronts of Section~\ref{sec:mechanism}, for three NRC seeds, four activity criteria and the three test orders. The first block gives the fraction of fronts whose nearest active pole lies within $h/4$, with $h/2$ in parentheses. The second block gives the median depth $\delta$ of that pole divided by the front width $\ell$. A snapshot without an active atom counts as a miss.}
\label{tab:interp}
\footnotesize
\setlength{\tabcolsep}{4pt}
\begin{tabular}{lcccccc}
\toprule
 & \multicolumn{3}{c}{Pole within $h/4$ (within $h/2$)} & \multicolumn{3}{c}{Median $\delta/\ell$} \\
\cmidrule(lr){2-4}\cmidrule(lr){5-7}
 & $\alpha=1.2$ & $\alpha=1.5$ & $\alpha=1.8$ & $\alpha=1.2$ & $\alpha=1.5$ & $\alpha=1.8$ \\
\midrule
\multicolumn{7}{l}{\emph{Active if the peak amplitude $|\gamma|/\delta\ge0.02$}} \\
\hspace{0.5em}Seed 0 & 0.884 (1.00) & 0.930 (1.00) & 0.933 (1.00) & 0.848 & 0.735 & 0.824 \\
\hspace{0.5em}Seed 1 & 0.853 (1.00) & 0.894 (1.00) & 0.909 (1.00) & 0.888 & 0.800 & 0.935 \\
\hspace{0.5em}Seed 2 & 0.862 (0.979) & 0.937 (1.00) & 0.964 (1.00) & 0.807 & 0.764 & 0.865 \\
\midrule
\multicolumn{7}{l}{\emph{Active if the peak amplitude $|\gamma|/\delta\ge0.05$}} \\
\hspace{0.5em}Seed 0 & 0.710 (1.00) & 0.806 (1.00) & 0.864 (1.00) & 0.940 & 1.01 & 1.08 \\
\hspace{0.5em}Seed 1 & 0.721 (1.00) & 0.786 (1.00) & 0.844 (1.00) & 0.982 & 1.01 & 1.03 \\
\hspace{0.5em}Seed 2 & 0.465 (0.562) & 0.709 (0.819) & 0.839 (0.928) & 0.896 & 0.892 & 0.912 \\
\midrule
\multicolumn{7}{l}{\emph{Active if the peak amplitude $|\gamma|/\delta\ge0.1$}} \\
\hspace{0.5em}Seed 0 & 0.357 (0.489) & 0.643 (0.942) & 0.698 (1.00) & 0.989 & 1.08 & 1.20 \\
\hspace{0.5em}Seed 1 & 0.0915 (0.312) & 0.233 (0.479) & 0.399 (0.653) & 1.24 & 1.18 & 1.16 \\
\hspace{0.5em}Seed 2 & 0.133 (0.150) & 0.343 (0.401) & 0.526 (0.601) & 0.859 & 0.916 & 0.950 \\
\midrule
\multicolumn{7}{l}{\emph{Active if the depth $\delta<0.75h$}} \\
\hspace{0.5em}Seed 0 & 0.856 (1.00) & 0.851 (1.00) & 0.782 (1.00) & 0.556 & 0.660 & 0.757 \\
\hspace{0.5em}Seed 1 & 0.873 (1.00) & 0.871 (0.999) & 0.838 (0.963) & 0.608 & 0.739 & 0.877 \\
\hspace{0.5em}Seed 2 & 0.873 (1.00) & 0.890 (0.999) & 0.883 (0.991) & 0.584 & 0.707 & 0.829 \\
\bottomrule
\end{tabular}
\end{table}
Table~\ref{tab:interp} adds the amplitude thresholds 0.02 and 0.1 to the pole statistics of Section~\ref{sec:mechanism}, for every seed and order.
The threshold 0.02 leaves 31 to 103 of the 128 atoms active in the median snapshot, so its fractions of 85.3 to 96.4 percent within $h/4$ say little about placement.
At the threshold 0.1, seeds 0 and 2 keep at most 7 atoms per median snapshot, but seed 1 keeps 52 to 55, on average at most 8.1 percent of them within one cell of a sub-cell front.
The amplitude scale of the atoms thus differs between seeds, which is consistent with the seed dependence of the amplitude criteria in Section~\ref{sec:mechanism}.

\subsection{Diagnostics}\label{app:add:diag}

\begin{table}[!tbp]
\centering
\caption{Two diagnostics of NRC, each run for seed 0 only. \textbf{(a)} Errors $\errA$ in percent on the original shifted splits at $\alpha=1.5$ when the network receives the training diffusivity $\nu=0.12$ while the operator keeps the true $\nu$. Entries are medians with the 10th and 90th percentiles, and no trajectory blew up. \textbf{(b)} Inference pruning of every atom whose peak amplitude $|\gamma|/\delta$ does not exceed the threshold, with medians of the active fraction and of the relative change of the face flux. Times are for one flux evaluation on two threads of the shared host.}
\label{tab:extra}
\footnotesize
\setlength{\tabcolsep}{4pt}
\textbf{(a)} Diffusivity channel set to the training value\par\smallskip
\begin{tabular}{llccc}
\toprule
Model & Input $\nu$ & Forcing $l\in\{7,8\}$ & $0.7\nu$ & $2\nu$ \\
\midrule
C0+ & n/a & 0.571 & 0.670 & 0.0799 \\
 & & {\scriptsize(0.409 to 0.784)} & {\scriptsize(0.551 to 0.899)} & {\scriptsize(0.0525 to 0.104)} \\
\midrule
\texttt{bb\_weno}, seed 0 & true & 0.419 & 0.652 & 0.317 \\
 & & {\scriptsize(0.377 to 0.530)} & {\scriptsize(0.517 to 0.902)} & {\scriptsize(0.286 to 0.363)} \\
\texttt{bb\_weno}, seed 0 & $0.12$ & 0.419 & 0.611 & 0.147 \\
 & & {\scriptsize(0.377 to 0.530)} & {\scriptsize(0.494 to 0.862)} & {\scriptsize(0.115 to 0.167)} \\
\midrule
NRC, seed 0 & true & 0.0555 & 1.27 & 0.214 \\
 & & {\scriptsize(0.0503 to 0.0623)} & {\scriptsize(0.733 to 2.24)} & {\scriptsize(0.203 to 0.241)} \\
NRC, seed 0 & $0.12$ & 0.0555 & 0.779 & 0.0414 \\
 & & {\scriptsize(0.0503 to 0.0623)} & {\scriptsize(0.466 to 1.24)} & {\scriptsize(0.0366 to 0.0488)} \\
\midrule
NRC, $\nu$-diverse, seed 0 & true & 0.0538 & 0.100 & 0.0149 \\
 & & {\scriptsize(0.0411 to 0.0804)} & {\scriptsize(0.0788 to 0.141)} & {\scriptsize(0.0131 to 0.0175)} \\
\bottomrule
\end{tabular}
\par\bigskip
\setlength{\tabcolsep}{3pt}
\textbf{(b)} Inference pruning\par\smallskip
\begin{tabular}{lcccccccccc}
\toprule
 & \multicolumn{3}{c}{Active fraction} & \multicolumn{3}{c}{Relative flux change} & \multicolumn{2}{c}{Time (ms)} & $\errA$ \\
\cmidrule(lr){2-4}\cmidrule(lr){5-7}\cmidrule(lr){8-9}
Threshold & $1.2$ & $1.5$ & $1.8$ & $1.2$ & $1.5$ & $1.8$ & $B=1$ & $B=24$ & (\%) \\
\midrule
None & n/a & n/a & n/a & n/a & n/a & n/a & 5.82 & 67.7 & 0.0328 \\
$10^{-2}$ & 0.590 & 0.729 & 0.905 & $3.05{\times}10^{-3}$ & $1.85{\times}10^{-3}$ & $1.37{\times}10^{-3}$ & n/a & n/a & n/a \\
$10^{-3}$ & 0.986 & 0.994 & 0.998 & $8.31{\times}10^{-5}$ & $4.36{\times}10^{-5}$ & $2.75{\times}10^{-5}$ & 6.87 & 67.3 & 0.0329 \\
$10^{-4}$ & 1.00 & 1.00 & 1.00 & $9.43{\times}10^{-7}$ & $5.94{\times}10^{-16}$ & $7.98{\times}10^{-16}$ & 6.61 & 66.8 & n/a \\
$10^{-5}$ & 1.00 & 1.00 & 1.00 & $5.44{\times}10^{-16}$ & $5.94{\times}10^{-16}$ & $7.98{\times}10^{-16}$ & n/a & n/a & n/a \\
\bottomrule
\end{tabular}
\end{table}
Table~\ref{tab:extra} holds the two diagnostics of seed 0 that Section~\ref{sec:ood} and \ref{app:atoms:pruning} quote for NRC.
In part~(a), the median $\errA$ of \texttt{bb\_weno} at $2\nu$ and $\alpha=1.5$ also falls with the training value of the diffusivity channel, from 0.317\% to 0.147\%, but stays above the 0.0799\% of C0+.
\section{Classical supplements}\label{app:classical}

This {section} states, with attribution, classical results that support Section~\ref{sec:theory}, and we claim no novelty for any of them.

\subsection{Stability on a tube and conditional convergence}\label{app:tube}

Remark~\ref{rem:stability} refers to the following threshold argument in the spirit of L\'opez-Marcos and Sanz-Serna~\cite{lopezmarcos1988stability}, whose Lipschitz hypotheses (S) and (M) we have not verified for NRC.

\begin{theorem}[Stability on a tube]\label{thm:tube}
Let $\|\cdot\|$ be a norm on $\R^N$, let $(U^n)_{0\le n\le N_T}$ be a reference sequence such as $U^n=\Ph u(t_n)$, and let the corrected scheme generate $V^{n+1}=\Sh(V^n)+\mathcal M(V^n)$ from a given $V^0$.
Assume that there are constants $\varrho>0$, $C_S\in\R$ and $C_M\ge0$ with $L:=C_S+C_M$ and $1+L\Delta t>0$ such that, for all $n<N_T$ and all $v\in\R^N$ with $\|v-U^n\|\le\varrho$,
\begin{align*}
&\text{(S)}\qquad \|\Sh(v)-\Sh(U^n)\|\le(1+C_S\Delta t)\,\|v-U^n\|,\\
&\text{(M)}\qquad \|\mathcal M(v)-\mathcal M(U^n)\|\le C_M\Delta t\,\|v-U^n\|,\\
&\text{(D)}\qquad \|U^{n+1}-\Sh(U^n)-\mathcal M(U^n)\|\le\Delta t\,\varepsilon_n,
\end{align*}
and put $\varepsilon=\max_n\varepsilon_n$.
If the error budget
\begin{equation}\label{eq:E:budget}
\mathcal B_n=(1+L\Delta t)^n\,\|V^0-U^0\|+\varepsilon\,\Delta t\sum_{k=0}^{n-1}(1+L\Delta t)^k
\end{equation}
satisfies $\max_{n\le N_T}\mathcal B_n<\varrho$, then the corrected trajectory stays in the tube and $\|V^n-U^n\|\le\mathcal B_n$ holds for every $n\le N_T$.
For $L>0$ this gives $\|V^n-U^n\|\le e^{Lt_n}\|V^0-U^0\|+\varepsilon\,(e^{Lt_n}-1)/L$, and for $L=-c_0<0$ it gives $\|V^n-U^n\|\le(1-c_0\Delta t)^n\|V^0-U^0\|+\varepsilon/c_0$ uniformly in the final time.
\end{theorem}

\begin{proof}
Assume by induction that $\|V^k-U^k\|\le\mathcal B_k<\varrho$ for all $k\le n$, which holds for $n=0$ because $\mathcal B_0=\|V^0-U^0\|$.
Then $V^n$ lies in the tube, and (S), (M) and (D) with the triangle inequality give
\begin{equation*}
\|V^{n+1}-U^{n+1}\|\le(1+L\Delta t)\,\|V^n-U^n\|+\Delta t\,\varepsilon\le(1+L\Delta t)\,\mathcal B_n+\Delta t\,\varepsilon=\mathcal B_{n+1},
\end{equation*}
where $1+L\Delta t>0$ allows $\|V^n-U^n\|$ to be replaced by its bound, and $\mathcal B_{n+1}<\varrho$ completes the induction.
The explicit forms follow from $(1+L\Delta t)^n\le e^{Lt_n}$ and the geometric sum, which is at most $1/c_0$ when $L=-c_0<0$.
\end{proof}

\begin{remark}[What the consistency error measures]\label{rem:tube}
Let $U^n=\Ph u(t_n)$ and $\mathcal M=\Delta t\,m_\theta$, and let $\mathcal T_h=\dfh/\Delta t$ be the truncation error per unit time.
Since $U^{n+1}-\Sh(U^n)=\dfh(u(t_n))$ by the definition of the one-step defect, the consistency error in (D) is $\varepsilon_n=\|\mathcal T_h(u(t_n))-m_\theta(\Ph u(t_n))\|$.
If two reference states $u$ and $w$ with $\Ph u=\Ph w$ both enter the maximum, with consistency errors $\varepsilon_u$ and $\varepsilon_w$, then
\begin{equation}\label{eq:E:fibre}
\max(\varepsilon_u,\varepsilon_w)\ge\frac{1}{2\Delta t}\,\|\Ph\Ev u-\Ph\Ev w\|,
\end{equation}
because $\Sh+\mathcal M$ takes one value $q$ at the common cell averages, with $\Delta t\,\varepsilon_u\ge\|\Ph\Ev u-q\|$ and $\Delta t\,\varepsilon_w\ge\|\Ph\Ev w-q\|$.
Non-closability on the reference set thus forces $\varepsilon$ to be at least the fiber oscillation divided by $2\Delta t$, which Theorem~\ref{thm:rea} bounds from above.
Training in $L^2$ with respect to the occupation measure, as in Proposition~\ref{prop:closable}, controls a mean of $\varepsilon_n^2$ but not the maximum along a test trajectory that Theorem~\ref{thm:tube} requires.
\end{remark}

\subsection{An interpolation inequality for the Riesz flux}\label{app:interp}

The following lemma gives the bound $B_3$ of Theorem~\ref{thm:consistency} with an explicit constant.

\begin{lemma}[Interpolation inequality]\label{lem:interp}
Let $1<\alpha<2$ and $\alpha-1<\eta\le1$, and let $u$ be bounded and H\"older continuous of order $\eta$ on $\R$, with the seminorm $[u]_\eta=\sup_{x\neq y}|u(x)-u(y)|/|x-y|^\eta$.
Then the Riesz flux of $u$ is bounded in terms of its size and its seminorm by
\begin{equation}\label{eq:E:interp}
\|\Ra u\|_{L^\infty}\le K_{\alpha,\eta}\,\|u\|_{L^\infty}^{1-\frac{\alpha-1}{\eta}}\,[u]_\eta^{\frac{\alpha-1}{\eta}},\qquad K_{\alpha,\eta}=\kappa_\alpha\Big(\frac{2}{1+\eta-\alpha}+\frac{2}{\alpha-1}\Big).
\end{equation}
Since $\Ra$ annihilates constants, $\|u\|_{L^\infty}$ may be replaced by $\inf_{c\in\R}\|u-c\|_{L^\infty}$, and the bound applies to H\"older functions on $\T$ through their periodic extension.
\end{lemma}

\begin{proof}
For $\eta>\alpha-1$ both half-line integrals converge absolutely in the kernel form of Lemma~\ref{lem:riesz},
\begin{equation*}
\Ra u(x)=\kappa_\alpha\int_{\R}\sgn(z)\,|z|^{-\alpha}\big(u(x+z)-u(x)\big)\dd z,
\end{equation*}
which we split at $|z|=\rho>0$.
The H\"older bound gives a near part of at most $2\kappa_\alpha[u]_\eta\rho^{1+\eta-\alpha}/(1+\eta-\alpha)$.
In the far part $u(x)$ integrates to zero against the odd integrable kernel, which leaves at most $2\kappa_\alpha\|u\|_{L^\infty}\rho^{1-\alpha}/(\alpha-1)$.
The choice $\rho=(\|u\|_{L^\infty}/[u]_\eta)^{1/\eta}$ gives Eq.~\eqref{eq:E:interp}, the degenerate cases are trivial, and $u-c$ has the same seminorm and Riesz flux as $u$.
\end{proof}

The splitting is of Landau-Kolmogorov type, and for the periodic odd cusps $\sgn(\sin x)|\sin x|^\eta$ with $\eta\in\{0.9,1\}$ the ratio of the two sides of Eq.~\eqref{eq:E:interp} ranges from $0.29$ to $0.87$ over the three training orders.
For the sub-grid field $w_\theta=g_\theta-\Ih\Ph g_\theta$ of NRC (seed 0) on four test states at $t=10$, the same ratio also lies between $0.29$ and $0.87$, so the bound is attained within a factor of four.

\subsection{Closed-form closures for the linear advection-dispersion equation}\label{app:ade}

For the linear fractional advection-dispersion equation $\partial_tu+\partial_x(cu-\nu\Ra u)=f$ with a constant speed $c\in\R$, all objects of Section~\ref{sec:commutator} are explicit.
In this subsection only, the cell centers are $x_j=jh$ and $N$ is odd, with $I_N=\{-\frac{N-1}{2},\dots,\frac{N-1}{2}\}$ and with $\sigma$ and $\T_N=\mathrm{span}\{e^{i\kappa x}:\kappa\in I_N\}$ as in Section~\ref{sec:decoder}.
We write $V_\kappa(v)=N^{-1}\sum_jv_je^{-i\kappa x_j}$ for the discrete Fourier coefficients and $\omega_k=ick+\nu|k|^\alpha$, so that $\Ev e^{ikx}=e^{-\Delta t\,\omega_k}e^{ikx}$ when $f=0$.
The aliasing identity behind Theorem~\ref{thm:rea} then reads
\begin{equation}\label{eq:E:alias}
V_\kappa(\Ph u)=\sum_{m\in\Z}\sigma(\kappa+mN)\,\hat u_{\kappa+mN},\qquad\kappa\in I_N,
\end{equation}
and the histopolant is $\widehat{(\Ih v)}_\kappa=V_\kappa(v)/\sigma(\kappa)$ for $\kappa\in I_N$, with all other coefficients equal to zero.

\begin{proposition}[Closed-form closures for the linear equation]\label{prop:ade}
Let $N$ be odd and let the notation be as above.
\begin{enumerate}
\item[(i)] For every $G\in L^1(\T)$ and every Fourier multiplier $\mathcal L$ the convolution with $G$ commutes with $\mathcal L$, so the top-hat filtered field $\bar u=G_h*u$ with $G_h=h^{-1}\mathbf{1}_{[-h/2,h/2]}$ satisfies $\partial_t\bar u+\partial_x(c\bar u-\nu\Ra\bar u)=\bar f$ exactly.
The cell averages are the samples $v_j=\bar u(x_j)$, and for $u_{\mathrm{sg}}=u-\Ih\Ph u$ one has $\Ph u_{\mathrm{sg}}=0$ but in general $\Ph\mathcal L u_{\mathrm{sg}}\neq0$.
\item[(ii)] If $f=0$ and the initial data lie in $\T_N$ almost surely, the coarse diagram is closable with the ideal decoder $\Ih$, and the exact coarse step maps $V_\kappa$ to $e^{-\Delta t\,\omega_\kappa}V_\kappa$ for every $\kappa\in I_N$.
\item[(iii)] Let $f=0$ and let $u$ be a real, centered, stationary Gaussian field with spectral density $q_k=\EE|\hat u_k|^2$ such that $\sum_k(1+|k|)^{2s}q_k<\infty$ for some $s>\frac12$.
Let the coarse scheme $\Sh$ be linear and translation invariant, so that $V_\kappa(\Sh v)=s_h(\kappa)V_\kappa(v)$ for all $v$ and all $\kappa\in I_N$.
Then, for every $\kappa\in I_N$ with $\sum_m\sigma(\kappa+mN)^2q_{\kappa+mN}>0$, the $L^2$-optimal closure $m^\star(V)=\EE[\dfh\mid V]$ is the circulant with the symbol
\begin{equation}\label{eq:E:circulant}
r(\kappa)=\frac{\sum_m\sigma(\kappa+mN)^2\,e^{-\Delta t\,\omega_{\kappa+mN}}\,q_{\kappa+mN}}{\sum_m\sigma(\kappa+mN)^2\,q_{\kappa+mN}}-s_h(\kappa).
\end{equation}
The irreducible error of the mode $\kappa$ is then
\begin{equation}\label{eq:E:floor}
\EE\big|V_\kappa(\dfh)-r(\kappa)V_\kappa\big|^2=\sum_m\sigma^2e^{-2\Delta t\,\Real\omega}q-\frac{\big|\sum_m\sigma^2e^{-\Delta t\,\omega}q\big|^2}{\sum_m\sigma^2q},
\end{equation}
where every $\sigma$, $\omega$ and $q$ on the right is evaluated at $\kappa+mN$.
\end{enumerate}
\end{proposition}

\begin{proof}
For item~(i), convolution with $G$ multiplies the Fourier coefficient of index $k$ by $2\pi\hat G_k$, so it commutes with every Fourier multiplier, and applying it to the equation gives the filtered equation.
The samples of $\bar u$ at the centers are the cell averages by the definition of $G_h$.
For an example with $\Ph\mathcal L u_{\mathrm{sg}}\neq0$, let $\ell$ be the symbol of $\mathcal L$, take $u=e^{i(\kappa+N)x}$ with $0\neq\kappa\in I_N$, and note that $\sigma(\kappa+N)\neq0$.
By Eq.~\eqref{eq:E:alias} the histopolant reproduces the alias as $\Ih\Ph u=\frac{\sigma(\kappa+N)}{\sigma(\kappa)}e^{i\kappa x}$, which gives
\begin{equation*}
V_\kappa(\Ph\mathcal L u_{\mathrm{sg}})=\sigma(\kappa+N)\big(\ell(\kappa+N)-\ell(\kappa)\big),
\end{equation*}
and this is nonzero whenever the symbol takes different values on the two aliasing wavenumbers.

For item~(ii), the space $\T_N$ is invariant under $\Ev$, on which $\Ph$ is injective with inverse $\Ih$ by Eq.~\eqref{eq:E:alias}.
The exact step multiplies the coefficient of $e^{i\kappa x}$ by $e^{-\Delta t\,\omega_\kappa}$, and $\Ph$ and $\Ih$ multiply and divide by the same factor $\sigma(\kappa)$, which therefore cancels.

For item~(iii), stationarity implies that the coefficients $\hat u_k$ with $k\ge1$ are independent circular complex Gaussian variables, independent of the real $\hat u_0$, with $\hat u_{-k}=\overline{\hat u_k}$.
By Eq.~\eqref{eq:E:alias}, both $V_\kappa$ and $V_\kappa(\dfh)=\sum_m\sigma(\kappa+mN)e^{-\Delta t\,\omega_{\kappa+mN}}\hat u_{\kappa+mN}-s_h(\kappa)V_\kappa$ are linear in the coefficients of the alias class $\kappa+N\Z$.
For odd $N$ the alias classes of the indices $\kappa\in I_N$ partition $\Z$, and two indices of one class never sum to zero unless $\kappa=0$.
Hence the pairs $(V_\kappa,V_\kappa(\dfh))$ with $\kappa\ge0$ are independent of each other, the pairs with $\kappa<0$ are their complex conjugates, and each pair with $\kappa\neq0$ is a proper complex Gaussian vector.
For such a vector the conditional mean is linear, $\EE[V_\kappa(\dfh)\mid V]=\EE[V_\kappa(\dfh)\overline{V_\kappa}]\,V_\kappa/\EE|V_\kappa|^2$, and the same formula holds for the real pair with $\kappa=0$.
Evaluating the two expectations with $\EE|\hat u_k|^2=q_k$ and the independence within the class gives the symbol of Eq.~\eqref{eq:E:circulant}.
The residual variance of a Gaussian regression is the Schur complement, and since $s_h(\kappa)V_\kappa$ is a function of $V$, it equals the Schur complement of $\sum_m\sigma e^{-\Delta t\,\omega}\hat u$ given $V_\kappa$, which is Eq.~\eqref{eq:E:floor}.
\end{proof}

These statements are the Wiener filter in the cell-average variables, so learning a closure is justified only for nonlinear fluxes or non-Gaussian statistics.
The floor of Eq.~\eqref{eq:E:floor} does not depend on $\Sh$, so a more accurate base scheme leaves a larger share of its smaller defect unlearnable, in agreement with the irreducible error of Proposition~\ref{prop:closable}.
A Monte Carlo regression on samples of such a field reproduces the symbol, for instance $r(4)=0.2036-0.1231i$ against $0.2033-0.1235i$, and the summed floor to a ratio of $0.995$.

\subsection{Ill-posedness of the ratio target}\label{app:ratio}

The baseline \texttt{c1\_mult} of \ref{app:training:baselines} learns the ratio of the exact next cell averages to a base step, a target that is badly conditioned for sign-changing data.

\begin{proposition}[Conditioning of closure targets]\label{prop:hadamard}
Let $b=\Sh(\Ph u)$ be the base step and $y=\Ph\Ev u$ the exact next cell averages, with $V=\Ph u$, and consider the following three targets for a learned correction.
\begin{enumerate}
\item[(i)] The additive target $y-b$ is a Lipschitz function of $y$ with constant one, for every base step.
\item[(ii)] The log-ratio target $\log(y_j/b_j)$ is bounded if and only if $0<c_1\le y_j/b_j\le c_2$ for constants $c_1$ and $c_2$, and on $\{y_j\ge y_{\min}>0\}$ it is Lipschitz in $y_j$ with constant $1/y_{\min}$.
\item[(iii)] The Hadamard ratio $y\oslash b$ is undefined on $\bigcup_j\{b_j=0\}$, and if the pair $(b_j,y_j)$ has a joint density $p\ge p_0>0$ on $[-\lambda,\lambda]\times[y_0-\lambda,y_0+\lambda]$ with $y_0>\lambda>0$, then $\EE|y_j/b_j|=\infty$.
In that case $\EE|y_j/b_j-m(V)|^2=\infty$ for every predictor with $\EE|m(V)|^2<\infty$, which includes all networks with bounded outputs.
If in addition $\Var(y_j\mid V)\ge c_3>0$ on $\{|b_j|<\lambda\}$, the loss is infinite for every Borel predictor $m$.
\end{enumerate}
\end{proposition}

\begin{proof}
Items~(i) and~(ii) follow from the definitions and from $|\log y-\log y'|\le|y-y'|/y_{\min}$ for $y,y'\ge y_{\min}$.
For item~(iii), the density bound gives $\EE|y_j/b_j|\ge p_0\int_{y_0-\lambda}^{y_0+\lambda}y\dd y\int_{-\lambda}^{\lambda}|b|^{-1}\dd b=\infty$, and Jensen's inequality gives $\EE|X-m(V)|^2\ge(\EE|X|-\EE|m(V)|)^2=\infty$ for $X=y_j/b_j$ when $\EE|m(V)|^2<\infty$.
Under the variance condition, $\EE[|X-m(V)|^2\mid V]\ge c_3/b_j^2$ on $\{|b_j|<\lambda\}$ because $b_j$ is a function of $V$, and this integrates to infinity since $b_j$ has a density of at least $2\lambda p_0$ there.
\end{proof}

In the closable case the unbounded predictor $m=y\oslash b$ has zero loss, so the obstruction concerns conditioning and learnability, not the population optimum.
It requires a base step that changes sign or approaches zero, as ours does in every snapshot, so a multiplicative coupling can suit strictly positive data.

\subsection{Rational approximation of fractional singularities}\label{app:rational}

The following statement restates the sinc construction of Stenger~\cite{stenger1986explicit}, as analyzed by Trefethen, Nakatsukasa and Weideman~\cite{trefethen2021exponential} and by Herremans, Huybrechs and Trefethen~\cite{herremans2023resolution}.

\begin{theorem}[Exponentially clustered poles for $x^\beta$]\label{thm:clustered}
Let $0<\beta<1$, $0<d<\pi$ and $n\ge1$, and define the step $\tau$ and the numbers of terms $M_1$ and $M_2$ by
\begin{equation*}
\tau=\sqrt{\frac{2\pi d}{\beta(1-\beta)n}},\qquad M_1=\Big\lceil\frac{2\pi d}{\beta\tau^2}\Big\rceil,\qquad M_2=\Big\lceil\frac{2\pi d}{(1-\beta)\tau^2}\Big\rceil.
\end{equation*}
The function
\begin{equation}\label{eq:E:stenger}
r(x)=\frac{\sin\pi\beta}{\pi}\,\tau\sum_{j=-M_1}^{M_2}e^{\beta j\tau}\,\frac{x}{e^{j\tau}+x}
\end{equation}
is rational of type $(m,m)$ with $m=M_1+M_2+1\le n+3$, and its poles $-e^{j\tau}$ cluster exponentially at the singularity $x=0$.
Whenever $\sqrt{2\pi d\beta(1-\beta)n}\ge\ln2$, the uniform error of this approximation on the unit interval satisfies
\begin{equation}\label{eq:E:stengerbound}
\sup_{0\le x\le1}\big|x^\beta-r(x)\big|\le C_{\beta,d}\,e^{-\sqrt{2\pi d\beta(1-\beta)n}},\qquad C_{\beta,d}=\frac{4}{\cos(d/2)}+\frac{\sin\pi\beta}{\pi}\Big(\frac1\beta+\frac1{1-\beta}\Big).
\end{equation}
\end{theorem}

\begin{proof}
For $x>0$ the substitution $s=e^y$ in a classical integral gives
\begin{equation*}
x^\beta=\frac{\sin\pi\beta}{\pi}\int_0^\infty s^{\beta-1}\,\frac{x}{s+x}\dd s=\frac{\sin\pi\beta}{\pi}\int_{\R}\psi_x(y)\dd y,\qquad \psi_x(y)=\frac{e^{\beta y}x}{e^y+x}.
\end{equation*}
The integrand is analytic in the strip $|\Imag y|<\pi$, and for $y=t+ib$ with $|b|\le d$ a direct expansion shows
\begin{equation*}
|e^y+x|^2-\cos^2(b/2)\,(e^t+x)^2=\sin^2(b/2)\,(e^t-x)^2\ge0,
\end{equation*}
so that $|e^y+x|\ge\cos(d/2)(e^t+x)$ holds on the closed strip of half-width $d$.
Hence $\int_{\R}|\psi_x(t+ib)|\dd t\le\pi x^\beta/(\cos(d/2)\sin\pi\beta)\le\pi/(\cos(d/2)\sin\pi\beta)$ for $0<x\le1$, and $\psi_x$ decays uniformly as $|t|\to\infty$ in the strip.
Theorem~5.1 of Trefethen and Weideman~\cite{trefethen2014exponentially} then bounds the error of the infinite trapezoidal sum with step $\tau$, after multiplication by $\sin(\pi\beta)/\pi$, by
\begin{equation*}
\frac{2}{\cos(d/2)\,\big(e^{2\pi d/\tau}-1\big)}\le\frac{4}{\cos(d/2)}\,e^{-2\pi d/\tau}\qquad\text{when }e^{2\pi d/\tau}\ge2.
\end{equation*}
The terms with $j<-M_1$ are bounded by $\frac{\sin\pi\beta}{\pi}\tau\sum_{j<-M_1}e^{\beta j\tau}\le\frac{\sin\pi\beta}{\pi\beta}e^{-\beta M_1\tau}$, using $x/(e^{j\tau}+x)\le1$ and $\tau/(e^{a\tau}-1)\le1/a$ for $a>0$.
The terms with $j>M_2$ are bounded in the same way by $\frac{\sin\pi\beta}{\pi(1-\beta)}e^{-(1-\beta)M_2\tau}$, using $e^{\beta j\tau}x/(e^{j\tau}+x)\le e^{-(1-\beta)j\tau}$ for $x\le1$.
By the choice of $M_1$ and $M_2$ both exponents are at least $2\pi d/\tau=\sqrt{2\pi d\beta(1-\beta)n}$, which yields Eq.~\eqref{eq:E:stengerbound}, and at $x=0$ both sides vanish.
Finally $M_1+M_2\le2\pi d/(\beta(1-\beta)\tau^2)+2=n+2$, and each term of Eq.~\eqref{eq:E:stenger} is of type $(1,1)$, which gives the type of $r$.
\end{proof}

As $d\uparrow\pi$ the exponent tends to $\pi\sqrt{2\beta(1-\beta)n}$, below the rate $2\pi\sqrt{\beta n}$ of best rational approximation established by Stahl~\cite{stahl2003best} after Newman's result for $|x|$~\cite{newman1964rational}.
The constant with $\cos(d/2)$ only slightly refines the one obtained by the same argument.
Over a grid of 90 parameter combinations the observed error never exceeds $0.167$ times the bound.

\begin{corollary}[Periodic cusps]\label{cor:cusp}
Let $0<\beta<2$, $x_0\in\T$, $0<d<\pi$ and $n\ge1$, and let $\tau$, $M_1$ and $M_2$ be as in Theorem~\ref{thm:clustered} with $\beta/2$ in place of $\beta$.
For $-M_1\le j\le M_2$ we put $c_j=e^{j\tau}$, $\zeta_j=(1+2c_j+2\sqrt{c_j+c_j^2})^{-1}$ and $a_j=\zeta_je^{ix_0}$, which defines one pole for every term of Eq.~\eqref{eq:E:stenger}.
Then there are real coefficients $c_0$ and $\gamma_j$ such that, whenever $\sqrt{2\pi d(\beta/2)(1-\beta/2)n}\ge\ln2$,
\begin{equation}\label{eq:E:cusp}
\sup_{x\in\T}\Big|\,\Big|\sin\frac{x-x_0}{2}\Big|^\beta-c_0-2\Real\sum_{j=-M_1}^{M_2}\frac{\gamma_j}{1-\bar a_je^{ix}}\Big|\le C_{\beta/2,d}\,e^{-\sqrt{2\pi d(\beta/2)(1-\beta/2)n}}.
\end{equation}
All poles lie at the position $x_0$, and their depths satisfy $1-\zeta_j=2\sqrt{c_j}\,(1+O(\sqrt{c_j}))$ as $c_j\to0$, so they cluster exponentially toward the cusp.
\end{corollary}

\begin{proof}
With $\varsigma=\sin^2\frac{x-x_0}{2}\in[0,1]$ and $y=x-x_0$, Theorem~\ref{thm:clustered} approximates $\varsigma^{\beta/2}$ by terms $w_j-w_jc_j/(c_j+\varsigma)$ with positive weights $w_j$.
Since $\zeta_j=1+2c_j-2\sqrt{c_j+c_j^2}$ solves $\zeta^2-2(1+2c_j)\zeta+1=0$, the Poisson kernel identity gives $\frac{c_j}{c_j+\varsigma}=\frac{4c_j\zeta_j}{1-\zeta_j^2}\big(2\Real\frac{1}{1-\zeta_je^{iy}}-1\big)$ with $\zeta_je^{iy}=\bar a_je^{ix}$, which yields Eq.~\eqref{eq:E:cusp} and the stated depths.
\end{proof}

Each term of Eq.~\eqref{eq:E:cusp} is a constant plus one atom $\psi_{a_j}$ with the closed-form Riesz flux of Lemma~\ref{lem:atoms}.
For $\beta=0.75$, $\alpha=1.5$ and 64 faces, 130 clustered poles give uniform and interface-flux errors of $1.1\times10^{-7}$ and $1.4\times10^{-7}$, against $8.4\times10^{-4}$ and $2.0\times10^{-4}$ for 8193 Fourier modes.
Clustering is necessary by the following classical result, which Section~\ref{sec:whypoles} observes on the benchmark data.

\begin{proposition}[Fixed poles give algebraic rates]\label{prop:fixedpoles}
Let $\varrho>0$, and let $\mathcal R^\varrho_n$ be the real rational functions of type $(n,n)$ whose poles lie at distance at least $\varrho$ from $[-1,1]$.
If a continuous $f$ satisfies $\inf_{r\in\mathcal R^\varrho_n}\|f-r\|_{L^\infty(-1,1)}\le Cn^{-s}$ for all $n\ge1$ with some $0<s<1$, then $f$ is H\"older continuous of every order $s'<s$ on every compact subinterval of $(-1,1)$.
In particular, the cusp $|x-x_0|^\eta$ with $x_0\in(-1,1)$ and $0<\eta<1$ cannot be approximated in $\mathcal R^\varrho_n$ at any rate $n^{-s}$ with $s>\eta$.
\end{proposition}

\begin{proof}[Proof sketch]
The Bernstein-type inequality of Borwein, Erd\'elyi and Zhang~\cite{borwein1994chebyshev} gives $|r'(x)|\le C_\varrho n\|r\|_{L^\infty(-1,1)}/\sqrt{1-x^2}$ on $\mathcal R^\varrho_n$, and the classical inverse argument applies it to near-best approximants from $\mathcal R^\varrho_{2^k}$.
Their differences are of size $O(2^{-ks})$, and splitting the telescoping sum at $2^{-K}\approx|x-y|$ gives the H\"older bound, which a cusp of order $\eta$ violates for every order above $\eta$.
\end{proof}

\begin{remark}[Relevance on the torus]\label{rem:analytic}
For the unforced periodic problem with $1<\alpha<2$, Kiselev, Nazarov and Shterenberg~\cite{kiselev2008blowup} state that solutions are real analytic in space for $t>0$, and for the forced benchmark analyticity is plausible but not proved.
Such solutions have no real cusps, so the results above act only through complex singularities near the real axis, where the adaptive pole depth of Propositions~\ref{prop:gap} and~\ref{prop:poles} plays the role of $c_j$.
They therefore motivate adaptive poles, while the rate at which the learned poles approximate the solution remains an open question.
\end{remark}

\end{document}